%% file: main.tex
\documentclass{amsart}
\usepackage[letterpaper,hmargin=0.9in,vmargin=0.8in]{geometry}

\usepackage{amsmath,amscd}
\usepackage{amsthm}
\usepackage{amssymb}
\usepackage[foot]{amsaddr}
\usepackage{mathtools}
\usepackage{stmaryrd}
\usepackage{hyperref}
\usepackage{url}
\usepackage{tikz-cd}
\usepackage{tikz}
\usepackage{tkz-euclide}
\usepackage{xcolor}
\usepackage{caption}
\usepackage{subcaption}
\usepackage[capitalize,noabbrev]{cleveref}
\usepackage{graphicx}
\usepackage{comment}
\usepackage{appendix}
\usepackage{enumerate}
\usepackage{bbm}
\usepackage{algorithm}
\usepackage{enumitem}
\usepackage{cancel}
\usepackage{algpseudocode}
\usepackage{amsthm}
\usepackage{thmtools}
\usepackage{soul}
\usepackage{bm}
\usetikzlibrary{positioning}
\usepackage[scr=boondox,  cal=cm]{mathalpha}
\usepackage[normalem]{ulem}
\usepackage{float}
\usepackage[inkscapelatex=false]{svg}

\usepackage{pifont}

\theoremstyle{plain}
\newtheorem*{theorem*}{Theorem}
\newtheorem{theorem}{Theorem}[section]
\newtheorem{proposition}[theorem]{Proposition}
\newtheorem{lemma}[theorem]{Lemma}
\newtheorem{corollary}[theorem]{Corollary}

\theoremstyle{definition}
\newtheorem*{definition*}{Definition}
\newtheorem{definition}[theorem]{Definition}
\newtheorem{example}[theorem]{Example}
\newtheorem{remark}[theorem]{Remark}

\crefname{lemma}{Lemma}{Lemmas}

\input{header}

\title[Estimating the persistent homology of $\rn$-valued functions]{Estimating the persistent homology of $\rn$-valued functions\\using function-geometric multifiltrations}

\author{Ethan André${}^1$, Jingyi Li${}^2$, David Loiseaux${}^3$, Steve Oudot${}^4$}
\address{${}^1$École Normale Supérieure, Paris, France. \url{eandre02@clipper.ens.psl.eu}}
\address{${}^2$École Polytechnique and Inria Saclay, Palaiseau, France. \url{jingyi.li@polytechnique.edu}}
\address{${}^3$Inria Saclay, Palaiseau, France. \url{david.loiseaux@inria.fr}}
\address{${}^4$École Polytechnique and Inria Saclay, Palaiseau, France. \url{steve.oudot@inria.fr}}

\begin{document}

\begin{abstract}
  Given an unknown $\rn$-valued function $\ff$ on a metric space~$X$, can we approximate the persistent homology of~$\ff$ from a finite sampling of~$X$ with known pairwise distances and function values? This question has been answered in the case $n=1$, assuming $\ff$ is Lipschitz continuous and $X$ is a sufficiently regular geodesic metric space, and using filtered geometric complexes with fixed scale parameter for the approximation.  In this paper we answer the question for arbitrary $n$, under similar assumptions and using function-geometric multifiltrations. Our analysis offers a different view on these multifiltrations by focusing on their approximation properties rather than on their stability properties. We also leverage the multiparameter setting to provide insight into the influence of the scale parameter, whose choice is central to this type of approach. From a practical standpoint, we show that our approximation results are robust to  input noise,  and that function-geometric multifiltrations have good statistical convergence properties. We also provide an algorithm to compute our estimators, and we use its implementation  to conduct extensive experiments, on both synthetic and real biological data, in order to validate our theoretical results.
\end{abstract}

\maketitle

\bibliographystyle{plain}

\input{Sections/intro.tex}

\input{Sections/backgd.tex}
\input{Sections/background/stats.tex}

\input{Sections/main_results.tex}

\input{Sections/proof_fix.tex}

\input{Sections/proof_vary.tex}

\input{Sections/stats.tex}

\input{Sections/comp.tex}

\input{Sections/expe.tex}



\paragraph*{\textbf{Acknowledgements.}} D.\,L. was supported by Inria Action Exploratoire PREMEDIT (Precision Medicine using Topology).

\bibliography{ref}

\appendix

\input{Sections/appendix.tex}

\input{Sections/main_proof_contractible.tex}

\end{document}

%% file: header.tex
\DeclareMathOperator*{\im}{Im}

\DeclareMathOperator*{\supp}{supp}
\newcommand{\lan}{\mathrm{Lan}}

\DeclareMathOperator*{\rk}{rk}

\newcommand{\xtwoheadrightarrow}[1]{\overset{#1}{\twoheadrightarrow}}

\newcommand{\e}{\varepsilon}
\newcommand{\RRR}{\mathbb{R}}
\newcommand{\NNN}{\mathbb{N}}

\newcommand{\kkk}{\mathbbm{k}}
\newcommand{\BB}{\mathcal{B}}
\newcommand{\CC}{\mathcal{C}}

\newcommand{\FF}{\mathcal{F}}

\newcommand{\OO}{\mathcal{O}}

\newcommand{\RR}{\mathcal{R}}

\newcommand{\PP}{\mathcal{P}}

\newcommand{\ff}{\mathscr{f}}

\newcommand{\bnu}{\bm{\nu}}
\newcommand{\rn}{\RRR^n}

\newcommand{\rp}{\RRR_{\geq 0}}
\newcommand{\rzp}{\RRR_{\geq 0}}

\newcommand{\vect}{\mathrm{\bf{vec}}}
\newcommand{\Vect}{\mathrm{\bf{Vec}}}

\newcommand{\ms}{\hs(\RR^{\bullet\to 2\bullet}(\FP))}

\newcommand{\msd}{\hs(\RR^{\delta\to 2\delta}(\FP))}
\newcommand{\msdt}{\hs(\RR^{\delta\to 2\delta}(\tFP))}
\newcommand{\msdk}{\hs(\RR^{\delta_k\to 2\delta_k}(\FP))}
\newcommand{\mshdk}{\hs(\RR^{\hat\delta_k\to 2\hat\delta_k}(\FP))}
\newcommand{\mszero}{\hzero(\RR^{\bullet\to 2\bullet}(\FP))}
\newcommand{\msone}{\hone(\RR^{\bullet\to 2\bullet}(\FP))}

\newcommand{\Rinc}{\RR^{\bullet\to 2\bullet}}

\newcommand{\oneo}{\bm{1}_0}
\newcommand{\thresh}{D}

\newcommand{\bmx}{{\bm{x}}}
\newcommand{\bmy}{\bm{y}}
\newcommand{\bmz}{\bm{z}}

\newcommand{\bmt}{\bm{t}}

\newcommand{\bmv}{\bm{v}}

\newcommand{\topc}{\mathrm{\bf{Top}}}

\newcommand{\id}{\mathrm{Id}}

\newcommand{\rhox}{\varrho_{\scriptscriptstyle X}}

\newcommand{\gr}{\mathrm{gr}}

\newcommand{\Homology}{H}
\newcommand{\hs}{\Homology_*}
\newcommand{\tff}{\tilde{\ff}}
\newcommand{\ffp}{\restr {\ff} P}
\newcommand{\tffp}{{\tilde{\ff}}|_P}
\newcommand{\hsf}{\ensuremath{\hs( \ff)}}

\newcommand{\hzero}{\Homology_0}
\newcommand{\hone}{\Homology_1}
\newcommand{\hi}{\Homology_i}

\newcommand{\honef}{\ensuremath{\hone \left( \ff\right)}}

\newcommand\restr[2]{{  \left.\kern-\nulldelimiterspace #1 \vphantom{\big|} \right|_{#2}  }}

\newcommand{\macrohsTFP}[1]{\ensuremath{\hs \left( #1(\tffp)\right)}}
\newcommand{\hsobtffp}{\macrohsTFP{\OO^{\bullet}}}
\newcommand{\hscbtffp}{\macrohsTFP{\CC^{\bullet}}}
\newcommand{\hsrbtffp}{\macrohsTFP{\RR^{\bullet\to 2\bullet}}}

\newcommand{\FP}{\ff|_P}

\newcommand{\OdFP}{\OO^{\delta}(\FP)}
\newcommand{\OeFP}{\OO^{\e}(\FP)}
\newcommand{\OddFP}{\OO^{2\delta}(\FP)}
\newcommand{\RdFP}{\RR^{\delta}(\FP)}

\newcommand{\CdFP}{\CC^{\delta}(\FP)}

\newcommand{\rprn}{\rp\times\rn}

\newcommand{\di}{d_{\mathrm{i}}}

\newcommand{\db}{d_{\mathrm{b}}}

\newcommand{\lanf}{\lan_{\iota}\hsf}

\newcommand{\lanfone}{\lan_{\iota}\honef}
\newcommand{\lanftop}{\lan_{\iota}\FF}

\newcommand{\dBo}{\db^{\bm{1}}}
\newcommand{\dBv}{\db^{\oneo}}
\newcommand{\dIo}{\di^{\bm{1}}}
\newcommand{\dIv}{\di^{\oneo}}
\newcommand{\dB}{\db^{\bmv}}
\newcommand{\dI}{\di^{\bmv}}
\newcommand{\dH}{d_{\textnormal{H}}}

\newcommand{\fp}{fp}

\newcommand{\even}{2\NNN}
\newcommand{\odd}{2\NNN+1}

\usepackage{dsfont}
\newcommand{\ind}[1]{\mathds{1}_{#1}}
\newcommand{\setprobmeas}[1]{\ensuremath{\mathcal{P}\left({#1}\right)}}
\newcommand{\abstd}[1]{\ensuremath{\mathcal{P}_{a,b}\left(#1\right)}}
\newcommand{\oracle}[0]{{\hs(\ff)}}
\newcommand{\oracleHi}[0]{{\Homology_i ( \ff)}}
\newcommand{\oracleHzero}[0]{{\Homology_0 \left( \ff\right)}}

\newcommand{\estoffsetd}[0]{\ensuremath{{\hs \left( \mathcal O^{\delta_k} \left( \restr {\ff}{X_k}\right)\right)}}}
\newcommand{\estoffsetdh}[0]{\ensuremath{{\hs \left( \mathcal O^{\hat\delta_k} \left( \restr {\ff}{X_k}\right)\right)}}}
\newcommand{\estcechd}[0]{\ensuremath{{\hs \left( {\mathcal{{C}}}^{\delta_k} \left( \restr {\ff}{X_k}\right)\right)}}}
\newcommand{\estcechdh}[0]{\ensuremath{{\hs \left({\mathcal{{C}}}^{\hat\delta_k} \left( \restr {\ff}{X_k}\right)\right)}}}
\newcommand{\estripsd}[0]{\ensuremath{{\hs \left( \mathcal R^{\delta_k\to 2\delta_k} \left( \restr {\ff}{X_k}\right)\right)}}}
\newcommand{\estripsdh}[0]{\ensuremath{{\hs \left( \mathcal R^{\hat\delta_k\to 2\hat\delta_k} \left( \restr {\ff}{X_k}\right)\right)}}}
\newcommand{\estoffsetdHi}[0]{\ensuremath{{\Homology_i \left( \mathcal O^{\delta_k} \left( \restr {\ff}{X_k}\right)\right)}}}
\newcommand{\estcechdHi}[0]{\ensuremath{{\Homology_i \left( {\mathcal{{C}}}^{\delta_k} \left( \restr {\ff}{X_k}\right)\right)}}}
\newcommand{\estripsdHi}[0]{\ensuremath{{\Homology_i \left( \mathcal R^{\delta_k\to 2\delta_k} \left( \restr {\ff}{X_k}\right)\right)}}}

\newcommand{\estgeneric}[0]{\ensuremath{\widehat{\hsf}_k}}
\newcommand{\estgenericHi}[0]{\ensuremath{\widehat{{H_i}(\ff)_k}}}
\newcommand{\estgenericHzero}[0]{\ensuremath{\widehat{H_0(\ff)}_k}}

\newcommand{\dd}{\,\mathrm{d}}

\newcommand{\setmeasurableset}[1]{\ensuremath{\mathcal{#1}}}
\newcommand{\Pbb}[0]{\ensuremath{\mathbb P}}
\newcommand{\statcompl}[1]{\ensuremath{#1^{\textnormal{c}}}}
\newcommand{\Ebb}[0]{\ensuremath{\mathbb E}}
\newcommand{\measspacetuple}[1]{\ensuremath{\left( #1, \setmeasurableset{#1}\right)}}

\newcommand{\diam}{\mathrm{diam}}

\newlist{assumptions}{enumerate}{1}
\setlist[assumptions,1]{label=\textnormal{(A\arabic*)}, ref=\textnormal{A\arabic*}}
\crefname{assumptionsi}{Assumption}{Assumptions}
\Crefname{assumptionsi}{Assumption}{Assumptions}

\usepackage{todonotes}

\definecolor{calpolypomonagreen}{rgb}{0.12, 0.3, 0.17}

\newcommand{\tFP}{\tff |_ P}

\newcommand{\MMA}{\texttt{MMA}}

\DeclareMathOperator*{\argmax}{argmax}
\DeclareMathOperator*{\argmin}{argmin}

%% file: Sections/intro.tex
\section{Introduction}\label{sec:introduction}
\subsection*{Context}

Let $X$ be a metric space and $\ff\colon X\to\rn$ a function, both unknown. Suppose we are given a finite point cloud $P$ sampled from~$X$, such that the pairwise distances between the points of~$P$ are known, as well as the function values at the points of~$P$. Given this input, can we approximate the persistent homology $\hs(\ff)$ of~$\ff$ (i.e., the persistence module induced in homology from the multifiltration of $X$ by the sublevel sets of~$\ff$) with provable guarantees?
This question was addressed in~\cite{chazal2011scalar} in the case~$n=1$. The authors proposed an estimator based on a nested pair of Rips complexes $\RR^{\delta}(P)\subseteq \RR^{2\delta}(P)$, where $\delta$ is a user-defined parameter, which they filtered by the values of~$\ff$ at the vertices using lower-star filtrations. The estimator was then the image of the morphism of $1$-parameter persistence modules induced in homology by the inclusion $\RR^{\delta}(P)\hookrightarrow\RR^{2\delta}(P)$. Assuming both the domain~$X$ and the function~$\ff$ are sufficiently regular (typically, $X$ is a compact geodesic metric space with positive convexity radius~$\rhox$, and $\ff$ is $c$-Lipschitz continuous), and $P$ is an $\e$-sample of~$X$ in the geodesic distance for some small enough value of~$\e$ (specifically, $\e<\rhox/4$), they proved that the estimator is $2c\delta$-interleaved with~$\hs(\ff)$ for any choice of parameter $\delta$ within the range $[2\e, \rhox/2)$. Thus, in cases where $\e$ is known or can be estimated reliably, one can get an $O(2c\e)$-interleaving with the  target $\hs(\ff)$, hence an $O(2c\e)$-matching with its barcode, by the algebraic stability theorem~\cite{bauer2015induced,chazal2009proximity,chazal2016structure}. Under a reasonable sampling model, $\e$ goes to zero as the number of sample points goes to infinity, hence so does the approximation error, which means that the proposed estimator is consistent.
This approach has since been applied in various contexts, notably in clustering where the algorithm ToMATo~\cite{chazal2013persistence} is an instance of the method in homology degree~$0$. 
However, two fundamental questions were left open in~\cite{chazal2011scalar,chazal2013persistence}:
\begin{enumerate}
\item[Q1] how to estimate $\e$, or to choose parameter~$\delta$ directly, when $\e$ is unknown;
\item[Q2] how to generalize the approach to arbitrary~$n$, i.e., to vector-valued functions~$\ff$.
\end{enumerate}
Q2 is precisely the question we asked at the beginning. Q1 is related to the general problem of scale parameter estimation, which has received a lot of attention in statistics, including in the context of topological data analysis~(TDA). Particularly relevant to our problem in the case~$n=1$ are~\cite{bobrowski2017topological,rolle2024stable},
where deviation bounds are leveraged to provide asymptotic rules on the choice of the neighborhood size versus the sample size and function level, for estimating~$\hzero(\ff)$ when $\ff$ is a density function, and for the related problem of estimating the homology of a fixed level set of a density or regression function. 
This kind of approach departs from the general philosophy underpinning persistent homology, which is to avoid choosing the scale parameter a priori by  considering the persistent information across all scales. 
More in line with this philosophy is the work on {\em Persistable}~\cite{scoccola2023persistable}, which provides users with a tool \`a la RIVET~\cite{lesnick-wright} to investigate the structure of a bifiltration parametrized by scale and density level, in order to make an informed choice for the scale  (possibly as a function of the density level). This work is close in spirit to ours, although tied to the case~$n=1$ with a specific choice of function.

\begin{figure}[t]
    \centering
    \includegraphics[scale=0.8]{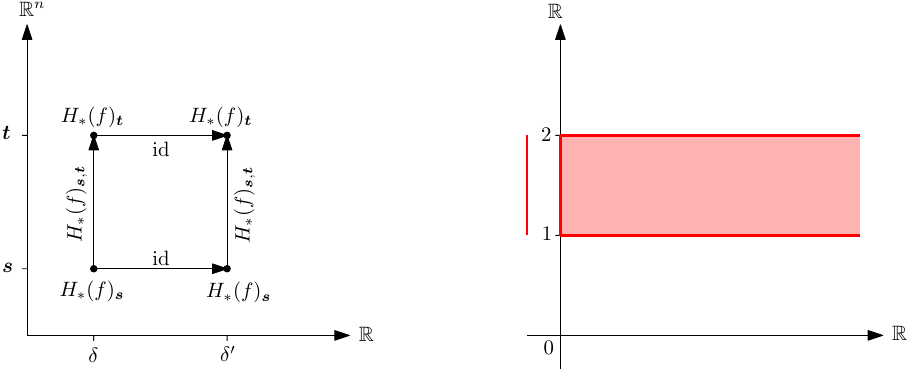}
    \caption{\textbf{Left}: an illustration of the left Kan extension $\lanf:\rprn\to\vect$, with identity maps horizontally and the structural morphisms of $\oracle$ vertically. \textbf{Right}: (case $n=1$) the left Kan extension of the interval module $\kkk^{[1,2]}$ is the interval module $\kkk^{\rp\times[1,2]}$ (in red). }
    \label{fig:example_kan_extension_R2}
\end{figure}

\begin{figure}[t]
	\centering
	\includegraphics[scale=0.25]{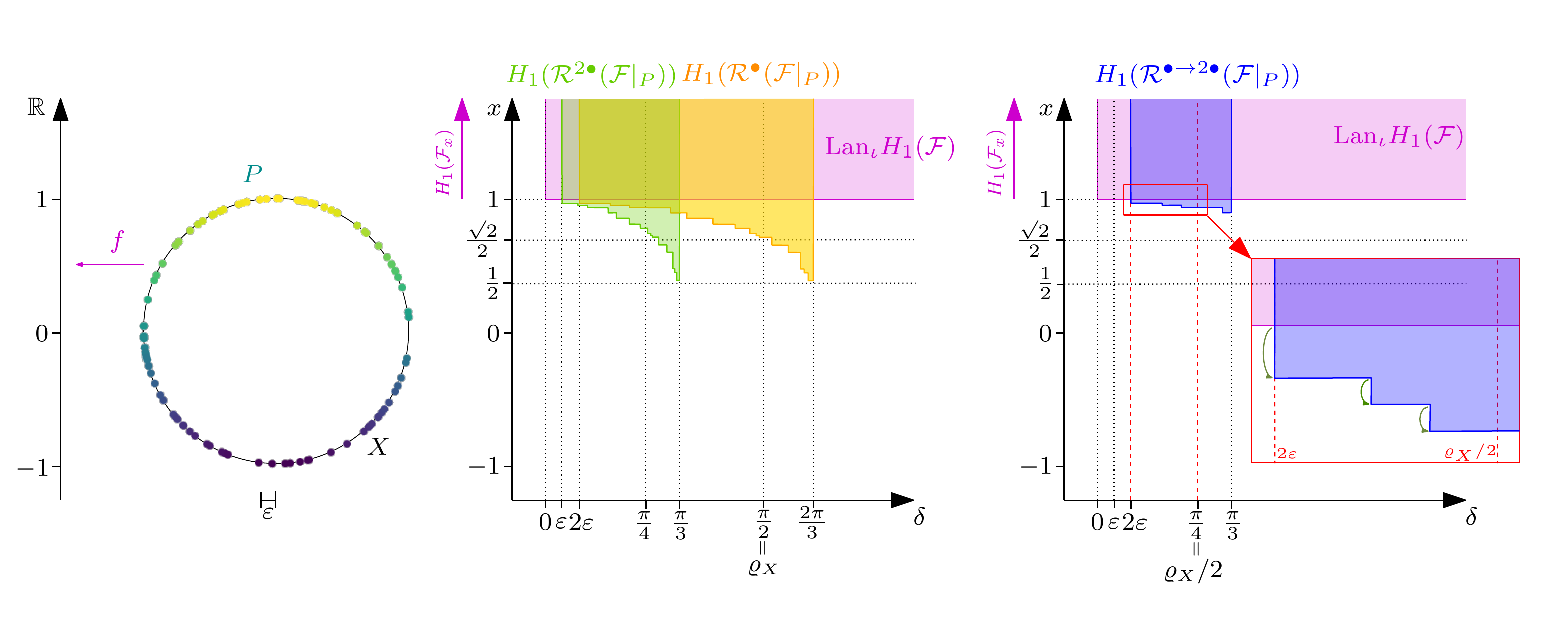}
	\caption{
         Left: the vertical height function on a sampled unit circle in the plane. Distances on the circle are given  by arclength. Center: $\hone(\ff)$ has a single interval summand (in magenta), starting at height~$1$, which extends to the free module $\lan_\iota \hone(\ff)$ generated at~$(0,1)$ in~$\rp\times\RRR$. The modules $\hone(\RR^{\bullet}(\FP))$ (in yellow) and $\hone(\RR^{2\bullet}(\FP))$ (in green) are interval modules in this simple scenario. Right: the estimator $\hone(\RR^{\bullet\to 2\bullet}(\FP))$ (in blue), which is also an interval module, approximates the target~$\lan_\iota \hone(\ff)$ in the vertical interleaving distance within any slab $[2\e, \delta_0]\times \RRR$ with $\delta_0<\rhox/2$, as per \cref{thm:estimator_varying_radius}. In turn, the vertical interleaving between the two modules implies a vertical matching between their multigraded Betti numbers within the slab (illustrated by green arrows in the close-up view), as per \cref{cor:stab_inv_rprn}.
         }
	\label{fig:example_M_from_rips}
\end{figure}

\subsection*{Contributions}
In this paper we address both Q1 and~Q2 using a unified persistence-based approach. 

For Q2 we study the extension of the estimator of~\cite{chazal2011scalar} to
vector-valued functions~$\ff\colon X\to\rn$ and generalize its approximation
guarantee to this setting. Specifically, we consider the nested pair of Rips
complexes $\RR^{\delta}(P)\subseteq \RR^{2\delta}(P)$, for a fixed~$\delta$,
which we filter by the values of~$\ff$ at the vertices using lower-star
filtrations (where the codomain~$\rn$ is equipped with the product order);
then, our estimator is the image of the morphism of $n$-parameter persistence
modules induced in homology by the inclusion
$\RR^{\delta}(P)\hookrightarrow\RR^{2\delta}(P)$ between filtered complexes. We
denote by $\msd$ this estimator. For the sake of generality we assume
$\ff$ to be $\omega$-continuous for some modulus of continuity $\omega$, that
is: $\lVert \ff(x)-\ff(y) \rVert_\infty \leq \omega(d_X(x,y))$ for any $x,y\in X$,
where $\omega\colon \rp\to\rp$ is a non-decreasing  subadditive
function such that $\omega(\delta)$ goes to~$0$ as $\delta$ does.
  This includes Lipschitz or H\"older continuous functions as special cases.
  Under this assumption, our first main result
  (Theorem~\ref{thm:estimator_fixed_radius}) states that the estimator~$\msd$
  is $\omega(2\delta)$-interleaved with its target~$\hs(\ff)$ for any choice of
  $\delta$ within the range~$[2\e, \rhox/2)$, under the same regularity
  condition on $X$ and $\e$-sampling condition on~$P$ as before. This
  generalizes the main result of~\cite{chazal2011scalar} to arbitrary~$n$. Our
  proof takes inspiration from the one in~\cite{chazal2011scalar} but uses a
  novel, purely diagrammatic formulation that clarifies it and makes it hold for any~$n\geq 1$.
  In the course of the proof we study two other related estimators:
  $\OO^{\delta}(\FP)$, which is based on the filtration of the $\delta$-offset
  of~$P$ by the sublevel sets of~$\ff$; and $\CC^{\delta}(\FP)$, which is based
  on the filtration of the $\delta$-\v{C}ech complex of~$P$ by the sublevel
  sets~of~$\ff$. 
  
\begin{figure}[tbp!]
  \centering
  \begin{subfigure}[b]{0.44\textwidth}
    \centering
    \includegraphics[width=\linewidth]{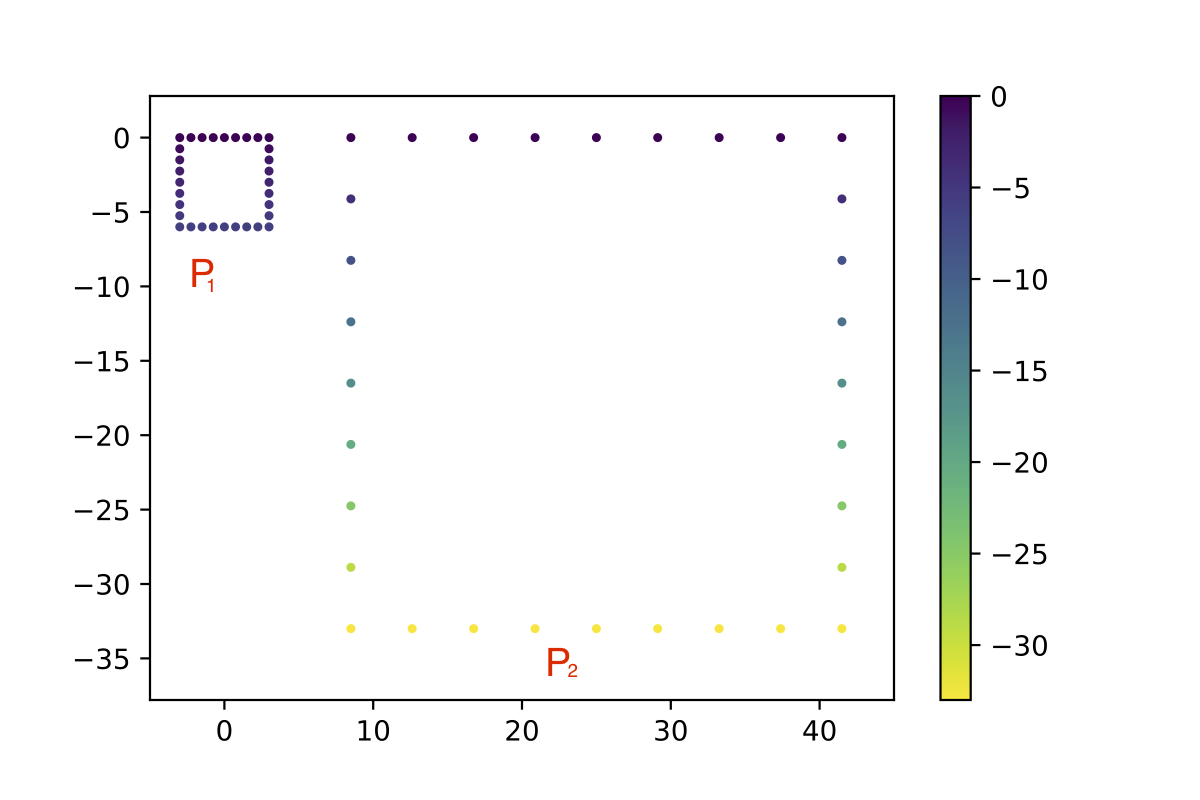}
    \caption{}\label{fig:intro_merge_int:a}
  \end{subfigure}
  \begin{subfigure}[b]{0.44\textwidth}
    \centering
    \includegraphics[width=\linewidth]{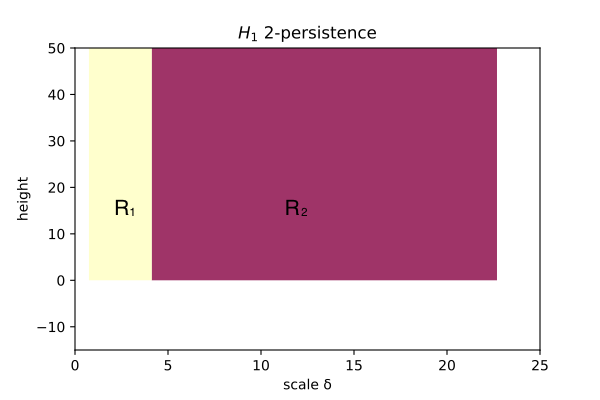}
    \caption{}\label{fig:intro_merge_int:c}
  \end{subfigure}
  \vspace{0.5cm}
  \begin{subfigure}[b]{0.44\textwidth}
    \centering
    \includegraphics[width=\linewidth]{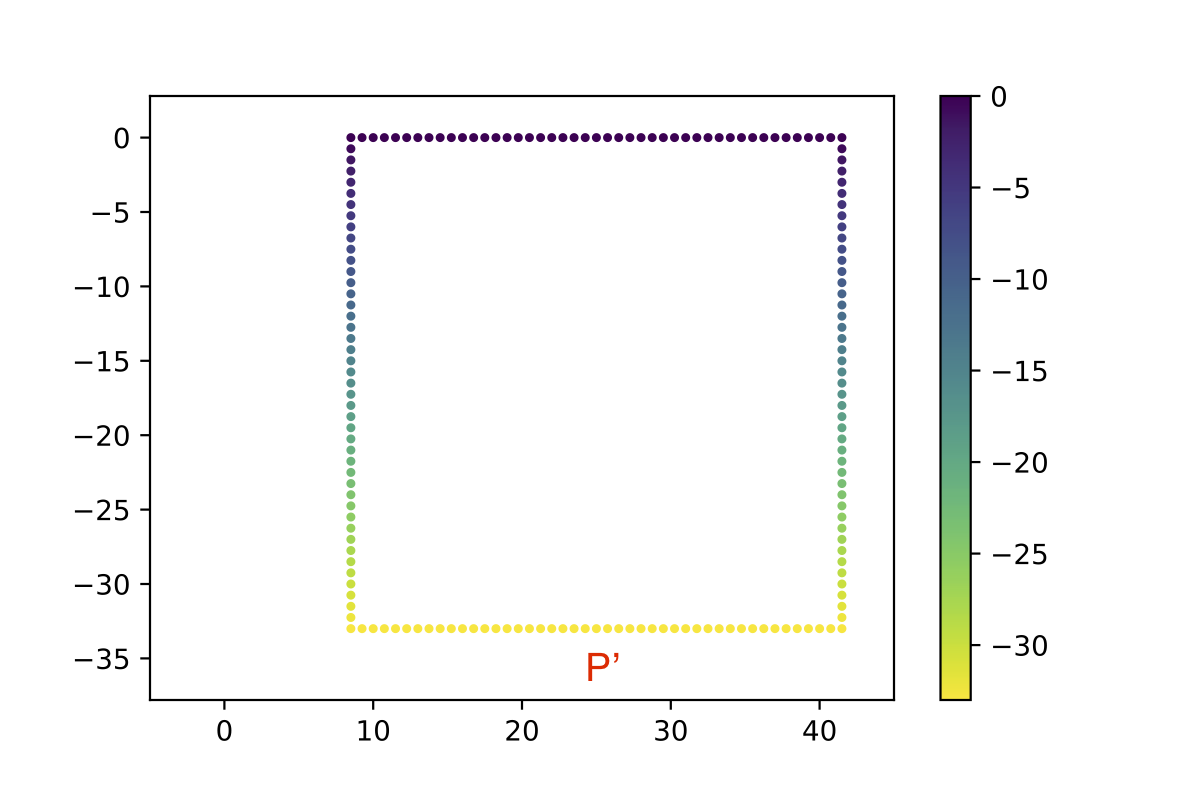}
    \caption{}\label{fig:intro_merge_int:b}
  \end{subfigure}
  \begin{subfigure}[b]{0.44\textwidth}
    \centering
    \includegraphics[width=\linewidth]{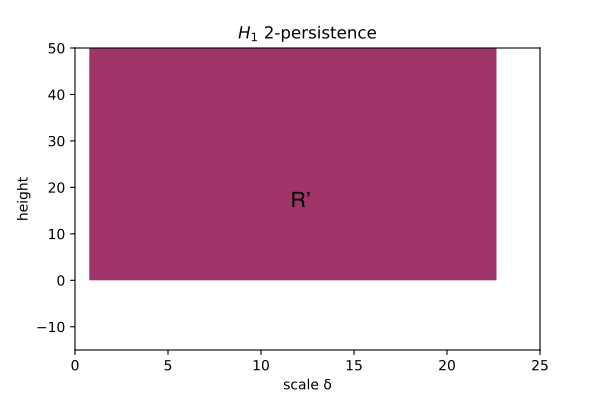}
    \caption{}\label{fig:intro_merge_int:d}
  \end{subfigure}
  \caption{Contrasting Theorem~\ref{thm:estimator_varying_radius} with Theorem~\ref{thm:estimator_fixed_radius}. \textbf{(a)}: the input is $P=P_1\sqcup P_2$,
  where $P_1$ and $P_2$ are two point clouds regularly sampled from two disjoint  squares $X_1, X_2$ in the plane. Distances within each square are shortest-path distances along the boundary, while distances between squares are infinite. Here $\ff$ is the vertical height function, and its persistent homology in degree~$1$ is considered.
  \textbf{(b)}: the estimator $\msone$ is an interval-decomposable module whose summands have half-open rectangle supports, respectively $R_1$ (in yellow) and $R_2$ (in red). The scalings of $X_1, X_2$, and their respective sampling densities, have been adjusted so that $R_1\cap R_2=\emptyset$ while $R_1\cup R_2$ is a single rectangle~$R'$ shown in subfigure~\textbf{(d)}. In turn, the interval module with support~$R'$ can be realized as the degree~$1$ persistent homology of another sample~$P'$ from a square in the plane, shown in subfigure~\textbf{(c)}.
\cref{thm:estimator_varying_radius} guarantees that $\msone$ is interleaved with $\lan_\iota\hone(\ff|_{X_1})$ over $R_1$ and with $\lan_\iota\hone(\ff|_{X_2})$ over $R_2$, leading to the two summands in~\textbf{(b)}. By contrast,  \cref{thm:estimator_fixed_radius} only guarantees interleavings within the vertical slices, which is not sufficient to discriminate the module with two summands in~\textbf{(b)} from the module with a single summand in~\textbf{(d)}. 
  }
  \label{fig:vertical_interleaving_wins_exact}
\end{figure}

For Q1 we propose two complementary approaches. The first one infers a relevant choice of~$\delta$ under an appropriate statistical model. This model assumes the points~$P$ are sampled i.i.d. according to some unkown probability distribution~$\mu$ that is supposed to be $(a,b)$-standard (\cref{def:ab_std_meas}). When parameters $a,b$ are known, we show that there is an explicit optimal choice~$\delta_k$ for parameter~$\delta$, as a function of the number~$k$ of sample points, that makes the sequence of approximations $ \left( {\msdk} \right)_{k\in \mathbb{N}}$ a {\em quasi-minimax} estimator of~$\hs(\ff)$. When $a,b$ are unknown, as would generally be the case in practice, we show that $\delta_k$ can be effectively estimated under our statistical model, so that the sequence $ \left( {\mshdk} \right)_{k\in \mathbb{N}}$ is again a quasi-minimax estimator of~$\hs(\ff)$. These statements are gathered in our second main result (\cref{thm:stat_results_overview}).

Our second approach to addressing~Q1 is meant for cases where our statistical model (or any statistical alternative) does not hold. It follows the general philosophy underpinning persistent homology by letting~$\delta$ vary over~$\rp$ and treating it as an extra parameter for persistence. This means that our estimator now becomes the $(n+1)$-parameter persistence module $\ms$ indexed over~$\rprn$. In order to enable comparisons, the target $\oracle$ must also be extended to $\rprn$, which we do by taking its left Kan extension~$\lanf$ along the poset embedding $\iota: \rn \stackrel{\cong}{\longrightarrow} \{0\}\times\rn\hookrightarrow\rprn$. The extension~$\lanf$ merely contains copies of~$\hs(\ff)$ in each vertical hyperplane~$\{\delta\}\times\rn$, connected by horizontal identities, as illustrated in Figure~\ref{fig:example_kan_extension_R2}. 
Our third main result (Theorem~\ref{thm:estimator_varying_radius}) states that the $n$-dimensional interleavings happening between $\hs(\ff)$ and  $\msd$ inside each vertical hyperplane $\{\delta\}\times\rn$ for $\delta\in [2\e, \rhox/2)$ commute with the horizontal morphisms in $\ms$ and in $\lanf$, so that, within any vertical slab~$[2\e,\delta_0]\times\rn$ with $\delta_0<\rhox/2$, they all together form a {\em vertical $\omega(2\delta_0)$-interleaving} (i.e., a $\omega(2\delta_0)$-interleaving along the direction $(0,1,\cdots, 1)^\top$ in $\rprn$) between $\ms$ and $\lanf$. 
This property, which is stronger than having an ordinary interleaving (Remark~\ref{rem:vert_wins}), intuitively implies that prominent features in $\oracle$ are not just present in $\msd$ at individual scales~$\delta$ in the range $[2\e, \delta_0]$, but also persist across large ranges of scales in $\ms$. In particular, they are not ephemeral in~$\rprn$ and should be detected by fine enough invariants of $(n+1)$-parameter modules. Thus, instead of probing $\rp$ in search for a good value for~$\delta$ as in~\cite{chazal2011scalar,chazal2013persistence}, we can now use the structure of the $(n+1)$-parameter module~$\ms$ to identify a range of  relevant values for parameter~$\delta$.  
  See Figure~\ref{fig:example_M_from_rips} for an illustration, and Figure~\ref{fig:vertical_interleaving_wins_exact} for a comparison of Theorem~\ref{thm:estimator_varying_radius} with Theorem~\ref{thm:estimator_fixed_radius}.

Since the vertical interleaving between $\ms$ and $\lanf$ implies an ordinary interleaving, it induces guarantees on the approximation of any interleaving-stable invariant of $\lanf$ by the corresponding invariant of $\ms$. Among these invariants, some behave particularly nicely with respect to vertical interleavings.  We illustrate this with multigraded Betti numbers, showing that the vertical interleaving between  $\ms$ and $\lanf$ within a slab leads to a vertical bottleneck matching between their multigraded Betti numbers within that slab (Corollary~\ref{cor:stab_inv_rprn}), as can be seen in the close-up view in Figure~\ref{fig:example_M_from_rips}.

In practice, function values may be corrupted with measurement noise, meanwhile geodesic distances may have to be approximated from the input data. Our framework accounts for these imprecisions by suitably adapting the parameters of the estimators and their approximation bounds (\Cref{prop:estimator_noise_fixed_radius,prop:estimator_rips_noise_distance_fixed_radius,prop:estimator_noise_varying_radius,prop:estimator_rips_noise_distance_varying_radius}).

Still for practical purposes, we provide an algorithm (\Cref{alg:compute_presentation_imfg}) for computing a free presentation of $\ms$, and thus also of $\msd$ for any fixed $\delta\geq 0$, in time that is comparable to that of computing a free presentation of the persistent homology of a single multifiltration. From a free presentation, a variety of invariants for $\ms$ and $\msd$, including their multigraded Betti numbers, can be derived using existing software. We have implemented our algorithm in the \texttt{multipers} library~\cite{multipers} and used it in our experiments.

\subsection*{Connection to function-geometric multifiltrations}
Our two main estimators derive in a certain way from the \emph{function-Rips multifiltration}~$\RR^{\bullet}(\FP)$. Specifically, $\ms$ is a {\em horizontal smoothing} (i.e., a smoothing in the sense of~\cite{chazal2016structure} along the first coordinate axis in~$\rprn$) of~$\hs(\RR^{\bullet}(\FP))$, while $\msd$ is the restriction of that smoothing to some fixed vertical hyperplane $\{\delta\}\times\rn$.  Similarly, the estimators $\OO^{\delta}(\FP)$ and $\CC^{\delta}(\FP)$ are restrictions, to that same hyperplane, of the {\em function-offset multifiltration}~$\OO^{\bullet}(\FP)$ and {\em function-\v{C}ech multifiltration}~$\CC^{\bullet}(\FP)$, respectively. As such, all the estimators considered in the paper are based on what we designate as {\em function-geometric multifiltrations}, which combine geometric filtrations with sublevel sets of functions. The introduction of $\CC^{\bullet}(\FP)$ and $\RR^{\bullet}(\FP)$ dates back to the first TDA paper on multiparameter persistence~\cite{carlsson2007theory}, but the line of work on function-geometric multifiltrations really started a few years ago, with the study of their stability under perturbations of the input and in particular under the presence of outliers. Several provably stable function-geometric multifiltrations have been introduced, including the \emph{multicover} bifiltration, the \emph{subdivision-\v{C}ech} and \emph{subdivision-Rips} bifiltrations, the \emph{degree-\v{C}ech} and \emph{degree-Rips} bifiltrations. The related work focuses on their stability, computation, or approximation by sparser filtrations on finite point clouds---see~\cite{alonso2024sparse,alonso2024delaunay,blumberg2024stability,buchet2024sparse,edelsbrunner2021multi,hellmer2024density,lesnick2024nerve,lesnick2024sparse,lesnick-wright,rolle2024stable,sheehy2012multicover}.  Our work offers a different perspective on function-geometric multifiltrations, by studying their convergence properties as estimators of a filtered metric space~$(X, \ff)$.
Our work also goes beyond the case~$n=1$, considering functions~$\ff$ valued in~$\rn$ for arbitrary~$n$. 

\subsection*{Structure of the paper}
We provide background material in \cref{sec:background}. Our problem statement, estimators, and main results are detailed in \cref{sec:results}.
Their proofs are given in
\cref{sec:fix_radius,sec:vary_radius,sec:stat}.  In \cref{sec:estim_comput} we provide our algorithm for computing presentations of images of morphisms between finitely presented persistence modules.
Finally, in \cref{sec:experiments} we present experimental results that illustrate our theoretical guarantees and we investigate additional properties of our estimators.

%% file: Sections/backgd.tex
\section{Background}\label{sec:background}
We assume some familiarity with basic category theory,  algebraic topology, measure theory, and topological data
analysis.  Let $\topc$ denote the category of topological spaces, and $\vect$ (resp. $\Vect$) the category of finite-dimensional (resp. all) vector spaces over a fixed field~$\kkk$. We also fix an arbitrary degree in homology, denoted
by~$*$, and we write $H_*(-)$ for singular homology groups in degree~$*$ with coefficients in~$\kkk$.

\subsection*{Filtrations and persistence modules}
We see $\rn$ as the product of $n$ copies of the real line, equipped with the product order noted~$\leq$. Thus, two
points $\bmx, \bmy\in\rn$ satisfy $\bmx\leq\bmy$ whenever $\bmx_i\leq\bmy_i$ for all $1\leq i\leq n$. We denote by
$(-\infty, \bmx]$  the downset $\{\bmy\in\rn \mid \bmy\leq\bmx\}$, and by $[\bmx, +\infty)$ the upset $\{\bmy\in\rn
\mid \bmy\geq\bmx\}$.
An \emph{$n$-parameter filtration} is a functor $\FF:\rn\to\topc$ whose constituent maps are inclusions, which means
that we have a topological space $\FF_{\bmx}$ for all $\bmx\in\rn$ and an inclusion $\FF_{\bmx, \bmy}\colon
\FF_{\bmx}\hookrightarrow\FF_{\bmy}$ for all $\bmx\leq\bmy\in\rn$.
An \emph{$n$-parameter persistence module} is a functor $M:\rn\to\Vect$, which
means that we have a $\kkk$-vector space~$M_{\bmx}$ for all $\bmx\in\rn$ and a
$\kkk$-linear map $M_{\bmx, \bmy}\colon M_{\bmx}\to M_{\bmy}$ for all
$\bmx\leq\bmy\in\rn$, such that $M_{\bmx, \bmx} = \id_{M_{\bmx}}$ and
$M_{\bmx,\bmz}= M_{\bmy,\bmz}\circ M_{\bmx,\bmy}$ for all
$\bmx\leq\bmy\leq\bmz\in\rn$.
{A {\em morphism} of persistence modules $M\to N$ is a natural
	transformation between functors, i.e., a family of linear maps
	$\varphi_{\bmx}\colon M_{\bmx}\to N_{\bmx}$ such that $N_{\bmx,\bmy} \circ
	\varphi_{\bmx} = \varphi_{\bmy}\circ M_{\bmx,\bmy}$ for all $\bmx\leq
\bmy\in\rn$.}

\subsection*{Geodesic metric spaces}\label{def:geodesic_space}
Throughout the paper, unless otherwise specified, $(X, d_X)$ is a compact
geodesic metric space. This means that, for all $x,y\in X$, there is a
shortest continuous path from~$x$ to~$y$ in~$X$, of length equal
to~$d_X(x,y)<\infty$.
In particular, the space is path-connected.
Let $P$ be a finite set of points of $X$. Given $\varepsilon>0$, we say
$P$ is a \emph{geodesic $\varepsilon$-sample} of $X$ if $\dH(P,X)<\e$, where $\dH$ denotes the Hausdorff distance in
$(X,d_X)$.
Define $B_X(x, r)$ as the open geodesic ball centered at $x\in X$ with radius
$r$, given by $B_X(x, r) = \{x' \in X \mid d_X(x, x') < r\}.$
A ball $B_X(x, r)$ is said to be \emph{convex} if, for any pair of points $y$ and $y'$ in $B_X(x, r)$, the shortest path in $X$ that connects $y$ to $y'$ is
unique and included in $B_X(x,r)$. Let
$\varrho(x):=\inf\{r>0\mid \text{$B(x,r)$ is not convex}\}$, and let $\rhox:=\inf\{\varrho(x) \mid x\in X \}\geq 0$. This quantity, called the  \emph{convexity radius} of~$X$, plays an important role because intersections of convex sets are convex and therefore contractible (see~\cref{sec:contractible_proof}), so the Nerve
Lemma~\cite[Corollary~4G.3]{hatcher} applies to unions of balls of radii less than~$\rhox$.

\subsection*{Modulus of continuity}\label{def:modulus_continuity}
A \emph{modulus of continuity} is a  non-decreasing  sub-additive
function $\omega:\rp\to\rp$ such that
$\omega(\delta)\xrightarrow[\delta\to 0]{}0$.
We say a function $\ff:
X\to \rn$ admits $\omega$ as a modulus of continuity, or that $\ff$ is $\omega$-continuous, if
\begin{equation}
	\left\lVert \ff(x) - \ff(y) \right\rVert_{\infty} \le \omega \left( d_X(x,y)\right), \quad \text{ for all $x,y\in X$.}
\end{equation}
If $\omega$ is of the form $\omega(x)=c\, x$ for a fixed constant $c\in\rp$, then a function $\ff$ that admits $\omega$
as a modulus of continuity is  $c$-Lipschitz continuous. More generally, if $\omega(x)=c\, x^\alpha$ for constants
$c\in\rp$ and $\alpha\in\RRR_{>0}$, then  $\ff$ is H\"older continuous of order~$\alpha$.

\subsection*{Functional and geometric filtrations}
Given a geodesic metric space $(X, d_X)$ and a function $\ff \colon X\to\rn$, the \emph{sublevel filtration}
$\FF:\rn\to\topc$ of $\ff$ is defined by $\FF_{\bmx}:=\ff^{-1}((-\infty,\bmx])$ for all $\bmx\in\rn$. The {\em
persistent homology} of~$\FF$ (and, by extension, of~$\ff$) is the persistence module
$\hs(\FF)$ (also written $\hs(\ff)$) defined by $\hs(\FF)_{\bmx} :=
\hs(\FF_{\bmx})$ and
$\hs(\FF)_{\bmx,\bmy} := \hs(\FF_{\bmx, \bmy})$ for all
$\bmx\leq\bmy\in\rn$.
When $\hs(\FF_{\bmx})$
	is finite-dimensional for all $\bmx\in\rn$, we say that $\hs(\ff)$
is \emph{pointwise finite-dimensional} (pfd).  

Given a finite set $P$ of points in X, and a real parameter $\delta \geq 0$, we call \emph{$\delta$-offset} of $P$ the
union of open geodesic balls $\OO^{\delta}(P):=\bigcup_{p\in P}B_X(p,\delta)$. We call \emph{$\delta$-\v{C}ech Complex}
of~$P$ the nerve of the collection of balls $\{B_X(p,\delta) \mid p\in P\}$, defined as the abstract simplicial complex
$\CC^{\delta}(P):=\{ \emptyset\neq \sigma\subseteq P \mid \bigcap_{p\in\sigma} B_X(p,\delta) \neq \emptyset\}$. For
$\delta<\rhox$, the balls $B_X(p,\delta)$ and their intersections are either empty or convex, so the Nerve
Lemma~\cite[Corollary~4G.3]{hatcher} ensures that $\OO^{\delta}(P)$ and $\CC^{\delta}(P)$ are homotopy equivalent,
which implies that their homology groups $\hs(\OO^{\delta}(P))$ and $\hs(\CC^{\delta}(P))$ are isomorphic. The
following persistent version of the Nerve Lemma ensures that the isomorphism is natural in both~$\delta$ and~$P$:
\begin{lemma}[\cite{chazal2011scalar}]
	\label{lem:nerve_theorem_commute}
	Let $P\subseteq P'$ be finite sets of points in a geodesic metric space~$(X, d_X)$. For any $\delta \leq \delta' <
	\rhox$, the following square commutes, where the isomorphisms are the ones provided by the Nerve
	Lemma~\cite[Corollary~4G.3]{hatcher}, and where the other two arrows are induced in homology by inclusions at the
	topological level:
	\[
		\begin{tikzcd}
			\hs(\CC^\delta(P)) \arrow[r] \arrow[d,"\cong"] & \hs(\CC^{\delta'}(P')) \arrow[d,"\cong"] \\
			\hs(\OO^{\delta}(P)) \arrow[r] & \hs(\OO^{\delta'}(P')).
		\end{tikzcd}
	\]
\end{lemma}
We call \emph{$\delta$-Rips complex} of~$P$ the abstract simplicial complex $\RR^{\delta}(P):= \{
\emptyset\neq\sigma\subseteq P \mid \max_{p,q\in\sigma} d_X(p,q) < \delta\}$.
\v{C}ech and Rips complexes of~$P$ are related as follows---see e.g.~\cite{co-tpbr-08}:
\begin{equation}\label{equ:include_cech_rips}
	\CC^{\delta}(P) \subseteq \RR^{2\delta}(P) \subseteq \CC^{2\delta}(P), \text{ for all $\delta\geq 0$}.
\end{equation}
Varying parameter~$\delta$ from $0$ to $+\infty$ yields respectively the {\em offset
filtration} $\OO^{\bullet}(P) := \left\{\OO^\delta(P)\right\}_{\delta\in\rp}$,
the {\em \v Cech filtration} $\CC^{\bullet}(P) :=
\left\{\CC^\delta(P)\right\}_{\delta\in\rp}$, and the {\em Rips filtration}
$\RR^{\bullet}(P) := \left\{\RR^\delta(P)\right\}_{\delta\in\rp}$.

\subsection*{Interval-decomposable and finitely presentable modules}

A non-empty subset $J\subseteq \rn$ is called an \emph{interval}  if it is convex and connected with respect to the
product order on~$\rn$.
The corresponding {\em interval module} $\kkk^J\colon:\rn\to\vect$ with support~$J$ is defined by:
\[
	\kkk^J_{\bmx} =
	\begin{cases}
		\kkk  &  \text{if}\ \bmx \in J  \\
		0   &  \text{otherwise}
	\end{cases}
	\quad \quad \quad
	\kkk^J_{\bmx, \bmy} =
	\begin{cases}
		\id_{\kkk}  &  \text{if}\ \bmx \leq \bmy \in J  \\
		0   &  \text{otherwise.}
	\end{cases}
\]
Given a multiset $\mathcal{D}$ of intervals in~$\rn$, we use the shorthand $\kkk^{\mathcal{D}}$ for the direct sum
$\bigoplus_{J\in\mathcal{D}}\kkk^J$.  A persistence module $M:\rn\to\vect$ is called \emph{interval-decomposable} if
$M\cong \kkk^{\BB(M)}$ for some multiset $\BB(M)$ of intervals in $\rn$: this multiset is then
unique~\cite{botnan2020decomposition}
and called the \emph{barcode}~of~$M$.

In the paper we call \emph{free interval} any upset $[\bmx,+\infty)$.
A module $M$ is free if it is interval-decomposable and its barcode $\BB(M)$ is made exclusively of free intervals.
Then, the \emph{rank} of~$M$ is the size of~$\BB(M)$. A persistence module $M$ is \emph{finitely presentable}~(\fp{})
if it is isomorphic to the cokernel of a morphism between two free $\rn$-indexed persistence modules of finite rank.
Such a morphism is referred to as a \emph{free presentation} of $M$.

In the following we will sometimes consider restrictions of modules or barcodes to {\em vertical slabs}, which are
intervals of the form $[a,b]\times\RRR^{n-1}$ for some $a\leq b\in\RRR$. When $A\subseteq\rn$ is a vertical slab, any
interval in~$\rn$ restricts to an interval in~$A$.

\subsection*{Morphisms of free modules and multigraded matrices}
We use the terminology and notations in \cite{lesnick2022computing} for multigraded matrices.
Let $M:\rn\to\vect$ be a persistence module. We say a non-zero element $v\in M$ is \emph{homogeneous} if $v\in
M_{\bmz}$ for some $\bmz\in\rn$. In this case, we write $\gr(v)=\bmz$. We define a \emph{basis} for a free persistence
module $F$ to be a minimal homogeneous set of generators.

Suppose we are given a basis $B=(b_j)_j$ for a finitely generated free persistence module $F$, with basis elements
ordered arbitrarily. We denote by $b_j$ the $j^{\text{th}}$ basis element of $B$. For $\bmz\in\rn$, any element $v\in
F_{\bmz}$ can be written as a $\kkk$-linear combination of basis elements with grades less than or equal to~$\bmz$:
\begin{equation}
	v = \sum\limits_{j:\gr(b_j)\leq\bmz} [v]^B_j F_{\gr(b_j),\bmz}(b_j).
\end{equation}
We write $[v]^B=([v]^B_j)_{1\leq j\leq |B|}$ for the vector of field coefficients completed with zeros at indices~$j$
for which $\gr(b_j)\nleq\bmz$.

Let $F'$ be another finitely generated free persistence module, and $B'=(b'_i)_i$ a basis for~$F'$ whose elements are
also ordered arbitrarily. A morphism $\gamma: F\to F'$ can then be represented up to natural isomorphism by a
$|B'|\times|B|$ matrix~$\Gamma$ with coefficients in~$\kkk$,
whose rows and columns are labeled by grades in~$\rn$ as follows:
\begin{itemize}
	\item the label of the $j^{\text{th}}$ column is $\gr(b_j)$,
	\item the label of the $i^{\text{th}}$ row is $\gr(b'_i)$.
\end{itemize}
This way, given some implicit choice of bases that we omit in our notations, any finite free presentation of a \fp{} persistence module is represented by a multigraded matrix $\Gamma$. We denote the $j^{\text{th}}$ column of~$\Gamma$ by $\Gamma[*,j]$, with label~$\gr(\Gamma[*,j])$, and we denote the $i^{\text{th}}$ row of~$\Gamma$ by $\Gamma[i,*]$, with label~$\gr(\Gamma[i,*])$.

\subsection*{Interleavings}

Let $\bmv\in\rzp^n$ be a fixed vector, and let $\varepsilon\geq 0$.
Given any functors $F,G:\rn\to\CC$ to some category~$\CC$, we say $F$ and $G$ are
\emph{$\varepsilon\bm{v}$-interleaved} if there are two natural transformations $f:F\to G[\varepsilon\bmv]$ and $g:G\to
F[\varepsilon\bmv]$ such that $g[\e\bmv]\circ f = \varphi^{2\varepsilon\bmv}_{F}$ and $f[\e\bmv]\circ g =
\varphi^{2\e\bmv}_{G}$. Here, $-[\e\bmv]$ denotes the $\e\bmv$-shift functor, defined on objects by $F[\e\bmv]_{\bmx} :=
F_{\bmx+\e\bmv}$ and on morphisms by $F[\e\bmv]_{\bmx, \bmy} := F_{\bmx+\e\bmv, \bmy+\e\bmv}$ for all
$\bmx\leq\bmy\in\rn$. Meanwhile, $\varphi^{2\e\bmv}_{F}\colon F\to F[2\e\bmv]$ and $\varphi^{2\e\bmv}_{G}\colon G\to
G[2\e\bmv]$ denote the natural morphisms from $F$ and $G$ to their respective $2\e\bmv$-shifts, defined by
$(\varphi^{2\e\bmv}_{F})_{\bmx} := F_{\bmx, \bmx+2\e\bmv}$ and $(\varphi^{2\e\bmv}_{G})_{\bmx} := G_{\bmx,
\bmx+2\e\bmv}$ for all $\bmx\in\rn$. The \emph{$\bmv$-interleaving distance} is defined as follows:
\[\dI(F,G) = \inf\{ \varepsilon \geq 0 \mid \text{  there exists an $\varepsilon\bmv$-interleaving between $F$ and
$G$}\}\in[0,+\infty].\]
In the following, we consider two specific types of interleavings:
\begin{enumerate}
	\item \textbf{Ordinary interleavings}:  Let
		$\bmv=\bm{1}:=(1,\cdots,1)^T$. Then, $\e\bm{1}$-interleavings are just ordinary $\varepsilon$-interleavings from the
		TDA literature,
  and $\dIo$ is called the \emph{ordinary interleaving distance}.
	\item \textbf{Vertical interleavings}: Let
		$\bmv=\bm{1}_0:= (0,1,\cdots,1)^T.$ Then, $\e\bm{1}_0$-interleavings are called \emph{vertical
		$\varepsilon$-interleavings}, and $\dIv$ is called the \emph{vertical interleaving distance}.
\end{enumerate}

\begin{remark}\label{rem:vert_wins}
	Vertical interleaving implies ordinary interleaving. Indeed, the vertical interleaving maps can be composed with the
	horizontal structure morphisms of the modules to form an ordinary interleaving. The amplitude~$\e$ of the interleaving
	does not change in the process. Therefore, vertical interleaving is a stronger condition than ordinary interleaving.
\end{remark}

\subsection*{Bottleneck matchings}
Let $\BB, \CC$ be multisets of intervals in $\rn$. Let $\bmv\in\rzp^n$ and $\e\geq 0$. An \emph{$\e\bmv$-bottleneck
matching} between $\BB$ and $\CC$ consists of a bijection $h:\BB'\to\CC'$ for some $\BB'\subseteq \BB$ and $\CC'
\subseteq \CC$ such that:
\begin{itemize}
	\item for all $I\in\BB'$, the interval modules $\kkk^I$ and $\kkk^{h(I)}$ are $\e\bmv$-interleaved;
	\item for all $I\in\BB\backslash \BB'\sqcup \CC\backslash \CC'$, the interval module $\kkk^I$ is $\e\bmv$-interleaved
		with the zero module.
\end{itemize}
The \emph{$\bmv$-bottleneck distance} is defined as follows:
\[
	\dB(\BB,\CC)=\inf\{ \e\geq 0\mid \text{there exists an $\e\bmv$-bottleneck matching between $\BB$ and
	$\CC$}\}\in[0,+\infty].
\]
In the following, we call \emph{ordinary $\varepsilon$-bottleneck matching} an $\e\bmv$-bottleneck matching with
$\bmv=\bm{1}$, and \emph{vertical $\e$-bottleneck matching} an $\e\bmv$-bottleneck matching with $\bmv=\oneo$.

\subsection*{Multigraded Betti numbers}

A \fp{} persistence module over $\rn$ is projective if and only if it is free. A \emph{projective resolution} of
a fp module~$M$ is an exact sequence \( \cdots \xrightarrow{\partial_3} P_2 \xrightarrow{\partial_2} P_1
\xrightarrow{\partial_1} P_0 \twoheadrightarrow M\), denoted by $P_\bullet\twoheadrightarrow M$,  where each $P_i$ is
fp and projective (hence free).
The resolution is \emph{minimal} if, for each homological degree $ i \in \NNN$, the free module \( P_i \) is of minimal
rank among all possible $i$-th terms in any projective resolution of $M$. The minimal projective resolution, when it
exists, is unique up to isomorphism, which implies that the barcode of each term $P_i$ is uniquely defined; this
barcode is called the \emph{multigraded Betti number} of $M$ in homological degree $i$, denoted by $ \beta_i(M)$. By
Hilbert's syzygy theorem, every \fp{} $\rn$-persistence module has a minimal projective
resolution~$P_\bullet\twoheadrightarrow M$ of  \emph{length}  at most~$n$, where the length is defined as the largest
$i \in \mathbb{N}$ such that $P_i \neq 0$. We let $\beta_{2\NNN}(M):=\bigsqcup_{i\in 2\NNN} \beta_i(M)$ and
$\beta_{2\NNN+1}(M):=\bigsqcup_{i\in 2\NNN+1} \beta_i(M)$.

The stability of mutigraded Betti numbers has been established as follows in the ordinary interleaving and bottleneck
distances:
\begin{theorem}[\cite{oudot2024stability}]
	\label{thm:stability_betti}
 For any \fp{} modules $M,N:\rn\to\vect$:
	\[
		\dBo(\beta_{\even}(M)\sqcup \beta_{\odd}(N),\beta_{\even} (N)\sqcup \beta_{\odd}(M))\leq
		\begin{cases}
			(n^2-1)\dIo(M,N), & \text{if $n>1$},\\
			2\dIo(M,N), & \text{if $n=1$}.
		\end{cases}
	\]
\end{theorem}

More precisely, the result from~\cite{oudot2024stability} says that, if $M$ and $N$ are ordinarily $\e$-interleaved,
then there exists an $O(\e)$-matching between barcodes $\beta_{\even}(M)\sqcup \beta_{\odd}(N)$ and $\beta_{\even}
(N)\sqcup \beta_{\odd}(M)$. In the case of a vertical $\e$-interleaving between $M$ and $N$, the conclusion can be
strengthened with the existence of a vertical $O(\e)$-matching between the barcodes, for a slightly larger constant
factor:
\begin{theorem}\label{thm:stability_betti_vertical}
	Let  $a\leq b\in\RRR$. For any \fp{} modules ${M,N\colon [a,b]\times\rn\to \vect}$:
	\[
		\dBv(\beta_{\even}(M)\sqcup \beta_{\odd}(N),\beta_{\even}(N)\sqcup \beta_{\odd}(M))\leq
		\begin{cases}
			(n-1)(n+2)\dIv(M,N), & \text{if $n>1$},\\
			3\dIv(M,N), & \text{if $n=1$}.
		\end{cases}
	\]
\end{theorem}
The proof mirrors that of Theorem~\ref{thm:stability_betti} and is provided in Appendix~\ref{sec:proof_stability_inv}
for completeness.

%% file: Sections/background/stats.tex
{
\subsection*{Statistics}\label{background:stats}\
In what follows, we fix a \emph{probability space} $(\Omega, \mathcal A, \Pbb)$, where $\Pbb$ is a measure of total mass~$1$, i.e., $\Pbb(\Omega)=1$, on the measurable space $(\Omega,\mathcal A)$.

Given another measurable space~$\measspacetuple E$,
a \emph{random variable} {taking values in} $\measspacetuple E$ is a measurable function $Z\colon (\Omega,
	\mathcal A)\to \measspacetuple E$.
We adopt the following terminology and notations, where $\PP(E)$ denotes the set of probability measures on~$\measspacetuple E$:
\begin{enumerate}[label=\textnormal{(\textit{\roman*})}]
	\item The {\em law} of a random variable $Z$ on $(E,\setmeasurableset E)$ is the
	      probability measure $P_Z\in \PP(E)$ defined by:
	      \begin{equation*}
		      \forall A\in \setmeasurableset E, \quad P_Z(A) := \Pbb \left(
		      Z^{-1}(A)\right) = \Pbb \left( Z\in A \right),
	      \end{equation*}
	      where we denote ${Z \in A}$ by $Z^{-1}(A)$.
We write $Z \sim \mu$ when $\mu = P_Z$.
	\item When $Z$ is integrable, its \emph{expectation} is defined as
	      \begin{equation*}
		      \Ebb(Z) = \int_\Omega Z(\omega)\dd \Pbb(\omega) = \int_{E} z \dd P_Z(z).
	      \end{equation*}
	      If $Z\sim\mu$, we write $\Ebb_{Z\sim \mu}$ for $\Ebb$ when
	      the law of $Z$ is not clear from the context.

	\item Random variables $Z_1,\dots,Z_k$ on~$\measspacetuple E$ are said to be \emph{(mutually) independent} if the law of the joint random variable $(Z_1,\dots,Z_k)$ on~$\measspacetuple {E\times\cdots\times E}$ is the product measure $P_{Z_1}\otimes P_{Z_2}\otimes\cdots\otimes P_{Z_k}$.
	\item Given a probability measure $\mu \in \PP(E)$, a \emph{$k$-sample} from~$\mu$ is a $k$-tuple $(Z_1,\dots,Z_k)$ of random variables that are independent and identically distributed (i.i.d.) with law~$\mu$.
	\item For a metric space $(X,d_X)$, we always consider its associated canonical
	      measurable space $(X, \mathcal{B}(X))$, where $\mathcal{B}(X)$ denotes the Borel
	      $\sigma$-algebra of $X$, i.e., the $\sigma$-algebra generated by the open sets
	      of $(X,d_X)$. 
\end{enumerate}

\begin{definition}[$(a,b)$-standard measure]\label{def:ab_std_meas}
	Given a metric space~$(X,d_X)$ and two real numbers $a>0$ and $b> 0$, a probability measure~$\mu$  on~$X$ is said to be \emph{$(a,b)$-standard} if 
	\begin{equation*}
		\forall x\in \supp(\mu) \textnormal{ and } r>0,\quad  \mu(B_{X}(x,r)) \ge \min \left\{ 1,  ar^b \right\} .
	\end{equation*}
	We denote by
	$\abstd X$
	the set of $(a,b)$-standard {probability} measures on $X$.
\end{definition}

\begin{definition}[Convergence rates]\label{def:convergence_rate}
	Let $(Z_k)_{k\in\NNN}$ be a sequence of random variables taking values in a metric space $(X,d_X)$,
	and let $Z$ be another random variable.
	The sequence $(Z_k)_{k\in\NNN}$ is said to {\em converge to $Z$} if it does so in expectation, i.e.,
	if $\Ebb \left( d_X(Z_k,Z) \right)\xrightarrow[k\to \infty]{}
		0$,
		where $d_X(Z_k,Z)$ is regarded as a random variable taking values in $(\RRR, |\cdot|)$.
	Given a non-increasing sequence of positive numbers $ (\xi_k)_{k\in \mathbb{N}}$ with $\xi_k \xrightarrow[k\to \infty]{} 0$, we say that $(Z_k)_{k\in\NNN}$ \emph{converges with rate at least}
	$(\xi_k)_{k\in\NNN}$ if
	$\Ebb \left( d_X(Z_k,Z) \right) \le \xi_k$ for any index $k\in \mathbb{N}$.
\end{definition}

\begin{definition}[Estimator]
	Let $\measspacetuple E$ be a measurable space and let $\theta_*\in E$.
	Let $X_k=(Z_1, \dots, Z_k)$ be a $k$-tuple of random variables taking values in some measurable space $\measspacetuple F$. 
	An \emph{estimator} $\hat\theta_k$ of $\theta_*$ \emph{based on} $X_k$ is an
	$X_k$-measurable random variable in $\measspacetuple E$, i.e., 
	there exists a measurable function $h\colon\measspacetuple F^k\to \measspacetuple E$ such that $\hat \theta_k = h(Z_1,\dots, Z_k)$.
\end{definition}

\begin{definition}[Minimax convergence rates]
	Let $\mathcal P$ be a set of probability measures on a metric space $(X,d_X)$ and $(Z_k)_{k\in \mathbb{N}}$ a sequence of i.i.d.\  random variables sampled according to $\mu$ in $\mathcal P$.
	Fix a point $\theta_*\in X$. For each $k\in\NNN$, let $\Theta_k$ denote the set of estimators of $\theta_*$ based on $(Z_1,\dots,Z_k)$. 
	
	Let $(\xi_k)_{k\in \mathbb{N}}$ be a
	non-increasing sequence of positive real numbers.
	The sequence $(\xi_k)_{k\in \mathbb{N}}$ is a \emph{minimax} rate on the tuple $(\left( \Theta_k
		\right)_{k \in \mathbb{N}}, \mathcal P)$ if there exist two constants $0<c<C<\infty$ such that, for
		any $k\in \mathbb{N}$, we have:
	\begin{equation*}
		c\xi_k \le
		\inf_{\hat\theta_k\in \Theta_k}\sup_{\mu\in \mathcal P}\Ebb_{(Z_1,\cdots,Z_k)\sim\mu^{\otimes k}}
		\left( d_X(\hat\theta_k,\theta_*) \right)
		\le 
		C\xi_k.
	\end{equation*}
	The rate $(\xi_k)_{k\in\NNN}$ is \emph{quasi-minimax} if there exists another constant $\alpha\in \mathbb{N}$ such that, for any $k\in\NNN$, we have: 
	\begin{equation*}
		c\xi_k \le
		\inf_{\hat\theta_k\in \Theta_k}\sup_{\mu\in \mathcal P}\Ebb_{(Z_1,\cdots,Z_k)\sim\mu^{\otimes k}}
		\left( d_X(\hat\theta_k,\theta_*) \right)
		\le 
		C(\log k)^\alpha\xi_k.
	\end{equation*}
	Given a modulus of continuity $\omega\colon \mathbb{R}_{\ge 0} \to \mathbb{R}_{\ge 0}$, the rate $(\xi_k)_{k\in\NNN}$ is \emph{$\omega$-minimax}
	if there exist two constants $0<c<C<\infty$ such that, for any $k\in \mathbb{N}$, we have:
	\begin{equation*}
		c\omega(\xi_k) \le
		\inf_{\hat\theta_k\in \Theta_k}\sup_{\mu\in \mathcal P}\Ebb_{(Z_1,\cdots, Z_k)\sim\mu^{\otimes k}}
		\left( d_X(\hat\theta_k,\theta_*) \right)
		\le 
		C\omega(\xi_k). 
		\end{equation*}
		The rate $(\xi_k)_{k\in\NNN}$ is \emph{$\omega$-quasi-minimax} if there exist constants $\alpha\in\NNN$ and $C'>0$ such that, for any $k\in\NNN$, we have:
		\begin{equation*}
		c\omega(\xi_k) \le
		\inf_{\hat\theta_k\in \Theta_k}\sup_{\mu\in \mathcal P}\Ebb_{(Z_1,\cdots, Z_k)\sim\mu^{\otimes k}}
		\left( d_X(\hat\theta_k,\theta_*) \right)
		\le 
		C\omega(C'(\log k)^{\alpha}\xi_k).
	\end{equation*}
	
	A sequence of estimators $( \hat\theta_k )_{k \in \mathbb{N}}\in \Theta$
	that achieves a minimax (resp. quasi-minimax, $\omega$-minimax,  or
	$\omega$-quasi-minimax) rate is called a \emph{minimax} (resp. \emph{quasi-mimimax}, \emph{$\omega$-minimax},  or $\omega$-\emph{quasi-minimax})
	estimator.
\end{definition}
}

%% file: Sections/main_results.tex
\section{Problem statement and main results}\label{sec:results}
Let $(X, d_X)$ be a compact geodesic metric space with positive
convexity radius~$\rhox$, and let
$\ff:X\to\RRR^n$ be a continuous function. 
Here, $\rn$ is equipped with the $\infty$-norm. 
Both $X$ and $\ff$ are unknown. By the Heine–Cantor Theorem,
$\ff$ admits a modulus of continuity $\omega\colon\rp \to \rp$. Neither $\omega$ nor $\rhox$ needs to be known, as they only play a role in our bounds and not in our constructions. The function $\ff$ is not assumed to be pfd, however our assumptions imply that it is tame in the sense that $\rk(\hsf_{\bmx,\bmy})$ is finite for all $\bmx<\bmy$ with $\bmx_i<\bmy_i$ for each $i\in\{1,\cdots,n\}$ (see Corollary~\ref{cor:thm_fixed_radius_implies_tame}).

Our input is a finite point cloud $P\subseteq X$ such that $\ff|_P$,
the restriction of~$\ff$ to~$P$, is known, as well as the pairwise geodesic
distances $d_X(p,q)$ between the points $p,q\in P$.
The problem we address is that of estimating~$\hs(\ff)$, the persistent
homology of the sublevel filtration of $\ff$, from our input. For this we
assume that $P$ is a geodesic $\e$-sample of~$X$, for some possibly unknown $\e$ that we assume
to be small enough compared to the convexity radius~$\rhox$.

To build our estimator, we first construct
the \emph{function-Rips multifiltration}~$\RR^{\bullet}(\FP):=
\left\{\RR^\delta \left( \FF_{\bmx}\cap P  \right)\right\}_{\delta\in\rp,\, \bmx\in\rn}$
from the restriction of $\ff$ to $P$.
Then, following~\cite{chazal2011scalar}, our estimator is a smoothed version
of~$\hs(\RR^{\bullet}(\FP))$, denoted by~$\hs(\Rinc(\FP))$ and defined formally
as the image of
the morphism of persistence modules induced in homology by the inclusion of the
filtration $\RR^{\bullet}(\FP)$ into its horizontal rescaling by a factor
of~$2$, that is:
\[\hs(\Rinc(\FP)) := \im\ \hs\left(\RR^{\bullet}(\FP) \hookrightarrow \RR^{2\bullet}(\FP)\right),\]
where \(\RR^{2\bullet}(\FP):= \left\{\RR^{2\delta}(\FF_{\bmx}\cap P)\right\}_{\delta\in\rp,\, \bmx\in\rn}.\)
\begin{remark}\label{rem:1vs2Rips}
The rationale behind using $\hs(\Rinc(\FP))$ instead of
$\hs(\RR^{\bullet}(\FP))$ itself as an estimator for~$\hsf$ is that, while the Rips complex can recover the homological type of a space from a finite sampling under some regularity
conditions on that space~\cite{l-vrcms-01}, in general it takes a pair of Rips
complexes with parameters within a factor of at least~$2$ of each other to do
so~\cite{co-tpbr-08}. And since sublevel sets may potentially behave wildly 
when the level approaches critical values of the function, in our work we apply the
construction for general spaces. It is unknown whether $\hs(\RR^{\bullet}(\FP))$ itself can approximate $\hsf$ under such general assumptions on~$X$ and~$\ff$ as ours.
\end{remark}

In the course of our analysis of the behavior of~$\hs(\Rinc(\FP))$, we will consider intermediate constructions, namely
the persistent homologies of:
\begin{itemize}
	\item the \emph{function-offset filtration}
		$\OO^{\bullet}(\FP):= \left\{\OO^{\delta}(\FF_{\bmx}\cap P)\right\}_{\delta\in\rp,\, \bmx\in\rn}$;
	\item the \emph{function-\v{C}ech filtration}
		$\CC^{\bullet}(\FP):= \left\{\CC^{\delta}(\FF_{\bmx}\cap P)\right\}_{\delta\in\rp,\, \bmx\in\rn}$.
\end{itemize}
These can also serve as alternative estimators when they can be built, for instance---in the case
of~$\CC^{\bullet}(\FP)$---when we can test the emptiness of the intersection of finitely many open geodesic balls.

We use our estimators in two types of scenarios: (1) when we have an estimate of the value~$\e$ for which $P$ is an
$\e$-sample of~$X$, and (2) when we do not have such an estimate at our disposal. The first scenario is the one
considered in previous work like~\cite{chazal2011scalar}, and it is relevant in that there is an abundant literature in
statistics on scale parameter estimation---some of which has already been used in the context of TDA with real-valued
functions~\cite{carriere2018statistical}. The second scenario is more general, and our approach to it leverages the
multiparameter setting to avoid the prior estimation of the scale parameter.  We investigate the two scenarios in
Sections~\ref{sec:est_e_known} and~\ref{sec:est_e_unknown} respectively. In Section~\ref{sec:handle_noise_in_input} we
	study the robustness of our results under perturbations of the input pairwise geodesic distances or
function values.


\subsection{Estimating $\hsf$ with known~$\e$}
\label{sec:est_e_known}

Suppose we know~$\e$ or some reasonable estimate. Then we can approximate our target $\hsf$ by restricting our
estimators to a vertical hyperplane $\{\delta\}\times\rn\subseteq\rprn$ for a suitable choice of parameter~$\delta$. We
denote these restrictions respectively by $\hs(\OdFP)$, $\hs(\CdFP)$
and $\msd$.

\begin{theorem}\label{thm:estimator_fixed_radius}
	Let $X$ be a compact geodesic space,
	let $\ff:X\to\rn$ be an $\omega$-continuous
	function 
	for some modulus of
	continuity $\omega:\rp\to\rp$, and
	let $P$ be a finite
	geodesic $\varepsilon$-sample of $X$. Then:
	\begin{enumerate}[label=\textnormal{(\textit{\roman*})}]
		\item \label{enum:thm:fixed_radius_offset}for any choice of $\delta\geq \varepsilon$, the persistence modules $\hsf$
			and $\hs(\OdFP)$ are ordinarily $\omega(\delta)$-interleaved;
		\item \label{enum:thm:fixed_radius_cech}for any choice of $\delta\in [\varepsilon,\rhox)$, the modules $\hsf$ and
			$\hs(\CdFP)$ are ordinarily $\omega(\delta)$-interleaved;
		\item \label{enum:thm:fixed_radius_rips}for any choice of $\delta\in [2\varepsilon,\rhox/2)$, the modules $\hsf$ and
			$\msd$ are ordinarily $\omega(2\delta)$-interleaved.
	\end{enumerate}
\end{theorem}
The proof of this result is given in Section~\ref{sec:fix_radius}.

\cref{thm:estimator_fixed_radius} can be read in two different ways. First, when the modulus of continuity $\omega$ is known, we obtain explicit bounds on the interleaving distance between our estimators and the target $\hsf$. Second, when we only know that $\ff$ is continuous, the Heine–Cantor Theorem implies that
$\ff$ admits some modulus of continuity~$\omega$, and even without knowning that~$\omega$, we are guaranteed by \cref{thm:estimator_fixed_radius} that our estimators converge to $\hsf$ as $\e$ goes to zero under suitable choices of parameter~$\delta$.

When $\e$ is known exactly, the best choice for~$\delta$ is
the lower bound of its admissible interval, i.e., $\delta=\e$ in
\cref{enum:thm:fixed_radius_offset,enum:thm:fixed_radius_cech} and $\delta=2\e$ in \cref{enum:thm:fixed_radius_rips}.
When $\e$ is only known approximately, the admissible interval provides some leeway for the choice of~$\delta$.
As already alluded to, the unknown quantities $\omega$ and $\rhox$ appear in the bounds but are not involved in the
construction of the estimators. Also, there is no approximation guarantee on~$\hs(\RdFP)$ but there is one on its
smoothing~$\msd$ according to \cref{enum:thm:fixed_radius_rips}. This item is the generalization of the main result
of~\cite{chazal2011scalar} from real-valued functions to vector-valued functions.  The other items serve as
intermediate steps in its proof, and they generalize their counterparts from~\cite{chazal2011scalar} to vector-valued
functions. They also stand as independent approximation results when the corresponding estimators can be built.

An immediate consequence of Theorem~\ref{thm:estimator_fixed_radius} is that any $\dIo$-stable invariant computed on
our estimators approximates the corresponding invariant defined on the target~$\hsf$. Here is for instance the
approximation guarantee obtained for the multigraded Betti numbers of $\msd$, which follows from
Theorem~\ref{thm:stability_betti}.
\begin{corollary}\label{cor:stab_inv_rn}
	Under the hypotheses of Theorem~\ref{thm:estimator_fixed_radius}~\ref{enum:thm:fixed_radius_rips}, and assuming
	further that the module~$\hsf$ is \fp{}, for any choice of $\delta\in[2\e, \rhox/2)$ we have the following
	inequalities where $M=\hsf$ and $N=\msd$:
	\begin{align*}
		& \dBo\left(\beta_{\even}(M)\sqcup \beta_{\odd}(N),\beta_{\even} (N)\sqcup \beta_{\odd}(M)\right) \leq
		\begin{cases}
			(n^2-1)\omega(2\delta) & \text{if $n>1$}, \\
			2\omega(2\delta)       & \text{if $n=1$}.
		\end{cases}
	\end{align*}
\end{corollary}

Another consequence of \cref{thm:estimator_fixed_radius}~(i) is that, under our assumptions,  the persistence module $\hsf$ satisfies some form of tameness, described in the following corollary whose proof is given in~\cref{sec:proof_of_tame}.
\begin{corollary}\label{cor:thm_fixed_radius_implies_tame}
	Let $X$ be a compact geodesic space, and  let $\ff:X\to\rn$ be an $\omega$-continuous function for some modulus of continuity $\omega:\rp\to\rp$. Then the persistence module $\hs(\ff)$ is tame in the sense that the rank $\rk(\hs(\ff)_{\bmx,\bmy})$ is finite for all $\bmx < \bmy\in\rn$ with $\bmx_i < \bmy_i$ for each $i\in\{1,\cdots,n\}$.
\end{corollary}


Suppose now the points of~$P$ are {i.i.d.\ samples drawn from} some unknown probability
measure~$\mu$ supported on~$X$. In this setting, standard statistical
techniques can be used to estimate~$\e$ with high probability, then
Theorem~\ref{thm:estimator_fixed_radius} can be applied to get statistical
estimates of $\hsf$ or of any $\dIo$-stable invariant thereof. This holds
under some regularity conditions on the measure, typically
$(a,b)$-standardness (\cref{def:ab_std_meas}) for some known or unknown
parameters~$a,b$. The precise set of conditions is given in
Section~\ref{sec:stat}
(\cref{assumption:ab_std,assumption:mod_cont_reg,assumption:admit_modcont,assumption:compact_support,assumption:manifold_support}),
together with the analysis
of the statistical estimators (Propositions~\ref{prop:cv_rate_known},
\ref{prop:minimax_rate} and~\ref{prop:cv_rate_unknown_ab}). We summarize our
results in the following statement---see also~\cref{rem:quasi-minimax}.
\begin{theorem}\label{thm:stat_results_overview}
	Suppose the points $p_1, \dots, p_k$ of~$P$ are {i.i.d.\ samples drawn from} some unknown
	probability measure~$\mu$, such that
	\cref{assumption:ab_std,assumption:mod_cont_reg,assumption:admit_modcont,assumption:compact_support}
	of
	Section~\ref{sec:stat} hold---in particular, $\mu$ is $(a,b)$-standard
	(\cref{def:ab_std_meas}), {and the target function $\ff$ is $\omega$-continuous.}
	If $a,b$ are known, then
	{there exists an explicit sequence of positive numbers $(\delta_k)_{k \in \mathbb{N}}$ }
	such that
	the {sequences of approximations } $\left(\hs(\OO^{\delta_k}(\FP))\right)_{k\in \mathbb{N}}$,
	$ \left( \hs(\CC^{\delta_k}(\FP)) \right)_{k\in \mathbb{N}}$ and $ \left( {\msdk} \right)_{k\in \mathbb{N}}$ of~$\hsf$
	are consistent and
	$\omega$-quasi-minimax {estimators of $\hsf$}. If $a,b$ are unknown, then, under the extra
	\cref{assumption:manifold_support}, there exists an explicit sequence of random positive numbers $(\hat\delta_k)_{k
	\in \mathbb{N}}$ such that
	$\left(\hs(\OO^{\hat\delta_k}(\FP))\right)_{k\in \mathbb{N}}$,
	$ \left( \hs(\CC^{\hat\delta_k}(\FP)) \right)_{k\in \mathbb{N}}$ and $ \left( {\mshdk} \right)_{k\in \mathbb{N}}$ are
	consistent and
	$\omega$-quasi-minimax estimators of $\hsf$.
\end{theorem}


\subsection{Estimating $\hsf$ with unknown~$\e$}
\label{sec:est_e_unknown}
{When $\e$ is unknown and cannot be effectively estimated---for instance when
	\cref{assumption:ab_std,assumption:mod_cont_reg,assumption:admit_modcont,assumption:compact_support,assumption:manifold_support}
	of Section~\ref{sec:stat} are not satisfied, we adopt the usual approach in persistence theory {which} is to not fix
	the scale parameter~$\delta$ but rather to consider it as another filtration parameter. This means using the full $(n+1)$-parameter
	modules $\hs(\OO^{\bullet}(\FP))$, $\hs(\CC^{\bullet}(\FP))$
	and $\ms$ as estimators. In order to state approximation results, we  extend our target~$\hsf$ to a
	persistence module over~$\rprn$
	by taking its left Kan extension $\lanf$ along the poset embedding $\iota: \rn \hookrightarrow \rprn$ given by
	$\bmx\mapsto (0,\bmx)$ for all $\bmx\in\rn$. Since the category $\rn$ is small and the category $\Vect$ is cocomplete,
	$\lanf$ is well-defined as a functor $\rprn\to \Vect$, and given pointwise  by the colimit formula~\cite{mac2013categories}:
	\begin{equation}\label{eq:Lanf_colim}
          \forall (\delta, \bmx) \in\rprn, \quad \lanf_{(\delta, {\bmx})}= \varinjlim \limits_{\substack{{\bmy}\in\rn\\ \iota
		({\bmy}) \leq (\delta, {\bmx})}} \hsf_{\bmy} \cong
		\hsf_{\bmx}.
	\end{equation}
	Moreover, $\lanf$ has
	structural morphisms that are identity maps horizontally (i.e., along the
	first coordinate axis) and the structural morphisms of $\hsf$ vertically
	(i.e., along any  direction orthogonal to the first coordinate axis), as
	illustrated in Figure~\ref{fig:example_kan_extension_R2}.
}

\begin{theorem}\label{thm:estimator_varying_radius}
	Let $X$ be a compact geodesic space,
	let $\ff:X\to\rn$ be an $\omega$-continuous
	function 
	for some modulus of
	continuity $\omega:\rp\to\rp$, and
	let $P$ be a finite
	geodesic $\varepsilon$-sample of $X$. Then:
	\begin{enumerate}[label=\textnormal{(\textit{\roman*})}]
		\item 
			for any $\delta_0\geq\e$, within the slab $[\e, \delta_0]\times\rn$ the restricted modules
			$\lanf|_{[\varepsilon,\delta_0]\times\rn}$ and \\ $\hs(\OO^{\bullet}(\FP))|_{[\varepsilon,\delta_0]\times\rn}$ are
			vertically $\omega(\delta_0)$-interleaved;
		\item 
			for any $\delta_0\in[\e, \rhox)$, within the slab $[\e, \delta_0]\times\rn$ the restricted modules
			$\lanf|_{[\varepsilon,\delta_0]\times\rn}$ and \\ $\hs(\CC^{\bullet}(\FP))|_{[\varepsilon,\delta_0]\times\rn}$ are
			vertically $\omega(\delta_0)$-interleaved;
		\item 
			for any $\delta_0\in[2\e,\rhox/2)$, within the slab $[2\e, \delta_0]\times\rn$ the restricted modules
			$\lanf|_{[2\varepsilon,\delta_0]\times\rn}$ and $\ms|_{[2\varepsilon,\delta_0]\times\rn}$ are vertically
			$\omega(2\delta_0)$-interleaved.
	\end{enumerate}
\end{theorem}
The proof of the theorem is given in Section~\ref{sec:vary_radius}.

This approximation result encapsulates the previous one in the sense that
restricting the estimators to the vertical hyperplane $\{\delta\}\times\rn$ for
$\delta=\delta_0$ recovers Theorem~\ref{thm:estimator_fixed_radius}. However,
the result says something deeper, namely: that the interleavings in the
vertical hyperplanes $\{\delta\}\times\rn$, for $\delta$ ranging over
$[\e,\delta_0]$  (resp. $[2\e,\delta_0]$), commute with the horizontal
structural morphisms of the modules,  so that they all together form a vertical
interleaving in the slab $[\e, \delta_0]\times\rn$ (resp. $[2\e,
\delta_0]\times\rn$).
{Intuitively, this means that prominent features in $\hsf$ are not just
	present in $\msd$ at individual scales~$\delta$ in the range $[2\e,
\delta_0]$,} {but that they persist across large ranges of scales in
	$\ms$. In particular, they are not ephemeral in~$\rprn$ and can be captured by
stable invariants.}


In terms of approximation accuracy, the result says that, the more one looks to the left (i.e., toward small positive
values of~$\delta_0$) in the parameter space $\rprn$, the more precisely the estimators approximate the target, until
some point where $\delta_0$ becomes too small and the approximation breaks down. On the opposite side, the more one
looks to the right (i.e., toward large values of ~$\delta_0$), the more the estimators drift away from the target,
until again the approximation eventually breaks down if the estimator is $\hs(\CC^{\bullet}(\FP))$ or $\ms$.

\begin{example}\label{ex:M_from_rips}
	Let $X$ be a unit circle in the plane, equipped with the geodesic distance~$d_X$ given by arc length. Let $P$ be a
	geodesic $\varepsilon$-sample of $X$ for a fixed value $\varepsilon>0$. Let $\ff$ be the height function along the
	vertical direction. See Figure~\ref{fig:example_M_from_rips} (left) for an illustration. We consider homology in
	degree~$1$, so $\honef$ has a single interval summand $\kkk^{[1, +\infty)}$, which in turn implies that $\lanfone$ has
	a single interval summand $\kkk^{\rp\times[1, +\infty)}$, shown in magenta in Figure~\ref{fig:example_M_from_rips}
	(center and right). The figure also depicts the modules $\hone(\RR^{\bullet}(\FP))$ (in yellow),
	$\hone(\RR^{2\bullet}(\FP))$ (in green), and $\hone(\RR^{\bullet\to 2\bullet}(\FP))$ (in blue), each one of which
	turns out to be composed of a single large interval summand without extra noise as shown in the figure.
	As predicted by Theorem~\ref{thm:estimator_varying_radius}~(iii), the estimator $\hone(\RR^{\bullet\to
	  2\bullet}(\FP))$ approximates the target $\lanfone$ at least within the range $[2\e, \rhox/2)$,
          with a vertical precision that
	deteriorates progressively (not more than linearly) as $\delta$ increases within that range.
\end{example}

Theorem~\ref{thm:estimator_varying_radius}, like its
counterpart Theorem~\ref{thm:estimator_fixed_radius},
implies that provably $\dIo$-stable invariants computed
from our estimators approximate the corresponding
invariants defined on their target, which this time
is~$\lanf$. {And since the interleaving is now
	vertical---hence stronger than an ordinary interleaving
	according to Remark~\ref{rem:vert_wins}, we can get more
	precise statements for
	some invariants.
	For instance, here is below the guarantee
	obtained for the multigraded Betti numbers of~$\ms$,
	which follows from
	Theorem~\ref{thm:stability_betti_vertical} and the fact
	that restrictions of \fp{} modules to a closed vertical
slab are \fp{}.}

\begin{corollary}\label{cor:stab_inv_rprn}
	Under the hypotheses of Theorem~\ref{thm:estimator_varying_radius}~\ref{enum:thm:fixed_radius_rips}, and assuming
	further that the module~$\hsf$ is \fp{}, within any slab $[2\e, \delta_0]\times\rn$ with $2\e\leq\delta_0<\rhox/2$ we
	have the following inequalities where $M=\lanf|_{[2\e, \delta_0]\times\rn}$ and $N=\ms|_{[2\e, \delta_0]\times\rn}$:
	\begin{align*}
		& \dBv\left(\beta_{\even}(M)\sqcup \beta_{\odd}(N),\beta_{\even} (N)\sqcup \beta_{\odd}(M)\right) \leq
		\begin{cases}
			(n-1)(n+2)\,\omega(2\delta_0) & \text{if $n>1$}, \\
			3\,\omega(2\delta_0)          & \text{if $n=1$}.
		\end{cases}
	\end{align*}
\end{corollary}

In other words, within the designated slab,  there is a bottleneck matching of controlled amplitude that matches the
multigraded Betti numbers of $\lanf$ and of $\ms$ vertically, i.e., orthogonally to the first coordinate axis.

\begin{example}\label{ex:M_from_rips_inv}
	Going back to Example~\ref{ex:M_from_rips}, the zoomed-in part of the image to the right of
	Figure~\ref{fig:example_M_from_rips}
	shows the existence of a vertical bottleneck matching (in green) of controlled amplitude between the multigraded Betti
	numbers of the restrictions of $\lanfone$ and $\hone(\RR^{\bullet\to 2\bullet}(\FP))$ to the slab~$[2\e, \rhox/2)$.
\end{example}



\subsection{Robustness to noise in the input}
\label{sec:handle_noise_in_input}

Theorems~\ref{thm:estimator_fixed_radius} and~\ref{thm:estimator_varying_radius} are stated for when
exact geodesic distances and function values are provided as input. In practice, function values may be subject to
measurement noise, and geodesic distances may have to be approximated from the input data---typically using distances
in some neighborhood graph. Our framework can take these imprecisions into account, modulo some adaptation in the
parameters of the estimators and in their approximation bounds.

Assume first that the data points $p\in P$ are assigned function values
$\tilde{\ff}(p)$ that are different from $\ff(p)$, and let $\zeta=\max_{p\in
P} \Vert \tilde{\ff}(p)-\ff(p)\Vert_{\infty}$. Note that the function
$\tilde{\ff}:P\to \rn$ itself may not necessarily admit $\omega$ as a modulus
of continuity; in fact, no regularity condition is placed on~$\tilde f$.

\begin{proposition}\label{prop:estimator_noise_fixed_radius}
	Let $X, \rhox, \ff, \omega, P, \e$ be as in Theorem~\ref{thm:estimator_fixed_radius}, and let $\tilde\ff$ be as above.
	Then:
	\begin{enumerate}[label=\textnormal{(\textit{\roman*})}]
		\item for any choice of $\delta\geq \varepsilon$, the modules $\hsf$ and $\hs(\OO^{\delta}(\tFP))$ are ordinarily
			$(\omega(\delta)+\zeta)$-interleaved;
		\item for any choice of $\delta\in [\varepsilon,\rhox)$, the modules $\hsf$ and $\hs(\CC^{\delta}(\tFP))$ are
			ordinarily $(\omega(\delta)+\zeta)$-interleaved;
		\item for any choice of $\delta\in [2\varepsilon,\rhox/2)$, the modules $\hsf$ and $\msdt$ are ordinarily
			$(\omega(2\delta)+\zeta)$-interleaved.
	\end{enumerate}
\end{proposition}

Assume now that the geodesic distance $d_X$ is replaced by some non-negative symmetric bivariate function $\tilde d_X$  on the input
point cloud~$P$, with the following property  where $\lambda,\kappa>0$ are constants such that
$\lambda=1+4\frac{\e}{\kappa}$:
\begin{equation}\label{equ:mu_graph}
	\forall p,q\in P, \frac{d_X(p,q)}{\kappa}\leq \tilde d_X(p,q) \leq 1+\lambda \frac{d_X(p,q)}{\kappa}.
\end{equation}
Such guarantees hold typically for approximations of~$d_X$ built from graph distances in some $\kappa$-neighborhood
graph---see e.g.~\cite[Lemma~7.1]{oudot2010geodesic}. Here $\kappa$ is chosen by the user, and $\lambda$ is known as
long as $\e$ is. Note that $\tilde d_X$ itself does not have to be a distance, in particular it may not satisfy the triangle
inequality.
%
Let us build the function-Rips multifiltration on~$P$ using $\tilde d_X$ as the ground `metric', and let us denote it
by~$\RR^{\bullet}_{\tilde d_X}(\FP)$ for clarity. Note that it is possible to build that multifiltration because Rips
complexes only require a non-negative symmetric bivariate function on the points to be defined. By contrast, unions of balls in~$X$
and their nerves cannot be built from~$\tilde d_X$, because it is defined only on the point cloud~$P$.
\begin{proposition}\label{prop:estimator_rips_noise_distance_fixed_radius}
	Let $X, \rhox, \ff, \omega, P, \e$ be as in \cref{thm:estimator_fixed_radius} . Let $\lambda=1+4\frac{\e}{\kappa}$ be
	as in \cref{equ:mu_graph}, and let $\RR^{\bullet}_{\tilde d_X}(\FP)$ be as above. Then, for any choice of $\delta \in
	[1+2\lambda\frac{\e}{\kappa},\, \frac{1}{2\lambda}(\frac{\rhox}{\kappa}-1))$,
	the persistence modules $\hsf$ and $\hs(\RR_{\tilde d_X}^{\delta\to
	(1+2\lambda\delta)}(\FP))$ are ordinarily
	$\omega(\kappa(1+2\lambda\delta))$-interleaved.
\end{proposition}

Note that $1+2\lambda\delta > 2\delta$, so the smoothing of the function-Rips multifiltration has to be by a factor
larger than~$2$ to compensate for the noise in the input pairwise geodesic distances.

As $\e$ goes to~$0$, one can let the neighborhood radius~$\kappa$  go to~$0$ as well, no faster than~$\e$ so that
$\frac{\e}{\kappa}$ remains bounded above by a constant: this allows the lower bound $(1+2\lambda\frac{\e}{\kappa})$
for the choice of $\delta$ to be bounded above by a constant and thus makes the amplitude of the interleaving go to~$0$
at the limit.

\Cref{prop:estimator_noise_fixed_radius,prop:estimator_rips_noise_distance_fixed_radius} can
be combined to build estimators of~$\hsf$ that are robust both to noise in the function values and to noise in the
pairwise geodesic distances---we leave this as an exercise to the reader. They also induce analogous results for the
estimation of $\dIo$-stable invariants of $\hsf$.

The proofs of~\Cref{prop:estimator_noise_fixed_radius,prop:estimator_rips_noise_distance_fixed_radius} are simple adaptations of the proof of
Theorem~\ref{thm:estimator_fixed_radius}, described in Remarks~\ref{rmk:proof_for_noise_function_fixed_radius}
and~\ref{rmk:proof_for_noise_distance_fixed_radius} respectively. The same adaptations work for Theorem~\ref{thm:estimator_varying_radius} as well (see
Remark~\ref{rmk:proof_for_noise_varying_radius}), yielding the following robustness guarantees for the estimation
of~$\lanf$:
\begin{proposition}\label{prop:estimator_noise_varying_radius}
	Let $X, \rhox, \ff, \omega, P, \e$ be as in
	Theorem~\ref{thm:estimator_varying_radius}. Assume that each data point
	$p\in P$ is assigned a function value $\tilde\ff(p)$ that may be different
	from~$\ff(p)$, and let $\zeta=\max_{p\in P}
	\left\|\tilde\ff(p)-\ff(p)\right\|_\infty$.
	Then:
	\begin{enumerate}[label=\textnormal{(\textit{\roman*})}]
		\item 
			for any $\delta_0\geq\e$, within the slab $[\e, \delta_0]\times\rn$ the restricted modules
			$\restr{\lanf} {[\varepsilon,\delta_0]\times\rn}$ and \\
			$\restr{\hsobtffp}{[\varepsilon,\delta_0]\times\rn}$
			are vertically $(\omega(\delta_0)+\zeta)$-interleaved;
		\item 
			for any $\delta_0\in[\e, \rhox)$, within the slab $[\e, \delta_0]\times\rn$ the restricted modules
			$\restr{\lanf} {[\varepsilon,\delta_0]\times\rn}$ and \\
			$\restr {\hscbtffp}{[\varepsilon,\delta_0]\times\rn}$
			are vertically $(\omega(\delta_0)+\zeta)$-interleaved;
		\item 
			for any $\delta_0\in[2\e,\rhox/2)$, within the slab $[2\e, \delta_0]\times\rn$ the restricted modules
			$\restr{\lanf} {[2\varepsilon,\delta_0]\times\rn}$ and
			$\restr {\hsrbtffp}{[2\varepsilon,\delta_0]\times\rn}$
			are vertically $(\omega(2\delta_0)+\zeta)$-interleaved.
	\end{enumerate}
\end{proposition}

\begin{proposition}\label{prop:estimator_rips_noise_distance_varying_radius}
	Let $X, \rhox, \ff, \omega, P, \e$ be as in Theorem~\ref{thm:estimator_varying_radius}. Assume that the geodesic
	distance $d_X(p,q)$ between each pair of data points $p,q\in P$ is replaced by some value~$\tilde d_X(p,q)$, such that
	\eqref{equ:mu_graph} holds for some constants $\lambda, \kappa$ with $\lambda=1+4\frac{\e}{\kappa}$. As previously,
	denote by~$\RR^{\bullet}_{\tilde d_X}(\FP)$ the function-Rips multifiltration built on~$P$ using $\tilde d_X$ instead
	of~$d_X$.
	Then, for any $\delta_0 \in [1+2\lambda\frac{\e}{\kappa},\, \frac{1}{2\lambda}(\frac{\rhox}{\kappa}-1))$,  within the
	slab $[1+2\lambda\frac{\e}{\kappa},\,\delta_0]\times\rn$ the restricted modules
	$\lanf|_{[1+2\lambda\frac{\e}{\kappa},\delta_0]\times\rn}$ and $\hs(\RR_{\tilde d_x}^{\bullet\to
	(1+2\lambda\bullet)}(\FP))|_{[1+2\lambda\frac{\e}{\kappa},\delta_0]\times\rn}$ are vertically
	$\omega(\kappa(1+2\lambda\delta_0))$-interleaved.
\end{proposition}

%% file: Sections/proof_fix.tex
\section{Proof of Theorem~\ref{thm:estimator_fixed_radius}}\label{sec:fix_radius}

\begin{proof}[Proof of~\cref{thm:estimator_fixed_radius} \ref{enum:thm:fixed_radius_offset}]
	Recall that $\FF$ is the filtration generated by the sublevel sets of
	$\ff$.
	It is sufficient to prove that  $\FF_{{\bmx} -\omega(\delta)\bm{1}}\subseteq {\OO^\delta(\FF_{\bmx}\cap P)} \subseteq \FF_{{\bmx} +\omega(\delta)\bm{1}}$ for all $\delta\geq\varepsilon$ and ${\bmx} \in\rn$; the claim follows then by functoriality of $H_*$. From now on, we fix a $\delta\geq\varepsilon$ and an $\bmx\in\mathbb{R}^n$.

	For any $q\in {\OO^\delta(\FF_{\bmx}\cap P)}
	=\bigcup\limits_{{\substack{p\in P \\ \ff(p)\le \bmx}}}B_{X}(p,\delta)$,
	there exists a point $p\in P$ such that $\ff(p)\le
	\bmx$ and $d_X(p,q)<\delta$.
	Since $\ff$ admits $\omega$ as a modulus of
		continuity, we have $\|\ff(p)-\ff(q)\|_{\infty}\leq  \omega( d_X(p,q))\leq
	\omega(\delta)$, therefore $\ff(q)\leq \ff(p)+ \omega(\delta)\bm{1}\leq
	{\bmx} +\omega(\delta)\bm{1}$ and so $q\in\FF_{{\bmx}
		+\omega(\delta)\bm{1}}$.

		For any $q\in\FF_{{\bmx}-\omega(\delta)\bm{1}}$, we have $\ff(q)\leq
				{\bmx}  - \omega(\delta)\bm{1}$. Since $P$ is an $\varepsilon$-sample of
	$X$, there exists $p\in P$ such that $d_X(p,q) < \varepsilon\leq\delta$, that
		is, $q\in B_X(p,\delta)$. Then, since $\ff$ admits $\omega$ as a
		modulus of continuity, $\| \ff(p)-\ff(q) \|_{\infty}\leq
	\omega(d_X(p,q))\leq \omega(\delta)$. It follows that $\ff(p)\leq
	\ff(q)+\omega(\delta)\bm{1}\leq {\bmx} $,  and $q\in B_X(p,\delta) \subseteq
	\bigcup\limits_{{\substack{p'\in P \\ \ff(p')\le \bmx}}}B_X(p',\delta)={\OO^\delta(\FF_{\bmx}\cap P)}$.
\end{proof}

\begin{proof}[Proof of Theorem~\ref{thm:estimator_fixed_radius}~\ref{enum:thm:fixed_radius_cech}]
    Follows from Theorem~\ref{thm:estimator_fixed_radius}~\ref{enum:thm:fixed_radius_offset} combined with the isomorphism of persistence modules $\hs(\CdFP) \stackrel{\cong}{\longrightarrow} \hs(\OdFP)$ induced by Lemma~\ref{lem:nerve_theorem_commute} (which applies because $\delta<\rhox$). 
\end{proof}




\begin{proof}[Proof of Theorem~\ref{thm:estimator_fixed_radius}~\ref{enum:thm:fixed_radius_rips}]
Let $M=\msd$. To establish the ordinary $\omega(2\delta)$-interleaving between $M$ and $\hsf$, we construct two morphisms $\kappa^{\delta}:\hsf\to M[\omega(2\delta)\bm{1}]$ and $\gamma^{\delta}:M\to \hsf[\omega(2\delta)\bm{1}]$ such that:
\begin{equation}\label{equ:interleaving}
 \kappa^{\delta} [\omega(2\delta)\bm{1}] \circ\gamma^{\delta} = \varphi^{2\omega(2\delta)\bm{1}}_{M}
 \quad\quad\quad\text{and}\quad\quad\quad  
\gamma^{\delta} [\omega(2\delta)\bm{1}] \circ\kappa^{\delta} = \varphi^{2\omega(2\delta)\bm{1}}_{\hsf}. 
\end{equation}

More precisely, factoring the morphism $\hs(\RR^{\delta}(\FP))\to \hs(\RR^{2\delta}(\FP))$ through its image:
\( \hs(\RR^{\delta}(\FP))\twoheadrightarrow M \hookrightarrow \hs(\RR^{2\delta}(\FP)), \)
we define $\kappa^{\delta}:\hsf\to M[\omega(2\delta)\bm{1}]$ as the following composition, where the last arrow is the $\omega(2\delta)\bm{1}$-shift of the above epimorphism, where the isomorphism comes from the Nerve Lemma, and where the rest of the arrows are induced by inclusions at the topological level
\footnote{The first arrow is the composition
	\(
	\hsf\to \hs(\OO^{\e}(\FP))[\omega(\e)\bm{1}]  \to
	\hs(\OO^{\frac{\delta}{2}}(\FP))\left[\omega(\frac{\delta}{2})\bm{1}\right]\to
	\hs(\OO^{\frac{\delta}{2}}(\FP))[\omega(2\delta)\bm{1}]
	\)
	induced by inclusions of topological spaces,
	where the inclusion
	$\FF_{\bmx} \subseteq \OO^{\e}(\FF_{\bmx+\omega(\e)\bm{1}} \cap
		P)$
	holds for all $\bmx\in\mathbb{R}^n$ according to
	the proof of
	Theorem~\ref{thm:estimator_fixed_radius}~\ref{enum:thm:fixed_radius_offset}, and where the inclusion
	$
		\OO^{\e}(\FF_{\bmx+\omega(\e)\bm{1}} \cap P)\subseteq
		\OO^{\frac{\delta}{2}}(\FF_{\bmx+\omega(\frac{\delta}{2})\bm{1}}
		\cap P)$
	holds given that we assumed
	$\frac{\delta}{2}\geq \varepsilon$.
}:
\begin{equation}
	\hsf
	\to \hs(\OO^{\frac{\delta}{2}}(\FP))[\omega(2\delta)\bm{1}]\xrightarrow{\cong} \hs(\CC^{\frac{\delta}{2}}(\FP))[\omega(2\delta)\bm{1}]\to \hs(\RR^{\delta}(\FP))[\omega(2\delta)\bm{1}]\twoheadrightarrow M[\omega(2\delta)\bm{1}],
\end{equation}

Similarly, we define $\gamma^{\delta}:M\to \hsf[\omega(2\delta)\bm{1}]$ as the following composition:
\begin{equation}
M \hookrightarrow \hs(\RR^{2\delta}(\FP))\to \hs(\CC^{2\delta}(\FP))\xrightarrow{\cong}\hs(\OddFP)\to \hsf[\omega(2\delta)\bm{1}].
\end{equation}

Now, proving~\eqref{equ:interleaving} boils down to showing that, for all ${\bmx}\in\RRR^n$, we have:
\begin{equation}\label{equ:interleaving_t}
        \kappa^{\delta}_{{\bmx} }  \circ \gamma^{\delta}_{{\bmx} -\omega(2\delta)\bm{1}} = M_{{\bmx}-\omega(2\delta)\bm{1},{\bmx}+\omega(2\delta)\bm{1}}
        \quad\quad\quad\text{and}\quad\quad\quad
    \gamma^{\delta}_{{\bmx}+\omega(2\delta)\bm{1}}  \circ \kappa^{\delta}_{{\bmx}} 
		= 
		\hsf_{{\bmx},{\bmx}+2\omega(2\delta)\bm{1}}.
\end{equation}

To prove the left-hand equality in~\eqref{equ:interleaving_t} we rely on the following diagram, where the isomorphisms come from the Nerve Lemma, where $\alpha, v, \alpha', w$ come from the factorization of $\hs(\RR^{\delta}(\FP))\to \hs(\RR^{2\delta}(\FP))$ through its image $M$, where $\beta$ is a section of~$\alpha$ (so $\alpha\circ\beta=\id_{M_{\bmx-\omega(2\delta)\bm{1}}}$) and $\gamma$ is a retraction of $w$ (so $\gamma\circ w=\id_{M_{\bmx+\omega(2\delta)\bm{1}}}$) in the category of vector spaces, and where all the other arrows are induced by inclusions at the topological level (colors are used to distinguish the filtrations involved: red for \v Cech, blue for offset, and black for the others):
\[
\begin{tikzcd}[font=\small]
 \hs(\mathcal{R}^{\delta}(\FP))_{{\bmx}-\omega(2\delta)\bm{1}}\arrow[d,"h"]\arrow[ddd,bend right=2.5cm, shift right=0.1cm, "s"'] \arrow[r, two heads, "\alpha"'] & M_{{\bmx}-\omega(2\delta)\bm{1}} \arrow[ld, hook',"v"] \arrow[rr, "m"] \arrow[l, shift right=5pt, hook', "\beta"']                                      &                                                                                                                                              & M_{{\bmx}+\omega(2\delta)\bm{1}} \arrow[ld, hook', shift right=1pt, "w"']                                                                     \\
 \hs(\mathcal{R}^{2\delta}(\FP))_{{\bmx}-\omega(2\delta)\bm{1}} \arrow[rd, "b"] \arrow[rr, "r"]                                                                                                                                                                    &                                               & \hs(\mathcal{R}^{2\delta}(\FP))_{{\bmx}+\omega(2\delta)\bm{1}}  \arrow[ru, shift right=4pt, two heads, "\gamma"']                                & \hs(\mathcal{R}^{\delta}(\FP))_{{\bmx}+\omega(2\delta)\bm{1}} \arrow[ldd, "s'"'] \arrow[l, "h'"] \arrow[u, two heads, "\alpha'"']                           \\
 & \textcolor{red}{\hs(\mathcal{C}^{2\delta}(\FP))_{{\bmx}-\omega(2\delta)\bm{1}}} \arrow[dd, "\cong"' near start, "c" near start, red] &                                                                                                                                              & \textcolor{red}{\hs(\mathcal{C}^{\frac{\delta}{2}}(\FP))_{{\bmx}+\omega(2\delta)\bm{1}}}\arrow[dd, "\cong"', "\sigma",red] \arrow[u, "g"'] \arrow[ld, "p"'] \\
\textcolor{red}{\hs(\mathcal{C}^{\delta}(\FP))_{{\bmx}-\omega(2\delta)\bm{1}}} \arrow[uu, "a"'] \arrow[rr, "\mu"{near start}, crossing over] \arrow[dd, "\cong"', "k",red] \arrow[ru, "q"] &                                                                                                                                                         & \textcolor{red}{\hs(\mathcal{C}^{\delta}(\FP))_{{\bmx}+\omega(2\delta)\bm{1}}} 
\arrow[uu, "i",shift left=0.9cm] &                                                                                                                         \\                                                                                                                                                          & \textcolor{blue}{\hs(\OO^{2\delta}(\FP))_{{\bmx}-\omega(2\delta)\bm{1}}} \arrow[r, "d" near start]                        & \hsf_{\bmx} \arrow[r, "e"]  & \textcolor{blue}{\hs(\OO^{\frac{\delta}{2}}(\FP))_{{\bmx}+\omega(2\delta)\bm{1}}}  \arrow[ld, "n"']  \\
 \textcolor{blue}{\hs(\OO^{\delta}(\FP))_{{\bmx}-\omega(2\delta)\bm{1}}} \arrow[ru, "l"] \arrow[rr, "\nu"]                                                        &                                               & \textcolor{blue}{\hs(\OO^{\delta}(\FP))_{{\bmx}+\omega(2\delta)\bm{1}}}                       
 \\
 \arrow[from=4-3, to=6-3, "\cong"' near start, "t" near start,red, shift right=1cm, crossing over]          
\end{tikzcd}
\]

Some paths in this diagram are equivalent, due either to the commutativity of inclusion maps at the topological level, or to Lemma~\ref{lem:nerve_theorem_commute}, or to the factorization of $\hs(\RR^{\delta}(\FP))\to \hs(\RR^{2\delta}(\FP))$ through its image $M$, or to the definition of section or retraction,
or finally to the fact that $m:= M_{\bmx-\omega(2\delta)\bm{1},\, \bmx+\omega(2\delta)\bm{1}}$ is the restriction of $r\colon \hs(\RR^{2\delta}(\FP))_{\bmx-\omega(2\delta)\bm{1}}\to \hs(\RR^{2\delta}(\FP))_{\bmx+\omega(2\delta)\bm{1}}$ to $M_{\bmx-\omega(2\delta)\bm{1}}$.
This  implies the left-hand equality in~\eqref{equ:interleaving_t} via the following sequence of elementary steps:
%
\begin{align*}
    &\kappa^\delta_{\bmx}\circ \gamma^\delta_{\bmx-\omega(2\delta)\bm{1}}= (\id_{M_{\bmx+\omega(2\delta)\bm{1}}}) \circ \alpha'\circ g \circ \sigma^{-1} \circ e \circ d \circ c \circ b\circ v \circ \id_{M_{\bmx-\omega(2\delta)\bm{1}}} \\
    &= \ \gamma\circ (w\circ \alpha') \circ g \circ \sigma^{-1} \circ e \circ d \circ c \circ b\circ v \circ \id_{M_{\bmx-\omega(2\delta)\bm{1}}}\\
    &=\gamma\circ (h'\circ g) \circ \sigma^{-1} \circ e \circ d \circ c \circ b\circ v \circ \id_{M_{\bmx-\omega(2\delta)\bm{1}}} =\gamma\circ i \circ (p\circ \sigma^{-1}) \circ  e \circ d \circ c \circ b\circ v \circ \id_{M_{\bmx-\omega(2\delta)\bm{1}}}\\
    &=\ \gamma \circ i \circ t^{-1} \circ n \circ  e \circ d \circ c \circ b\circ v \circ (\id_{M_{\bmx-\omega(2\delta)\bm{1}}}) = \gamma \circ i \circ t^{-1} \circ n \circ  e \circ d \circ c \circ b\circ (v \circ \alpha) \circ \beta\\
    &=\ \gamma \circ i \circ t^{-1} \circ n \circ  e \circ d \circ c \circ b\circ (h) \circ \beta = \gamma \circ i \circ t^{-1} \circ n \circ  e \circ d \circ c \circ (b\circ a) \circ s \circ \beta\\
    &=\ \gamma \circ i \circ t^{-1} \circ n \circ  e \circ d \circ (c \circ q) \circ s \circ \beta = \gamma \circ i \circ t^{-1} \circ (n \circ  e \circ d \circ l) \circ k \circ s \circ \beta  \\
    &= \gamma \circ i \circ (t^{-1} \circ \nu \circ k) \circ s \circ \beta = \gamma \circ (i \circ \mu \circ s) \circ \beta = \gamma \circ r \circ (h) \circ \beta = \gamma \circ r \circ v \circ (\alpha \circ \beta) \\ 
    &= \gamma \circ (r \circ v) \ =\ (\gamma\circ w) \circ m=m.
\end{align*}

To prove the right-hand equality in~\eqref{equ:interleaving_t} we rely on the following diagram, where the isomorphisms once again come from the Nerve Lemma, where $\alpha', w$ come from the factorization of $\hs(\RR^{\delta}(\FP))\to \hs(\RR^{2\delta}(\FP))$ through its image $M$, and where all the other arrows are induced by inclusions at the topological level (colors are used to distinguish the filtrations involved: red for \v Cech, blue for offset, and black for the others): 
\begin{equation}\label{equ:commute_diagram_2}
	\begin{tikzcd}
		\hs(\mathcal{R}^{\delta}(\FP))_{{\bmx}+\omega(2\delta)\bm{1}} \arrow[rd, "s'"]\arrow[rr,"h'",bend left=10pt] \arrow[r, two heads,"\alpha'"]                                          &  M_{\bmx+\omega(2\delta)\bm{1}} \arrow[r,hook,"w"] & \hs(\mathcal{R}^{2\delta}(\FP))_{{\bmx}+\omega(2\delta)\bm{1}} \arrow[d, "b'"]                          \\
		\textcolor{red}{\hs(\mathcal{C}^{\frac{\delta}{2}}(\FP))_{{\bmx}+\omega(2\delta)\bm{1}}}\arrow[d, "\cong"', "\sigma",red]  \arrow[u, "g"'] \arrow[r, "p"]&                                                         \textcolor{red}{\hs(\mathcal{C}^{\delta}(\FP))_{{\bmx}+\omega(2\delta)\bm{1}}}\arrow[r, "q'"]\arrow[ru, "i"] & \textcolor{red}{\hs(\mathcal{C}^{2\delta}(\FP))_{{\bmx}+\omega(2\delta)\bm{1}}} \arrow[d, "c'","\cong"',red]  \\
		\textcolor{blue}{\hs(\OO^{\frac{\delta}{2}}(\FP))_{{\bmx}+\omega(2\delta)\bm{1}}} \arrow[rr, "\pi"]  &                                                                                                                & \textcolor{blue}{\hs(\OO^{2\delta}(\FP))_{{\bmx}+\omega(2\delta)\bm{1}}}                      \arrow[d, "d'"]                                                             \\
		\hsf_{\bmx} \arrow[rr, "\eta"] \arrow[u, "e"']                                                                            &                                                                                                                & \hsf_{{\bmx}+2\omega(2\delta)\bm{1}}
	\end{tikzcd}
\end{equation}

Equivalence of paths in this diagram is due either to the commutativity of inclusion maps at the topological level, or to Lemma~\ref{lem:nerve_theorem_commute}, or to the factorization of $\hs(\RR^{\delta}(\FP))\to \hs(\RR^{2\delta}(\FP))$ through its image $M$. This  implies the right-hand equality in~\eqref{equ:interleaving_t} via the following sequence of elementary steps:
\begin{align*}
\gamma^{\delta}_{\bmx+\omega(2\delta)\bm{1}}\circ \kappa^{\delta}_{\bmx} \ &=\ d' \circ c' \circ b' \circ (w\circ \alpha') \circ g \circ \sigma^{-1} \circ e\ =\ d' \circ c' \circ b' \circ (h') \circ g \circ \sigma^{-1} \circ e\\
&=\ d' \circ c' \circ b' \circ i\circ (s' \circ g) \circ \sigma^{-1} \circ e\ =\ d' \circ c' \circ (b' \circ i)\circ p \circ \sigma^{-1} \circ e \\
&=\ d' \circ (c' \circ q'\circ p \circ \sigma^{-1}) \circ e\ =\ (d' \circ \pi \circ e)=\eta.
\end{align*}
\end{proof}

\begin{remark}\label{rmk:proof_for_noise_function_fixed_radius}
	Assume that each data point $p\in P$ is assigned a function value $\tilde\ff(p)$ that may be different from~$\ff(p)$, and let $\zeta=\max_{p\in P} \left\|\tilde\ff(p)-\ff(p)\right\|_\infty$.
	Then, the respective sublevel filtrations
	$\{\FF_{\bmx}\cap P\}_{\bmx \in\mathbb{R}^n}$ and
	$\{\tilde \FF_{\bmx}\cap P\}_{\bmx\in \rn}$
	of
	$\ffp$ and $\tffp$ are
	$\zeta\bm{1}$-interleaved,
	therefore so are
	$\OO^\delta(\FP) $ and $\OO^\delta(\tFP)$.
	\cref{prop:estimator_noise_fixed_radius}
	follows then by the exact same proof as for
	\cref{thm:estimator_fixed_radius},
	with $\FP$ replaced by $\tFP$, with
	$\OO^\delta(\FF_{\bmx}\cap P)$ replaced by
	$\OO^\delta(\tilde\FF_{\bmx}\cap P)$,
	and with $\omega(\delta)$ and $\omega(2\delta)$
	replaced respectively by
	$\omega(\delta)+\zeta$ and $\omega(2\delta)+\zeta$.
\end{remark}

\begin{remark}\label{rmk:proof_for_noise_distance_fixed_radius}
	Assume that the geodesic distance $d_X(p,q)$ between each pair of data points
	$p,q\in P$ is replaced by some value~$\tilde d_X(p,q)$, such that
	\eqref{equ:mu_graph} holds for some constants $\lambda, \kappa$ with
	$\lambda=1+4\frac{\e}{\kappa}$. Observe then that the proof of
	Theorem~\ref{thm:estimator_fixed_radius}~(iii) only depends on
	Lemma~\ref{lem:nerve_theorem_commute} and on the following sequence
	of inclusions between filtrations for $\delta\geq 2\e$:
	\begin{equation}\label{equ:cech_rips_exact}
		\CC^{\e}(\FP)\hookrightarrow \CC^{\frac{\delta}{2}}(\FP)\hookrightarrow \RR^{\delta}(\FP)\hookrightarrow \CC^{\delta}(\FP)\hookrightarrow \RR^{2\delta}(\FP)\hookrightarrow  \CC^{2\delta}(\FP), 
	\end{equation}
	which, under the assumptions that~\eqref{equ:mu_graph} holds and that $\delta\geq 1+ 2\lambda\frac{\e}{\kappa}$, can be replaced by the following new sequence of inclusions between filtrations:
	\begin{equation}\label{equ:cech_rips_noise}
		\CC^{\e}(\FP)\hookrightarrow \CC^{\frac{\kappa(\delta-1)}{2\lambda}}(\FP)\hookrightarrow \RR^{\delta}_{\tilde d_X}(\FP)\hookrightarrow \CC^{\kappa\delta}(\FP)\hookrightarrow \RR^{1+2\lambda\delta}_{\tilde d_X}(\FP)\hookrightarrow  \CC^{\kappa(1+2\lambda\delta)}(\FP).
	\end{equation}
	Then, after replacing each filtration in the sequence of inclusions~\eqref{equ:cech_rips_exact} by its counterpart in the sequence~\eqref{equ:cech_rips_noise}, Proposition~\ref{prop:estimator_rips_noise_distance_fixed_radius} follows by the exact same proof as Theorem~\ref{thm:estimator_fixed_radius}~\ref{enum:thm:fixed_radius_rips}.
\end{remark}


%% file: Sections/proof_vary.tex
\section{Proof of Theorem~\ref{thm:estimator_varying_radius}}\label{sec:vary_radius}
We begin with a simple observation. Let $\lanftop$ denote the left Kan extension, in the category~$\topc$, of the sublevel filtration $\FF$ along the embedding $\iota\colon \rn \hookrightarrow \rprn$ given by $\bmx\mapsto (0,\bmx)$ for all $\bmx\in\rn$. Since the category  $\rn$ is small and the category $\topc$ is cocomplete, $\lanftop$ is well-defined and given pointwise by the colimit formula:
\[ \forall (\delta, \bmx)\in\rprn, \quad \lanftop_{(\delta, {\bmx})}= \varinjlim \limits_{\substack{{\bmy}\in\rn\\ \iota ({\bmy}) \leq (\delta, {\bmx})}} \FF_{\bmy} \cong
\FF_{\bmx},
\]
and similarly for the structural maps, which end up being mere inclusions. We then have the simple relation:
\begin{equation}\label{equ:Lan-Lan}
\hs(\lanftop) \cong \lanf.
\end{equation}
With this observation in place, we can now proceed with the proof of the theorem.

\begin{proof}[Proof of Theorem~\ref{thm:estimator_varying_radius}~(i)]
	According to the proof of Theorem~\ref{thm:estimator_fixed_radius}~(i), for
	any $\delta\in [\varepsilon, \delta_0]$ there is an ordinary $n$-dimensional
	$\omega(\delta)$-interleaving between the restricted filtrations
	$\lanftop|_{\{\delta\}\times\rn}$ and
	$\OO^{\bullet}(\FP)|_{\{\delta\}\times\rn}$ inside the vertical hyperplane
	$\{\delta\}\times\rn\subset\RRR^{n+1}$.
	Since $\delta\leq\delta_0$, there is
	a fortiori an ordinary $n$-dimensional {$\omega(\delta_0)$-interleaving}
	between $\lanftop|_{\{\delta\}\times\rn}$ and
	$\OO^{\bullet}(\FP)|_{\{\delta\}\times\rn}$. 
	And since the interleaving maps
	are inclusions, they commute with the horizontal structural inclusions in
	$\lanftop|_{[\varepsilon,\delta_0]\times\rn}$ and in
	$\OO^{\bullet}(\FP)|_{[\varepsilon,\delta_0]\times\rn}$, therefore they all
	together form a vertical $\omega(\delta_0)$-interleaving between
	$\lanftop|_{[\varepsilon,\delta_0]\times\rn}$ and
	$\OO^{\bullet}(\FP)|_{[\varepsilon,\delta_0]\times\rn}$.
	The claim follows
	then from~\eqref{equ:Lan-Lan} and the functoriality of~$\hs$.
\end{proof}

\begin{proof}[Proof of Theorem~\ref{thm:estimator_varying_radius}~(ii)]
	Follows from Theorem~\ref{thm:estimator_varying_radius}~(i) combined with the
	isomorphism of persistence modules $\hs(\CC^{\bullet}(\FP))|_{[\e, \delta_0]}
	\stackrel{\cong}{\longrightarrow} \hs(\OO^{\bullet}(\FP))|_{[\e, \delta_0]}$
	induced by Lemma~\ref{lem:nerve_theorem_commute} (which applies because
	$\delta_0<\rhox$).
\end{proof}


\begin{proof}[Proof of Theorem~\ref{thm:estimator_varying_radius} (iii)]
	In the following we use the shorthands $M=\ms|_{[2\e,\delta_0]\times\rn }$ and
	$N=\lanf|_{[2\e,\delta_0]\times\rn}$.
	For any $\delta\in [2\e,\delta_0]$, the
	proof of Theorem~\ref{thm:estimator_fixed_radius} (iii) exhibits a pair of
	morphisms $\gamma^{\delta}, \kappa^{\delta}$  that form an ordinary
	$n$-dimensional $\omega(2\delta)$-interleaving between the restricted modules
	$M|_{\{\delta\}\times\rn}$ and $N|_{\{\delta\}\times\rn}$ inside the vertical
	hyperplane $\{\delta\}\times\rn\subset\RRR^{n+1}$.
	Post-composed by
	structural morphisms of the restricted modules, $\gamma^{\delta}$ and
	$\kappa^{\delta}$ yield a pair of morphisms $\bar\gamma^{\delta},
		\bar\kappa^{\delta}$  that form an ordinary $n$-dimensional
	$\omega(2\delta_0)$-interleaving between $M|_{\{\delta\}\times\rn}$ and
	$N|_{\{\delta\}\times\rn}$.
	For the precise definitions of
	$\bar\gamma^{\delta}$ and $\bar\kappa^{\delta}$, see
	diagrams~\eqref{equ:commute_gamma} and~\eqref{equ:commute_kappa} below.
	Proving that, all together, the morphisms $\bar\gamma^{\delta},
		\bar\kappa^{\delta}$ for $\delta\in [2\e, \delta_0]$ form a vertical
	$\omega(2\delta_0)$-interleaving between $M$ and $N$ boils down to showing
	that they commute with the horizontal structural morphisms of $M$ and $N$,
	which, in turn, reduces to showing that the following squares commute for all
	$\delta\leq\delta'\in [2\e,\delta_0]$ and $\bmx\in\rn$:

	\begin{equation}\label{equ:commut_kappa_gamma}
		\begin{tikzcd}
			M_{(\delta,\bmx)} \arrow[d, "\bar\gamma^{\delta}_{\bmx}"']  \arrow[r] & M_{(\delta',\bmx)} \arrow[d, "\bar\gamma^{\delta'}_{\bmx}"] \\
			N_{(\delta,\bmx+\omega(2\delta_0)\bm{1})}  \arrow[r]  &  N_{(\delta',\bmx+\omega(2\delta_0)\bm{1})}
		\end{tikzcd}
		\quad\quad\quad
		\begin{tikzcd}
			N_{(\delta,\bmx)} \arrow[d, "\bar\kappa^{\delta}_{\bmx}"']  \arrow[r] & N_{(\delta',\bmx)} \arrow[d, "\bar\kappa^{\delta'}_{\bmx}"] \\
			M_{(\delta,\bmx+\omega(2\delta_0)\bm{1})}  \arrow[r]  &  M_{(\delta',\bmx+\omega(2\delta_0)\bm{1})}
		\end{tikzcd}
	\end{equation}
	Unfolding the definitions of $\bar\gamma^\delta_{\bmx}$,
	$\bar\gamma^{\delta'}_{\bmx}$, $\bar\kappa^\delta_{\bmx}$ and
	$\bar\kappa^{\delta'}_{\bmx}$
	in these squares yields the following commutative diagrams, where the equalities come from the pointwise colimit formula for~$\lanf$ (\Cref{eq:Lanf_colim}), where the
	isomorphisms come from the Nerve lemma, where the injections and surjections
	come from the factorization $\hs(\RR^{\bullet}(\FP)) \twoheadrightarrow \ms
		\hookrightarrow \hs(\RR^{2\bullet}(\FP))$, and where the rest of the
	morphisms either come from inclusions at the topological level or are the
	structural morphisms of~$M$ and~$N$:
	\begin{equation}\label{equ:commute_gamma}
		\begin{tikzcd}
			M_{(\delta,\bmx)}\arrow[dddddd,bend right, dotted,"\bar\gamma^{\delta}_{\bmx}"',shift right=13] \arrow[r]\arrow[d,hook] & M_{(\delta',\bmx)} \arrow[d,hook] \arrow[dddddd,bend left, dotted,"\bar\gamma^{\delta'}_{\bmx}",shift left=13] \\
			\hs(\mathcal{R}^{\bullet}(\FP))_{(2\delta,\bmx)} \arrow[d] \arrow[r]  & \hs(\mathcal{R}^{\bullet}(\FP))_{(2\delta',\bmx)}\arrow[d]  \\
			\hs(\mathcal{C}^{\bullet}(\FP))_{(2\delta,\bmx)} \arrow[d, "\cong"'] \arrow[r] & \hs(\mathcal{C}^{\bullet}(\FP))_{(2\delta',\bmx)}\arrow[d, "\cong"] \\
			\hs(\OO^{\bullet}(\FP))_{(2\delta,\bmx)} \arrow[d] \arrow[r] & \hs(\OO^{\bullet}(\FP))_{(2\delta',\bmx)} \arrow[d] \\
			N_{(2\delta,\bmx+\omega(2\delta)\bm{1})} \arrow[r] \arrow[d, equal] & N_{(2\delta', \bmx+\omega(2\delta')\bm{1})} \arrow[d, equal] \\
			N_{(\delta,\bmx+\omega(2\delta)\bm{1})} \arrow[r] \arrow[d] & N_{(\delta', \bmx+\omega(2\delta')\bm{1})} \arrow[d] \\
			N_{(\delta, \bmx+\omega(2\delta_0)\bm{1})} \arrow[r] & N_{(\delta', \bmx+\omega(2\delta_0)\bm{1})}
		\end{tikzcd}
	\end{equation}
	\begin{equation}\label{equ:commute_kappa}
		\begin{tikzcd}
			N_{(\delta,\bmx)} \arrow[r] \arrow[d, equal] \arrow[dddddd,bend right, dotted,"\bar\kappa^{\delta}_{\bmx}"',shift right=18] & N_{(\delta',\bmx)}\arrow[d, equal] \arrow[dddddd,bend left, dotted,"\bar\kappa^{\delta'}_{\bmx}",shift left=18] \\
			N_{(\frac{\delta}{2},\bmx)} \arrow[r] \arrow[d] & N_{(\frac{\delta'}{2},\bmx)}\arrow[d] \\
			\hs(\OO^{\bullet}(\FP))_{(\frac \delta 2, \, \bmx+\omega(2\delta)\bm{1})} \arrow[r]\arrow[d,"\cong"'] & \hs(\OO^{\bullet}(\FP))_{(\frac {\delta'}2,\,\bmx+\omega(2\delta')\bm{1})}\arrow[d,"\cong"] \\
			\hs(\CC^{\bullet}(\FP))_{(\frac \delta 2,\,\bmx+\omega(2\delta)\bm{1})} \arrow[r] \arrow[d] & \hs(\CC^{\bullet}(\FP))_{(\frac {\delta'}2,\,\bmx+\omega(2\delta')\bm{1})}\arrow[d]\\
			\hs(\RR^{\bullet}(\FP))_{(\delta,\,\bmx+\omega(2\delta)\bm{1})} \arrow[r]\arrow[d, two heads] & \hs(\RR^{\bullet}(\FP))_{(\delta',\,\bmx+\omega(2\delta')\bm{1})}\arrow[d, two heads] \\
			M_{(\delta, \bmx+\omega(2\delta)\bm{1})} \arrow[r] \arrow[r] \arrow[d] & M_{(\delta', \bmx+\omega(2\delta')\bm{1})} \arrow[d] \\
			M_{(\delta, \bmx+\omega(2\delta_0)\bm{1})} \arrow[r] \arrow[r] & M_{(\delta', \bmx+\omega(2\delta_0)\bm{1})}
		\end{tikzcd}
	\end{equation}
	The commutativity of these two diagrams implies the commutativity of the squares in~\eqref{equ:commut_kappa_gamma}, hence the result.
\end{proof}

\begin{remark}\label{rmk:proof_for_noise_varying_radius}
	Since the proofs of Propositions~\ref{prop:estimator_noise_fixed_radius}
	and~\ref{prop:estimator_rips_noise_distance_fixed_radius} are literally the
	same as that of Theorem~\ref{thm:estimator_fixed_radius} (up to a change of
	parameters in the filtrations, see
	Remarks~\ref{rmk:proof_for_noise_function_fixed_radius}
	and~\ref{rmk:proof_for_noise_distance_fixed_radius}),
	Propositions~\ref{prop:estimator_noise_varying_radius}
	and~\ref{prop:estimator_rips_noise_distance_varying_radius} derive from them
	in exactly the same way as Theorem~\ref{thm:estimator_varying_radius} derives
	from Theorem~\ref{thm:estimator_fixed_radius}.
\end{remark}

%% file: Sections/stats.tex
\section{Statistical framework and proof of Theorem~\ref{thm:stat_results_overview}}\label{sec:stat}
In this section, we investigate the statistical performance of the estimators
introduced in \cref{thm:estimator_fixed_radius}.
We show that the {scale parameter~$\delta$ for which} our estimators converge can be estimated,
{leading to} theoretical quasi-minimax convergence rates
(\cref{prop:cv_rate_known,prop:minimax_rate}), even when the regularity of the
sampling measure is not known (\cref{prop:cv_rate_unknown_ab}).
Our analysis is inspired from the one in~\cite{carriere2018statistical}, which we adapt to our setting.

	Throughout the section we use the following setup.
Let $(X,d_X)$ be a metric space, $\ff\colon X\to \mathbb{R}^n$ a
function, and $\mu\in \setprobmeas X$ a probability measure supported
on~$X$, i.e., $\supp(\mu)\subseteq X$.
We consider an i.i.d.\ $k$-sample $X_k=(Z_1,\dots,Z_k)$ drawn from the measure~$\mu$, and our goal is to estimate $\hs(\ff)$.
Note that we write $X_k$ instead of $P$ as in the previous sections, to emphasize that the point cloud consists of $k$ independent draws.
In our statements we will make the following assumptions, or some subset thereof:
\begin{assumptions}
	\item \label{assumption:compact_support} $X$ is a compact geodesic space {with convexity radius~$\rhox>0$};
	\item \label{assumption:admit_modcont} The function $\ff$ admits a
	modulus of continuity $\omega\colon \mathbb{R}_{\ge 0}\to
		\mathbb{R}_{\ge 0}$;
	\item \label{assumption:ab_std} The sampling measure $\mu$ is an $(a,b)$-standard probability measure whose support is the entire space~$X$;
	\item \label{assumption:mod_cont_reg} The modulus of continuity
	$\omega$ has a controlled vanishing rate:
	$\delta = O(\omega(\delta))$ as $\delta\to 0$;
	\item \label{assumption:manifold_support}
	$X$ is a
	compact smooth manifold, and $\mu$ decomposes as
	${\mu=\mu_1+\mu_2}$, where $\mu_2(X)>0$ and $\mu_2$ is
	absolutely continuous w.r.t. the uniform measure on
	$X$, with positive density on $X$.
\end{assumptions}

Let us comment on these assumptions.
\cref{assumption:compact_support,assumption:admit_modcont} are the ones appearing in \cref{thm:estimator_fixed_radius}. 
{\cref{assumption:ab_std} ensures that, with high
probability, the sample~$X_k$ covers the space $X$, which is
paramount for correctly approximating $\hsf$. The property is
guaranteed by the  following result, adapted from {\cite[Theorem~3]{cuevas2004boundary}}.}
\begin{lemma}[{\cite[Theorem 2]{chazal2014convergence}}]\label{lemma:cuevas}
	Under \cref{assumption:ab_std}, for any $\eta>0$,
	\begin{equation}\label{eq:opt_delta}
		\mathbb{P}\left( \dH(X_k, X)> 2 \eta \right) \le  \frac{2^b}{a \eta^b}
		e^{-ka\eta^b},
	\end{equation}
	where $\dH$ denotes the Hausdorff distance in $(X,d_X)$.
\end{lemma}

The gist of \cref{assumption:mod_cont_reg} is that, if $\omega(\delta)$
becomes too small compared to~$\delta$ as $\delta\to 0$, then
$\omega$-continuous
functions become constant, and the error induced by our estimator, even though very
small, can no longer be controlled {in terms of the}  modulus of continuity $\omega$.
See \cref{lemma:no_modcont_assumption} for more details.
This assumption encompasses several forms of  regularity, {including Lipschitz and Hölder continuity}.
Note that \cref{assumption:mod_cont_reg} differs from the
assumption in \cite[Section 3.2]{carriere2018statistical},  {which requires}
the map ${\delta\mapsto\frac{\omega(\delta)}{\delta}}$ {to be}
non-increasing.
In fact, when $\omega \neq 0$, this monotonicity condition implies \cref{assumption:mod_cont_reg}  (see \cref{lemma:ours_vs_bertrand_assumption}),
so our assumption is more general.


In the following, we consider two different scenarios:
first, when the
constants~$a,b$ for which $\mu$ is $(a,b)$-standard are known
(Section~\ref{sec:ab_known});
second, when these constants are unknown
(Section~\ref{sec:ab_unknown}).
In the latter scenario, \cref{assumption:manifold_support} is required as well, in order
to control the vanishing rate of the sampling error $\dH(X_k,X)$ as
$k\to\infty$.

\subsection{Known regularity of $\mu$}\label{sec:ab_known}
When $a,b$ are known, one can compute an asymptotically optimal
sequence~$(\delta_k)_k$ of values for parameter $\delta$ depending on the
sample size~$k$:

\begin{equation}\label{eq:delta_k}
	\delta_k := 4 {\left( \frac{2\log(k)}{ak} \right)}^{\frac 1 b}.
\end{equation}

\begin{proposition}\label{prop:cv_rate_known}
	Under
	\cref{assumption:ab_std,assumption:compact_support,assumption:admit_modcont,assumption:mod_cont_reg},
	for $\estgeneric$ chosen among
	the following $X_k$-measurable estimators:
	\begin{equation*}
		\estoffsetd, \quad \estcechd,\quad \textnormal{or }\estripsd,
	\end{equation*}
	we have the following convergence rate {for any fixed threshold~$\thresh>0$}:
	\begin{equation}\label{eq:cv_rate_known}
		\mathbb{E}_{X_k\sim \mu^{\otimes k}} \left[
			\min\left\{\dIo\left(
			{\oracle},\,
			\estgeneric
			\right),\,  D\right\}
			\right]
		\lesssim
		\omega
		\left(
		{ 2 }
		\delta_k
		\right)
		= \omega \left(
		8\left( \frac{2\log(k)}{ak} \right)^{\frac 1 b}
		\right),
	\end{equation}
	where
	the multiplicative constant depends only on $a, b, \rhox, \omega$, and~$\thresh$.
\end{proposition}
%

Note that the estimator $\estgeneric$ is viewed as a random variable taking
values in the space of 
persistence modules, equipped with the
interleaving distance $\dIo$ as an extended pseudo-metric. The modulus of
continuity provides adaptive convergence rates with respect to
the regularity of the target function $\ff$.
These rates are typically of the order $ c\left( \frac{\log(k)}{ak} \right)^{\frac 1 b}$
for $c$-Lipschitz functions,
or
$\left( \frac{\log(k)}{ak} \right)^{\frac \alpha b}$ for $\alpha$-Hölder functions.

Note also that thresholding the interleaving distance at a fixed value~$\thresh$ is necessary for the result to hold, because the interleaving distance may take infinite values on events of small yet positive measure; see \cref{lemma:stats:needs_thresholding} for details.
		{Thresholding $\dIo\left({\oracle},\, \estgeneric\right)$ at~$\thresh$  can be implemented  by restricting the modules $\oracle$ and~$\estgeneric$ to the downset of $(\min \ff_1+2\thresh,\, \cdots,\, \min \ff_{n}+2\thresh)$, where $\ff_1, \dots, \ff_n$ are the components of~$\ff$, and then re-extending the modules to~$\mathbb{R}^n$ by padding with zero vector spaces. This operation has only marginal effect on the modules when \[\thresh>\left\|\left(\max \ff_1-\min \ff_1, \,\cdots,\, \max \ff_n-\min \ff_n\right)\right\|_\infty.\]}

                \begin{proof}[Proof of \cref{prop:cv_rate_known}]
	The proof follows closely the argument of {\cite[Proposition~11]{carriere2018statistical}}.
	Under \cref{assumption:compact_support,assumption:admit_modcont}, \cref{thm:estimator_fixed_radius} can be applied on the event:
	\begin{equation}
		A_k := \left\{ 2\dH(X_k,X) \le \delta_k \le \rhox/2 \right\}.
	\end{equation}
	For sufficiently large sample size $k$, we have $\delta_k\le \min \left\{  \rhox/2,1 \right\}$ deterministically.
	Therefore, in the following we assume that $k$ is large enough and then $A_k$ becomes the probabilistic event $\{2d_{\textnormal{H}}(X_k,X)\le \delta_k \}$.
	Following the proof of {\cite[Proposition 11]{carriere2018statistical}},
	under \cref{assumption:ab_std},
	this choice of $\delta_k$ with \cref{lemma:cuevas}
	ensures that the following upper bound is valid:
	\begin{equation}\label{eq:approx_cuevas_ineq}
		\Pbb (\statcompl{A_k})
		= \Pbb\left(d_{\textnormal{H}}(X_k,X) > \frac{\delta_k}{2}\right)
		\le
		\frac {2^b}{a \left( \frac {2\log(k)}{ak} \right)}
		e^{-ka \left( \frac{2\log(k)}{ak} \right)}
		\leq \frac{2^b}{2k\log(k)}.
	\end{equation}

	On the event $A_k$, \cref{thm:estimator_fixed_radius} applies. On the complementary event $\statcompl{A_k}$, the  approximation error
	$\dIo\left(\oracle,\, \estgeneric\right)$ may be infinite (\cref{lemma:stats:needs_thresholding}), but the thresholded error is bounded above by~$\thresh$.
	This leads to the following inequality:
	\begin{equation}\label{eq:stats:err_appr_bound}
		\min\left\{\dIo\left(
		{\oracle},\,
		\estgeneric
		\right),\,  D\right\}
		\lesssim
		\ind{\statcompl{A_k}}\cdot \thresh + \ind{A_k}\cdot \omega( { 2 } \delta_k)
		\lesssim
		\ind{\statcompl{A_k}}+  \omega( 2\delta_k).
	\end{equation}
	Taking the expectation,
	\Cref{eq:approx_cuevas_ineq,eq:stats:err_appr_bound} yield
	\begin{equation}
		\mathbb{E}_{X_k\sim \mu^{\otimes k}} \left[
			\min\left\{\dIo\left(
			{\oracle},\,
			\estgeneric
			\right),\,  D\right\}\right]
		\lesssim \frac{2^b}{2k\log(k)}
		+ \omega(2\delta_k).
	\end{equation}
	Finally, since the first term is bounded by $(\delta_k)^b$,
	it is at most
	$\delta_k$
	(as $\delta_k\le 1$ and $b\ge 1$ by \cref{lemma:b_ge_1})
	up to a multiplicative constant.
	Thus, we have:
	\begin{equation}\label{eq:stats:known_f_and_reg}
		\mathbb{E}_{X_k\sim \mu^{\otimes k}} \left[
			\min\left\{\dIo\left(
			{\oracle},\,
			\estgeneric
			\right),\,  D\right\}
			\right]
		\lesssim \delta_k + \omega( 2\delta_k).
	\end{equation}
	We conclude using \cref{assumption:mod_cont_reg}.
\end{proof}



\subsection{Unknown regularity of $\mu$}
\label{sec:ab_unknown}
When the constants $a$ and $b$ are not known,
\cref{prop:cv_rate_known} still applies but the values $\delta_k$ are not available a priori.
We address this issue by introducing an estimator $(\hat\delta_k)_k$ of the
sequence $(\delta_k)_k$ in \cref{eq:stats:hatdeltak}, together with the corresponding plug-in estimators in \cref{eq:stats:def_estimators_plugin}.
In \cref{prop:cv_rate_unknown_ab}, we show that the plug-in estimators preserve the convergence rate, up to an additional logarithmic factor.

Our estimator of $(\delta_k)_k$ is based on the {subsampling} strategy of \cite{carriere2018statistical}.
Fix $\beta > 0 $, and define the sequence
$s= (s_k)_{k\in \mathbb{N}}$ by
$s_k = \left\lceil\frac { k}{\left( \log (k) \right)^{1+\beta}}\right\rceil$. We then define the sequence of random variables
$(\hat \delta_k)_{k\in \mathbb{N}}$ by
\begin{equation}\label{eq:stats:hatdeltak}
	\hat\delta_k :=
	d _{\textnormal{H}}(X_{s_k},X_k),
\end{equation}
where $X_{s_k}\subseteq X_k$ consists of the first $s_k$ sample points of~$X_k$.

\begin{proposition}\label{prop:cv_rate_unknown_ab}
	Under	\cref{assumption:admit_modcont,assumption:mod_cont_reg,assumption:ab_std,assumption:compact_support,assumption:manifold_support},
	for \estgeneric{} chosen among the following {$X_k$-measurable} estimators:
	\begin{equation}\label{eq:stats:def_estimators_plugin}
		\estoffsetdh, \quad \estcechdh,
		\quad
		\textnormal{or}
		\quad
		\estripsdh,
	\end{equation}
	we have the following convergence rate
	for any
		{fixed threshold~$\thresh>0$}:
	\begin{equation*}
		\mathbb E_{{X_k\sim \mu^{\otimes k}}} \left[ \min\left\{\dIo\left({\oracle},\, \estgeneric\right),\, {\thresh}\right\} \right]
		\lesssim
		\mathbb{E}_{{X_k\sim \mu^{\otimes k}}}\left(\omega(2\hat\delta_k)\right)
		\lesssim  \omega \left[  4\left( \frac{(\log(k))^{2+\beta}}{ak} \right)^{\frac 1 b}\right],
	\end{equation*}
	where the multiplicative constant depends only on
	$a, b, \rhox, \omega$, {and $\thresh$}.
\end{proposition}

\begin{proof}
	The proof is similar to that of  \cref{prop:cv_rate_known} but involves additional technical details. Consider the event on which \cref{thm:estimator_fixed_radius} applies:
	\begin{equation}
		A_k := \left\{ 2d_{\textnormal{H}}(X_k,X) \le \hat\delta_k \le \rhox/2 \right\}.
	\end{equation}
	Following the proof of \cref{prop:cv_rate_known}, we obtain the upper bound
	\begin{equation}\label{eq:di_bound}
		\min\left\{ \dIo\left({\oracle},  \estgeneric \right),D \right\}
		\lesssim
		\ind{\statcompl{A_k}}\cdot {\thresh} + \ind{A_k}\cdot  \omega(2\hat\delta_k)
		\lesssim
		\ind{\statcompl{A_k}}+  \omega(2\hat\delta_k)
		.
	\end{equation}
	It remains to show that,
	in expectation,
	{$\omega(2\hat\delta_k)$} is the leading term; that is
	$\Pbb(\statcompl{A_k}) \lesssim \mathbb{E}\left( \omega(2\hat\delta_k) \right)$. {By the union bound, we have}
	\begin{equation}\label{eq:stats:ak_to_control}
		\mathbb{P}\left( \statcompl{A_k} \right)
		\le \Pbb \left( d_{\textnormal{H}}(X_{k},X) > \frac
		{\hat\delta_{k}}{2} \right)
		+
		\Pbb \left( \hat\delta_k > \frac \rhox 2 \right).
	\end{equation}
	We begin with the second term on the right-hand side of \cref{eq:stats:ak_to_control}, for which a rough upper bound is obtained
	as follows. The first inequality uses the fact that
		{$X_{s_k}\subseteq X_k\subseteq X$, hence $\dH(X_{s_k}, X)\geq
				\dH(X_{s_k}, X_k)$,}
		{while the second inequality follows from \cref{lemma:cuevas}:}
	\begin{equation}\label{eq:stats:ak_to_control1}
		\mathbb P\left(\hat\delta_k > \frac{\rhox} {2}\right)
		\le \mathbb P\left(d_{\textnormal{H}}(X_{s_k},X)> \frac{\rhox}2\right)
		{\le  \frac {2^b}{a\left(\frac {\rhox} 4\right)^{b} }
				e^{-a(\frac{\rhox}4)^b \left\lceil\frac{k}{\left(\log (k)\right)^{1+\beta}}\right\rceil}}
		\lesssim e^{-\log(k)} = \frac 1 {k}.
	\end{equation}
	The first term on the right-hand side of \cref{eq:stats:ak_to_control} corresponds to the (B) term in
	the proof of \cite[Proposition~13]{carriere2018statistical}. For sufficiently large $k$, it can be bounded in expectation using a packing argument under \cref{assumption:manifold_support} (see \cref{lemma:stats:packing_bootstrap}):
	\begin{equation}\label{eq:stats:ak_to_control2}
		\mathbb{P} \left( d_{\textnormal{H}}(X_{k},X)
		>\frac
		{\hat\delta_{k}}{2} \right)
		\le
		\frac{2^b}{2k\log(k)}.
	\end{equation}
	Therefore, combining \cref{eq:stats:ak_to_control,eq:stats:ak_to_control1,eq:stats:ak_to_control2}, we have
	\begin{equation}\label{eq:stats:ak_to_control_overall}
		\mathbb{P}\left( \statcompl{A_k} \right)\lesssim \frac{1}{k}+\frac{2^b}{2k\log(k)} \lesssim  \frac 1 k.
	\end{equation}
	Invoking \cref{assumption:manifold_support} once more, we obtain the quasi-minimax bound on $\hat\delta_k$ (see \cref{lemma:minimax_distance_sample}):
	\begin{equation}\label{eq:lower_bound_hat_delta}
		\Ebb \left( \hat\delta_k \right) \gtrsim \frac 1 {k^{\frac 1 b}} \ge \frac 1 k \gtrsim \mathbb P(A^c_k)
		\quad \textnormal{and}\quad
		\Ebb \left( \omega(2\hat\delta_k) \right)\lesssim\omega \left[  4\left(
			\frac{2\log^{2+\beta}(k)}{ak} \right)^{\frac 1 b}\right].
	\end{equation}
	We conclude with \cref{assumption:mod_cont_reg}.
\end{proof}

\subsection{Minimum rate}
\cref{prop:cv_rate_known,prop:cv_rate_unknown_ab} provide
consistent estimators for $\hsf$ based on a $k$-sample $X_k$ of $\mu$.
A natural question is then: does there exist
another $X_k$-based estimator
$\estgeneric$ achieving better rates?
The answer is negative as per the following proposition.

\begin{proposition}[Minimum rate]\label{prop:minimax_rate}
	For any fixed threshold $D\ge 1$ and
	homology degree $i<b$,
	we have the following lower bound:
	\begin{equation}\label{equ:mini_rate}
		\sup_{\substack{n\geq 1 \\[0.5ex] \textnormal{$X$ satisfying~\ref{assumption:compact_support}}\\[0.5ex]   \textnormal{$\ff\colon X\to \rn$ satisfying \ref{assumption:admit_modcont}}}}
		\,
		\inf_{\widehat{\Homology_i(\ff)}}
		\,
		\sup_{\substack{\mu \in \abstd X}}
		\mathbb{E}_{X_k\sim \mu^{\otimes k}}
		\left[ \min\left\{\dIo\left({\Homology_i(\restr \ff {\mathrm{supp}(\mu)})} ,\, {\widehat{\Homology_i(\ff)}} \right),\, \thresh \right\}\right]  \gtrsim \omega \left( \frac 1 2 \left( \frac 1 {ak} \right)^{ \frac{1}{b}} \right),
	\end{equation}
	where the multiplicative constant depends only on
	$a, b, \rhox, \omega$ {and $\thresh$}.
\end{proposition}
\begin{proof}
	The proof  follows closely the arguments in \cite[Proposition 12]{carriere2018statistical} and \cite[Section B.2]{chazal2014convergence}.
	The key step in establishing the lower bound is Le Cam's lemma
	(\cref{lecams_lemma}), which reduces the problem to constructing
	two $(a,b)$-standard measures such that the target function on their support induces persistence modules as different as possible, while
	the total variation between these two measures remains of order $\frac 1 k$, for a sample size $k$.

	Fix a sample size $k\in \mathbb{N}_{> 0}$. For the outmost $\sup$ in~\cref{equ:mini_rate}, we consider the special case where $n=1$, $X=([0,1]^b, \left\lVert \cdot \right\rVert_\infty)$ and function $\ff:X\to \RRR$ is defined as
	\begin{equation}\label{eq:lower_bound_function}
		\ff\colon y\in [0,1]^b \longmapsto
			-\omega\left(\min \left\{d_X\left(y,\left\{ \frac 1 2 x_k \right\}^{ i+1}\times [0,1]^{ b-i-1}\right)
			, \frac 1 2 x_k\right\}\right), \quad
		\textnormal{where }
		x_k := (ak)^{-\frac 1 b}.
		\end{equation}
	In the following we assume that $k$ is large enough, so that $x_k\le 1$ and
	$\omega(x_k)\le D$.
	Note that $\ff(y)$ is the opposite of the (thresholded) distance of $y$ to the convex set
	$\left\{ \frac 1 2 x_k \right\}^{ i+1}\times [0,1]^{ b-i-1}$ in $[0,1]^b$.
	Therefore, $\ff^{-1}((-\infty, t])$
	is empty  when $t<- \omega(\frac 1 2 x_k)$, homotopy equivalent to the  $i$-sphere $\mathbb S^{i}$ if $- \omega(\frac 1 2 x_k) \le t<0$ and contractible
	otherwise. For example, when $b=2$, $i=0$ and $X=[0,1]^2$, the sublevel set $\ff^{-1}((-\infty, t])$ is homotopy equivalent to $\mathbb{S}^0$ for any $t\in [-\omega(\frac{1}{2}x_k),0)$. When $b=3$, $i=1$ and $X=[0,1]^3$, the sublevel set $\ff^{-1}((-\infty, t])$ is homotopy equivalent to $\mathbb{S}^1$ for any $t\in [-\omega(\frac{1}{2}x_k),0)$.

	By the subadditivity and non-increasing property  of the modulus of
	continuity $\omega$, the function $\ff$ is $\omega$-continuous. Then consider the two probability distributions on the metric space $X$:
	\begin{equation*}
		P_0 := \delta_0 \quad
		\textnormal{ and }\quad
		P_{1} := \left( 1 - \frac 1 k  \right) P_0 + \frac 1k
		\mathcal{U}_{[0, x_k]^b},
	\end{equation*}
	where $\delta_0$ is the dirac measure based on $0$, and
	${\mathcal{U}_{[0, x_k]^b}}$
	is the uniform measure on
	$[0, x_k]^b$.

	Trivially, $P_0$ is $(a,b)$-standard.
	For any $x\in [0, x_k]^b$, and positive $r>0$, we have:
	\begin{equation*}
		P_{1} (B_X(x,r)) \ge \frac 1 k \mathcal U_{[0, x_k]^{b}}(B_X(x,r))
		\ge
		\min\left\{
			\frac{1}{k}((ak)^{\frac{1}{b}}r)^b 
			,1\right\} = \min\{ar^b, 1\},
	\end{equation*}
	where the second inequality follows from the fact that the density of  $  \mathcal
		U_{[0,x_k]}$ with respect to the Lebesgue measure is $(x_k)^{-1}= (ak)^{\frac 1 b}$.
	This concludes that  $P_{1}$ is also $(a,b)$-standard.

	We now consider the pseudo-distance $\rho$ between persistence modules,
	defined by:
	\begin{equation*}
		\rho(M,N) : = \min \left\{\dIo \left( M, N \right),D\right\},
	\end{equation*}
	and the map $\theta \colon P\mapsto \Homology_i (\restr \ff {\mathrm{supp}(P)})$ on the $(a,b)$-standard measures $P$.
	Using Le Cam's lemma (\cref{lecams_lemma}), for any fixed estimator $\hat \theta$ based on $X_k$ we get the lower bound:
	\begin{equation*}
		\sup_{P\in \abstd{{[0,1]^b}}}
		\Ebb_{X_k\sim P^{\otimes k}}
		\left[ \rho \left( \hat \theta, \theta(P) \right) \right]
		\ge
		\frac 1 8
		\rho (\theta(P_0), \theta(P_{1})) \left( 1 - \mathrm{TV}\left( P_0,P_{1} \right) \right)^{2k}.
	\end{equation*}
	It remains to control this lower-bound.
	We start with the first multiplicative term,
	which contains two cases.
	In the first case, where $i=0$, we have: 
	\begin{equation*}
		\Homology_0 (\restr \ff {\mathrm{supp}(P_0)}) \simeq
		\kkk^{\left[-\omega( \frac 1 2 x_k ), +\infty\right)}
		\quad \textnormal{and} \quad
		\Homology_0 (\restr \ff {\mathrm{supp}(P_{1})}) \simeq
		\kkk^{\left[-\omega( \frac 1 2 x_k ), +\infty\right)}
		\oplus
		\kkk^{\left[-\omega(\frac 1 2 x_k ),0\right)},
	\end{equation*}
	and in the second case, where $i>0$, we end up with:
	\begin{equation*}
		\Homology_i (\restr \ff {\mathrm{supp}(P_0)}) \simeq
		0
		\quad \textnormal{and} \quad
		\Homology_i (\restr\ff{\mathrm{supp}(P_{1})}) \simeq
		\kkk^{\left[-\omega(\frac 1 2 x_k ),0\right)}.
	\end{equation*}
	Therefore, we have in both cases:
	\begin{equation*}
		\rho(\theta(P_0),\theta(P_1)) =
		\min\left\{ \frac 1 2 \omega \left( \frac 1 2 x_k \right),D \right\}
		=
		\min\left\{ \frac 1 2 \omega \left( \frac 1 {2({ak})^{\frac 1 b}} \right),D \right\} .
	\end{equation*}
	It remains to show that the total variation is of correct order:
	\begin{equation*}
		\mathrm{TV} \left( P_0,P_{1} \right)
		= \sup_{A\subseteq {[0,1]^b} \textnormal{ measurable}}
		\left\lvert P_0(A) - P_{1}(A) \right\rvert
		= \sup_{A\subseteq {[0,1]^b} \textnormal{ measurable}}
		\left\lvert \frac 1 k (\delta_0 - \mathcal{U}_{{[0,x_k]^b}})(A)  \right\rvert
		= \frac 1 k.
	\end{equation*}
	This concludes that:
	\begin{equation*}
		\sup_{P\in \abstd {{[0,1]^b}}}
		\Ebb_{X_k\sim P^{\otimes k}}
		\left[ \rho \left( \hat \theta, \theta(P) \right) \right]
		\ge
		\frac 1 {8}
		\left(
			\min \left\{ \frac 1 2\omega \left( \frac 1 {2(ak)^{\frac 1 b}} \right),D \right\}
			\right)
		\left( 1 - \frac 1 k \right)^{2k}
		\gtrsim \omega \left( \frac 1 {2(ak)^{\frac 1 b}} \right).
	\end{equation*}
	Since $\hat{\theta}$ can be any estimator based on $X_k$, the lower bound in~\cref{equ:mini_rate} holds.
\end{proof}

\begin{remark}[Quasi-minimax rate]\label{rem:quasi-minimax}
	A consequence of \cref{prop:minimax_rate} is that the estimators $\estgeneric$ of $\oracle$ given in \cref{prop:cv_rate_known,prop:cv_rate_unknown_ab} are \emph{consistent and {$\omega$-}quasi-minimax} estimators.
	When $X$ contains an open subset of $\mathbb{R}^d$, this rate is known to be minimax for
	specific homology degrees---see for instance
	\cite[Theorem~5]{chazal2014convergence}.
	To our knowledge, this rate is not known to be minimax on general geodesic
	spaces.
\end{remark}

%


%% file: Sections/comp.tex
\section{Estimators Computation}
\label{sec:estim_comput}

In this section we describe our options for computing the estimator $\ms$, hence also $\msd$ for any fixed $\delta$ by restriction.

\subsection{Filtrations computation}
The estimator comes from the function-Rips filtration $\RR^{\bullet}(\FP)$. This is a multifiltration of the power set  $2^P\setminus\{\emptyset\}$,  whose  size (i.e., total number of simplices) is $2^{|P|}-1$ and can be reduced to $O(| P|^{r+2})$ by constructing only its $(r+1)$-skeleton if homology in degree~$r$ is to be considered. Evidently, this upper bound is practical only for very small values of $r$.


As mentioned in introduction, function-geometric multifiltrations in general have received attention in the particular case where $n=1$, $X$ is a Euclidean space, and $\ff$ is some  density estimator. In this setting, various bifiltrations have been proposed with controllable size and provable equivalence or closeness, in terms of persistent homology, to our filtrations $\OO^{\bullet}(\FP)$, $\CC^{\bullet}(\FP)$, $\RR^{\bullet}(\FP)$ or to some variants thereof---see e.g.~\cite{alonso2024sparse,alonso2024delaunay,buchet2024sparse,lesnick2024sparse}.
It would make sense to investigate possible  extensions of these constructions to  our setup, especially to the approximation of our smoothed version $\hs(\Rinc(\FP))$ of the persistent homology of the function-Rips filtration. This would help drastically reduce the size of the filtrations involved in the construction of our estimators. 
 It would also make sense to consider larger values of $n$, not just $n=1$, and larger classes of functions~$\ff$ beyond density estimators. 

\subsection{Computing a presentation of $\ms$}\label{sec:inv_compute}
Assuming the filtration $\RR^{\bullet}(\FP)$ has been built, we now turn to the computation of a free presentation of $\ms$, from which many invariants—such as multigraded Betti numbers—can be derived. The difficulty stems from the fact that the module is not induced in homology by a single filtration, but by the inclusion between two nested filtrations $\RR^{\bullet}(\FP)\hookrightarrow \RR^{2\bullet}(\FP)$. Specifically, we want a free presentation of the image of a morphism between \fp{} persistence modules induced in homology from this inclusion between filtrations. To the best of our knowledge, no efficient method for obtaining it has been described in the literature, so we provide one here. A simpler variant to obtain a free presentation of the cokernel is described in \cite[Proposition~4.14]{dey2025decomposing}.

We begin by addressing a more general problem: computing a free presentation of the image of any morphism of persistence modules $\tau:M\to N$, and then we apply the results to $\ms$.
Section~\ref{sec:construct_presentation} outlines the mathematical construction of a free presentation of $\im(\tau)$ and then provides Algorithm~\ref{alg:compute_presentation_imfg} to compute such a free presentation. In Section~\ref{sec:compute_presentation_smoothing}, we apply Algorithm~\ref{alg:compute_presentation_imfg} to obtain a free presentation of $\ms$. Finally, Section~\ref{sec:presentation_degree_zero} addresses the special case of homology in degree~0, for which the situation is much simpler.

\subsubsection{From presentations of individual persistence modules to presentations of images of morphisms}\label{sec:construct_presentation}
Let $\tau:M\to N$ be a morphism  between finitely presented persistence modules, with respective free presentations  $p_1:P_1\to P_0$ and $q_1:Q_1\to Q_0$.  Since $P_0$ is projective, we can lift $\tau$ to a map $\gamma:P_0\to Q_0$ so that the following diagram commutes:
\begin{equation}\label{equ:commute_diagram_known}
\begin{tikzcd}
P_1 \arrow[r, "p_1"] & P_0 \arrow[r, "p_0",two heads] \arrow[d, "\gamma"] & M \arrow[r] \arrow[d, "\tau"] & 0 \\
Q_1 \arrow[r, "q_1"] & Q_0 \arrow[r, "q_0",two heads]                     & N \arrow[r]                           & 0.
\end{tikzcd}
\end{equation}
From this diagram, we can derive a free presentation of $\im(\tau)$ as follows.

Consider the pullback $P'_1$ of the subdiagram $P_0\xrightarrow{\gamma} Q_0 \xleftarrow{q_1} Q_1$, 
i.e., the kernel of the morphism $P_0\oplus Q_1 \xrightarrow{\gamma - q_1} Q_0$, where $(\gamma - q_1)(u,v) \coloneq \gamma(u)-{q_1}(v)$ for any $(u,v)\in P_0\oplus Q_1$. 
Let $\pi_0:P'_1\to P_0$ and $\pi_1:P'_1\to Q_1$ denote the canonical projections from $P'_1\subseteq P_0\oplus Q_1$ to $P_0$ and $Q_1$, respectively. Next, let $P'_2$ be the free cover of $P'_1$, yielding the surjection $\beta:P'_2\twoheadrightarrow P'_1$.  We illustrate these morphisms in the following commutative diagram, where $P'_1, P_0, Q_0, Q_1$ form a pullback square:

\begin{equation}\label{equ:diagram_pre_im}
\begin{tikzcd}
P'_2 \arrow[r, "\beta", two heads] & P'_1 \arrow[rrd, "\pi_0", bend left] \arrow[rdd, "\pi_1"', bend right] &                      &                                          &                                                        &   \\
                                   &                                                                        & P_1 \arrow[r, "p_1"] & P_0 \arrow[r, "p_0",two heads] \arrow[d, "\gamma"] & M \arrow[r] \arrow[d, "\tau"] & 0 \\
                                   &                                                                        & Q_1 \arrow[r, "q_1"] & Q_0 \arrow[r, "q_0",two heads]                     & N \arrow[r]                           & 0
\end{tikzcd}
\end{equation}

Then the following proposition yields a free presentation of $\im(\tau)$.
\begin{proposition}\label{prop:math_present_img}
The sequence $P'_2 \xrightarrow{\pi_0\circ\beta} P_0 \xrightarrow{\eta} \im(\tau)\to 0$ is exact, where $\eta$ comes from the factorization of $\tau \circ p_0$ through its image $\im(\tau)$. Consequently, the morphism $\pi_0\circ\beta:P'_2\to P_0$ is a free presentation of $\im(\tau)$, and in fact a finite free presentation..
\end{proposition}
\begin{proof}
First, the morphism $\eta$ is surjective since it is equal to the composition 
$ P_0\xtwoheadrightarrow{p_0} M \twoheadrightarrow \im(\tau),
$.
Then we have to prove $\ker(\eta)=\im(\pi_0\circ\beta)$, which comes from the fact that 
 $\ker(\eta)=\ker(\tau\circ p_0)=\ker(q_o\circ \gamma)=\gamma^{-1}(\ker(q_0))=\gamma^{-1}(\im(q_1))=\im(\pi_0)=\im(\pi_0\circ\beta)$. Therefore, the sequence is exact. By definition, $P_2'$ and $P_0$ are free modules, so the morphism $\pi_0\circ\beta$ is a free presentation of $\im(\tau)$.

 Moreover, since $M$ and $N$ are \fp{}, the modules $P_0$, $Q_0$ and $Q_1$ are of finite rank. Thus, the module $P_0\oplus Q_1$ is also of finite rank. By~\cite[Lemma~3.14]{bjerkevik2021ell}, the kernel $P'_1$ of the morphism $\gamma - q_1: P_0\oplus Q_1 \to Q_0$  is therefore \fp{}, so its free cover $P'_2$ is of finite rank, which concludes the proof that $\pi_0\circ\beta$ is a finite free presentation of $\im(\tau)$.
\end{proof}

\begin{algorithm}[H]
\caption{Computing a free presentation of $\im(\tau)$} 
\label{alg:compute_presentation_imfg}
\begin{algorithmic}[1]
\Statex \textbf{Input:} A $k\times m_1$ multigraded matrix $\Gamma$ that represents the morphism $\gamma:P_0\to Q_0$.  A $k\times m_2$ multigraded matrix $\Upsilon$ that represents the morphism $q_1:Q_1\to Q_0$. $\Gamma$ and $\Upsilon$ should share the same basis of $Q_0$.
\Statex \textbf{Output:} An $m_1\times m$ multigraded matrix $R$ that is a free presentation of $\im(\tau)$.
\State  $D \gets [\Gamma\mid-\Upsilon]$ \Comment{Concatenate $\Gamma$ and $-\Upsilon$ column-wise.}
\State $S\gets \textsf{KerMinGen}(D)$    \Comment{Compute the minimal generators of $\ker(\gamma-q_1)$.}
\State $R\gets S[1:m_1,*]$    \Comment{$R$ consists of the first $m_1$ rows of $S$.}
\State \textbf{return} $R$ 
\end{algorithmic}
\end{algorithm}

\begin{figure}[H]
  \centering
  \begin{minipage}[b]{0.5\textwidth}
    \centering
    \includegraphics[width=\linewidth]{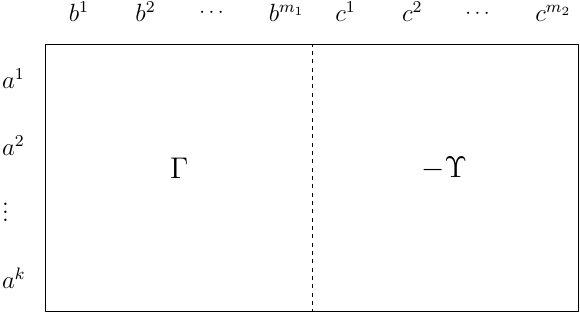}
    \caption*{(a)}
    \label{fig:visual1}
  \end{minipage}
  \hspace{1cm}
  \begin{minipage}[b]{0.3\textwidth}
    \centering
    \includegraphics[width=4cm]{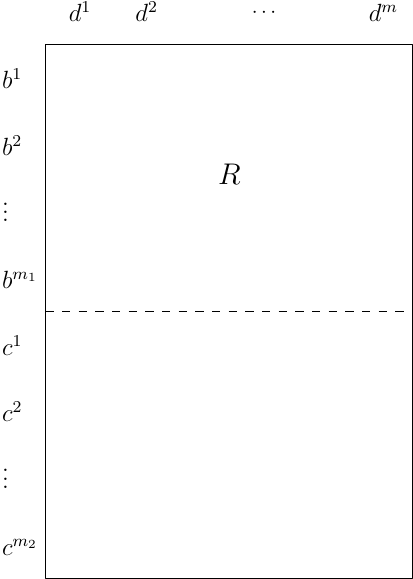}
    \caption*{(b) }
    \label{fig:visual2}
  \end{minipage}
  \caption{Visualization of the input and output of Algorithm~\ref{alg:compute_presentation_imfg}. \textbf{(a)}: The multigraded matrix $D$ that is the concatenation $\Gamma$ and $-\Upsilon$ by columns. \textbf{(b)}: The multigraded matrix $S$, whose first $m_1$ rows form the output multigraded matrix $R$.}
  \label{fig:visualize_alg}
\end{figure}

Proposition~\ref{prop:math_present_img} gives rise to~\cref{alg:compute_presentation_imfg} for computing a free presentation of $\im(\tau)$, assuming that the morphisms $\gamma$ and $q_1$ between free modules are provided as inputs as the form of multigraded matrices $\Gamma$ (for $\gamma$) and $\Upsilon$ (for $q_1$) sharing the same basis of $Q_0$.  The output of Algorithm~\ref{alg:compute_presentation_imfg} is a multigraded matrix $R$ that represents the morphism $\pi_0\circ\beta$, which is a free presentation of $\im(\tau)$.

Recall that the persistence module $P'_1$ is the kernel of the morphism $\gamma-q_1:P_0\oplus Q_1\to Q_0$. To obtain the morphism $\pi_0\circ \beta$, which serves as a free presentation of $\im(\tau)$, we first compute the minimal set of generators of $\ker(\gamma-q_1)$, where each generator is a linear combination of basis elements of $P_0\oplus Q_1$, then we project these minimal generators onto $P_0$, as detailed in Algorithm~\ref{alg:compute_presentation_imfg}.




 We now provide a detailed explanation of Algorithm~\ref{alg:compute_presentation_imfg}. In line 1, as shown in Figure~\ref{fig:visualize_alg}(a), we concatenate the multigraded matrices $\Gamma$ and $-\Upsilon$ column-wise to form a new multigraded matrix $D$, representing the morphism $\gamma-q_1: P_0 \oplus Q_1\to Q_0$. Note that $\Upsilon$ and $-\Upsilon$ share the same row and column grades; $-\Upsilon$ is obtained by multiplying each entry of $\Upsilon$ by $-1$.
In line 2, we compute the minimal generators of $\ker(\gamma-q_1)$ using \textsf{KerMinGen}, a function designed to compute the minimal generators of the kernel of any morphism between free modules (see below). As shown in Figure~\ref{fig:visualize_alg}(b),
 the output of \textsf{KerMinGen} in line 2 can be written as a multigraded matrix $S$ of size $(m_1+m_2)\times m$. Each row of $S$ corresponds to a basis element of either $P_0$ or $Q_1$, with the row grade reflecting the grade of the corresponding basis element. Each column of $S$ corresponds to a generator of $\ker(\gamma-q_1)$, with the column grade indicating the birth time of that generator. The length of each column in $S$ is $m_1+m_2$, coming from the fact that each generator is a linear combination of basis elements from $P_0$ and $Q_1$.
In line 3, we take the first $m_1$ rows of $S$ to project the minimal set of generators of $\ker(\gamma-q_1)$ to $P_0$ and obtain the $m_1\times m$ multigraded matrix $R$ representing the morphism $\pi_0\circ \beta:P'_2 \to P_0$, which is a free presentation of $\im(\tau)$.

The complexity of Algorithm~\ref{alg:compute_presentation_imfg} is determined by that of \textsf{KerMinGen}. In general, Schreyer's algorithm is applicable to compute the generators of the kernel of a morphism between free modules, but its worst-case time complexity is doubly exponential with respect to the number of parameters and it is notably slow on large-scale datasets in practical applications \cite{schreyer1980berechnung,cox2005using,erocal2016refined,la1998strategies}. Fortunately, efficient algorithms  exist for one- and two-parameter persistence modules. In the one-parameter case, $\ker(\gamma-q_1)$ can be computed via Gaussian elimination, with a worst-case time complexity of $O(k\cdot (m_1+m_2)^2)$ \cite[Algorithm~3]{lesnick2022computing}. In the two-parameter case, the specialized algorithm from \cite[Algorithm~7]{lesnick2022computing} runs in $O((m_1+m_2)(m_1+m_2+k)\min(k,m_1+m_2))$ time, and the queue strategy proposed by \cite[Section~4]{kerber2021fast} can be applied to significantly improve both speed and memory efficiency in practice.

\subsubsection{Applying Algorithm~\ref{alg:compute_presentation_imfg} to $\ms$}\label{sec:compute_presentation_smoothing}

Algorithm~\ref{alg:compute_presentation_imfg} can be directly applied to compute a free presentation of $\ms$, given a multigraded matrix $\Upsilon$ of size $k\times m_2$ that represents a free presentation of $\hs(\RR^{2\bullet}(\FP))$. In this setting, a free presentation of $\hs(\RR^{\bullet}(\FP))$ is obtained simply by rescaling the grades in $\upsilon$. Then, $m_1=k$ and the multigraded matrix $\Gamma$ is the identity matrix $\id$ of size $k\times k$, where each row grade is defined as $\gr(\Gamma[i,*]):=\gr(\Upsilon[i,*])$, and each column grade is defined as 

\[\gr(\Gamma[*,i]):= (2,\underbrace{1, 1, \ldots, 1}_{n\text{ times}})\odot \gr(\Gamma[i,*])\in\RRR^{n+1},\] 
for any index $i\in\{1,\ldots, k\}$, where $\odot$ denotes the Hadamard (componentwise) product.
In other words, using the notations from Figure~\ref{fig:visualize_alg}(a), we have 
\[b^i= (2, \underbrace{1, 1, \ldots, 1}_{n\text{ times}}) \odot a^i = (2a^i_1,a^i_2,\ldots,a^i_{n+1})\in\RRR^{n+1},\] 
where $a^i= (a^i_1,a^i_2,\ldots,a^i_{n+1})\in\RRR^{n+1}$.

We implemented this algorithm as part of the \texttt{multipers}~library~\cite{multipers} for the case where $\ff$ is a real-valued function, so that $\ms$ is a two-parameter persistence module. We leverage Kerber and Rolle's queue strategy to accelerate the computation \cite[Section~4]{kerber2021fast}. 

\subsubsection{The case of homology in degree~$0$}\label{sec:presentation_degree_zero}
This setting is very special because of the following isomorphism of persistence modules:
\begin{align}\label{eq:isom_betti0}
\begin{aligned}
    \mszero &\cong \hzero(\RR^{2\bullet}(\FP)).
\end{aligned}
\end{align}
Indeed, as parameter~$\delta$ increases,  no new path-connected component can appear in the Rips complex: path-connected components can only merge. Consequently,  computing free presentations for $\mszero$ can now be done from a single filtration instead of a pair of filtrations.


%% file: Sections/expe.tex
\section{Experiments}\label{sec:experiments}
In this section, we illustrate our approximation and convergence results through practical experiments on both synthetic and real-world datasets. The code and datasets are available at \url{https://github.com/JingyiLi-612/mph_estimation_experiments}.


\begin{figure}[t]
	\centering
	\begin{subfigure}[b]{0.44\textwidth}
		\centering
		\includegraphics[width=\textwidth]{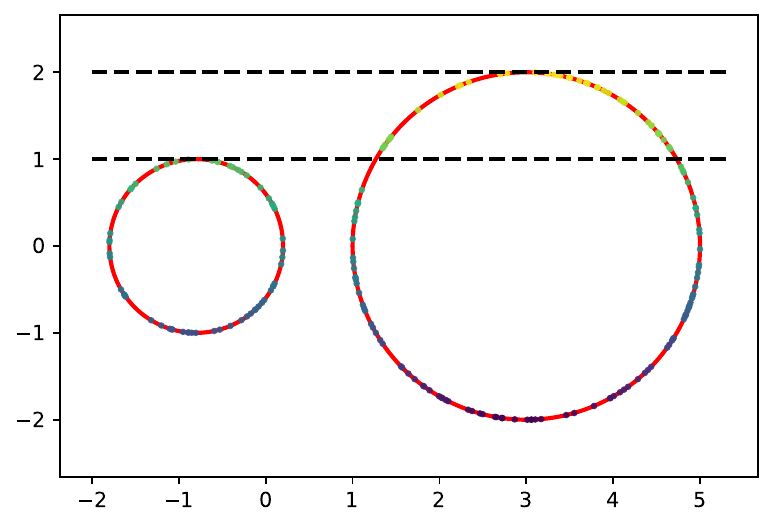}
		\caption{}
	\end{subfigure}
	\begin{subfigure}[b]{0.44\textwidth}
		\centering
		\includegraphics[width=\textwidth]{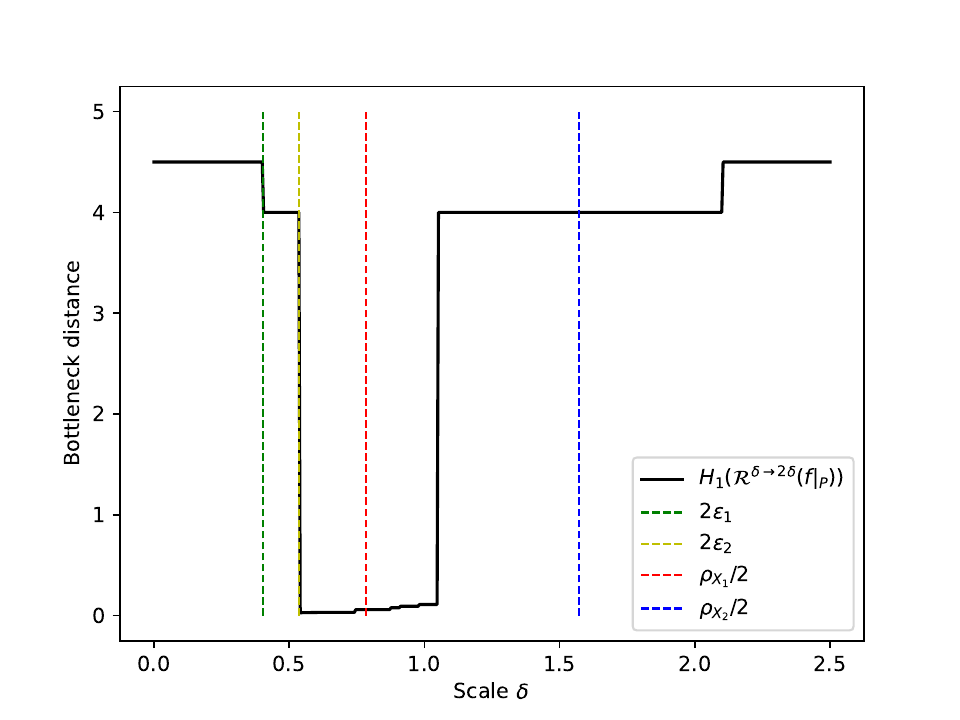}
		\caption{}
	\end{subfigure}
	\vspace{0.5cm}
	\begin{subfigure}[b]{0.44\textwidth}
		\centering
		\includegraphics[width=\textwidth]{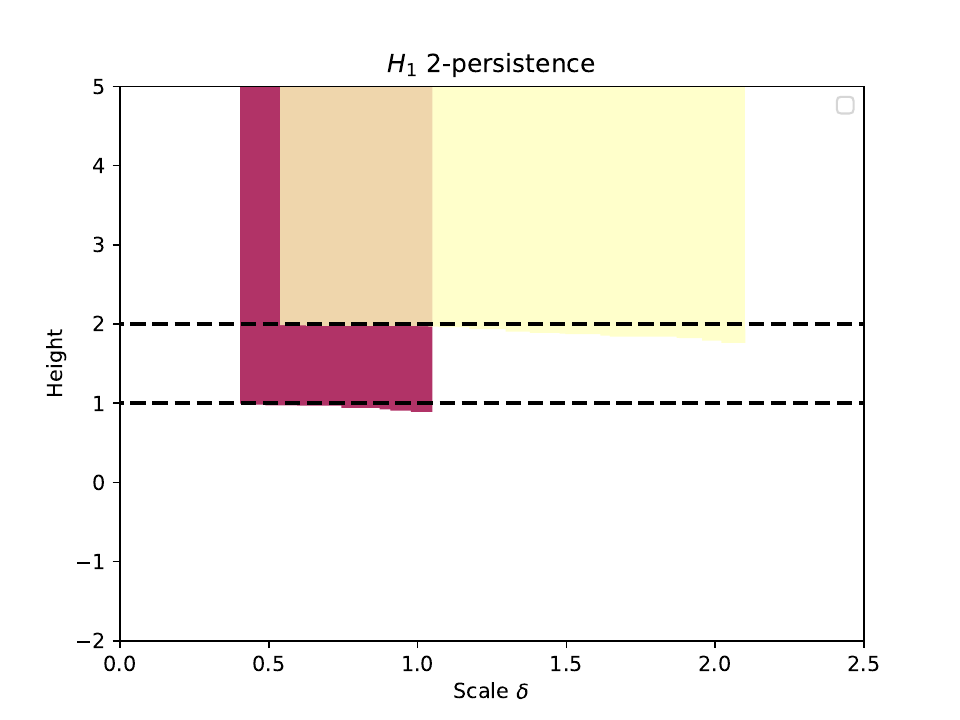}
		\caption{}
	\end{subfigure}
	\begin{subfigure}[b]{0.44\textwidth}
		\centering
		\includegraphics[width=\textwidth]{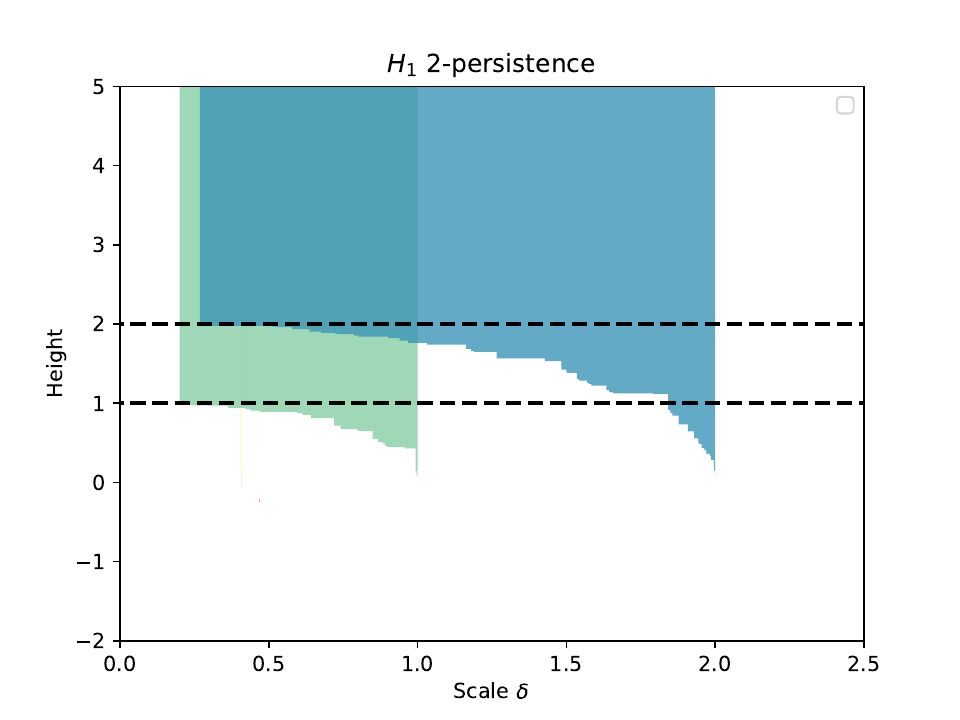}
		\caption{}
	\end{subfigure}
	\caption{
		\textbf{(a)}:~A sample $P$ from a space $X=X_1\sqcup X_2$ composed of two circles equipped with geodesic
		distances, each uniformly sampled with distinct radii and concentration levels. The
		target function $\ff\colon X \to \mathbb{R}$ is the height
		function on the two circles. The barcode of the target $H_1(\ff)$
		consists of two infinite bars, each originating from a level marked as a dashed line.
		\textbf{(b)}:~Bottleneck distance between the barcode of the estimator
		$\Homology_1 \left( \RR^{\delta\to 2\delta}(\FP) \right)$
		and that of the target $\Homology_1 \left( \ff \right)$ as a function of $\delta$.  Here, $\e_1$ (resp. $\e_2$) is
		the sampling error of $X_1$
		(resp. $X_2$), and $\varrho_{X_1}$ (resp. $\varrho_{X_2}$) the convexity radius of $X_1$ (resp. $X_2$). All infinite
		bars are truncated at~10 to ensure the bottleneck distances are finite.
		\textbf{(c)}:~Visualization of the estimator~$\Homology_1 \left( \RR^{\bullet\to 2\bullet}(\FP) \right)$  computed
		using \MMA.
		\textbf{(d)}:~Visualization of the estimator~$\Homology_1 \left( \CC^{\bullet}(\FP) \right)$ computed
		using \MMA.  Each colored region represents a persistent topological feature.  Dashed lines indicate birth
	times of bars in the barcode of~$H_1(\ff)$.	}
	\label{fig:expe:three_annulus_dataset}
\end{figure}

\subsection{Toy examples}
We validate the claims in~\cref{thm:estimator_fixed_radius,thm:estimator_varying_radius} through a simple example
consisting of
two circles with distinct radii and concentration levels, as shown in~\cref{fig:expe:three_annulus_dataset}~(a).
In~\cref{fig:expe:three_annulus_dataset}~(b) we analyze the error between the estimator $H_1(\RR^{\delta\to
2\delta}(\FP))$ and the target $H_1(\ff)$ as a function of the scale parameter~$\delta$. As guaranteed
by~\cref{thm:estimator_fixed_radius}, for $\delta$ within the range~$[2\e,\rhox/2)$ the error between the estimator and
the target
is small, controlled by $\delta$. The approximation remains good even for $\delta$ going beyond that range, which is
not explained by the theory but occurs here because the space and function are very regular in this example.

In~\cref{fig:expe:three_annulus_dataset}~(c) we show a visualization of our estimator  $H_1(\RR^{\bullet\to
2\bullet}(\FP))$ using \MMA{}~\cite{loiseaux2022fast}. As expected from~\cref{thm:estimator_varying_radius}, we observe
two topological features in our estimator, corresponding to the two circles in the dataset. The values of~$\delta$ at
which these features  appear then disappear are aligned with the endpoints of the interval of values of~$\delta$ for
which $H_1(\RR^{\delta\to 2\delta}(\FP))$ is a good approximation of~$H_1(\ff)$ in~\cref{fig:expe:three_annulus_dataset}~(b).

For comparison, in~\cref{fig:expe:three_annulus_dataset}~(d) we show a visualization of the estimator
$H_1(\CC^{\bullet}(\FP))$ using \MMA{}, for which we use the Euclidean distance instead of the geodesic distance in
order to enable the construction of the function–\v{C}ech filtration. Again, we see two topological features
corresponding to the two circles, and $H_1(\CC^{\delta}(\FP))$  is a good approximation of~$H_1(\ff)$ for $\delta$
within a fairly large range.


We now add noise to the dataset sampled from the two circles, as shown
in~\cref{fig:expe:three_annulus_noise_dataset}~(a).  As expected
from~\cref{sec:handle_noise_in_input}, the trend is the same. Note
that,
while~\cref{prop:estimator_rips_noise_distance_fixed_radius,prop:estimator_rips_noise_distance_varying_radius}
prescribe to use a factor $\alpha>2$ in
$H_1(\RR^{\bullet\to\alpha\bullet}(\FP))$ in order to accurately
approximate the target $H_1(\ff)$ in the presence of noise in the
metric, \cref{fig:expe:three_annulus_noise_dataset} shows that, on
this example, our estimator $H_1(\RR^{\bullet\to2\bullet}(\FP))$
performs well even with a factor~$\alpha$ set to~2. 

\begin{figure}[t]
	\centering
	\begin{subfigure}[b]{0.44\textwidth}
		\centering
		\includegraphics[width=\textwidth]{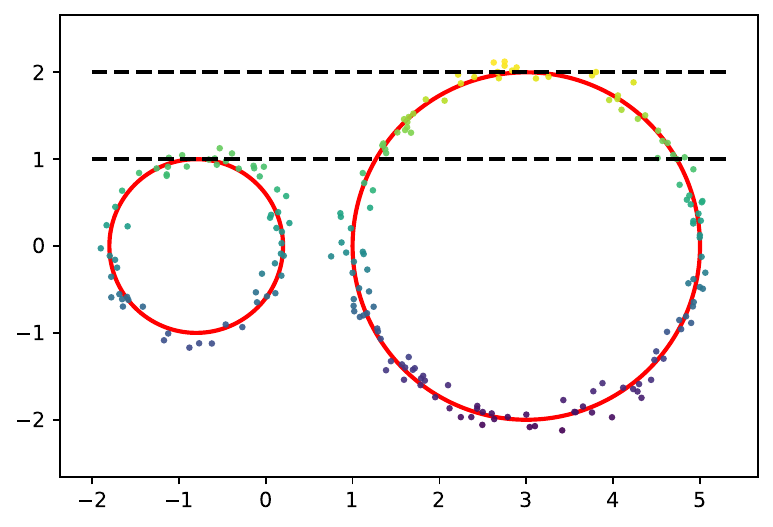}
		\caption{}
	\end{subfigure}
	\begin{subfigure}[b]{0.44\textwidth}
		\centering
		\includegraphics[width=\textwidth]{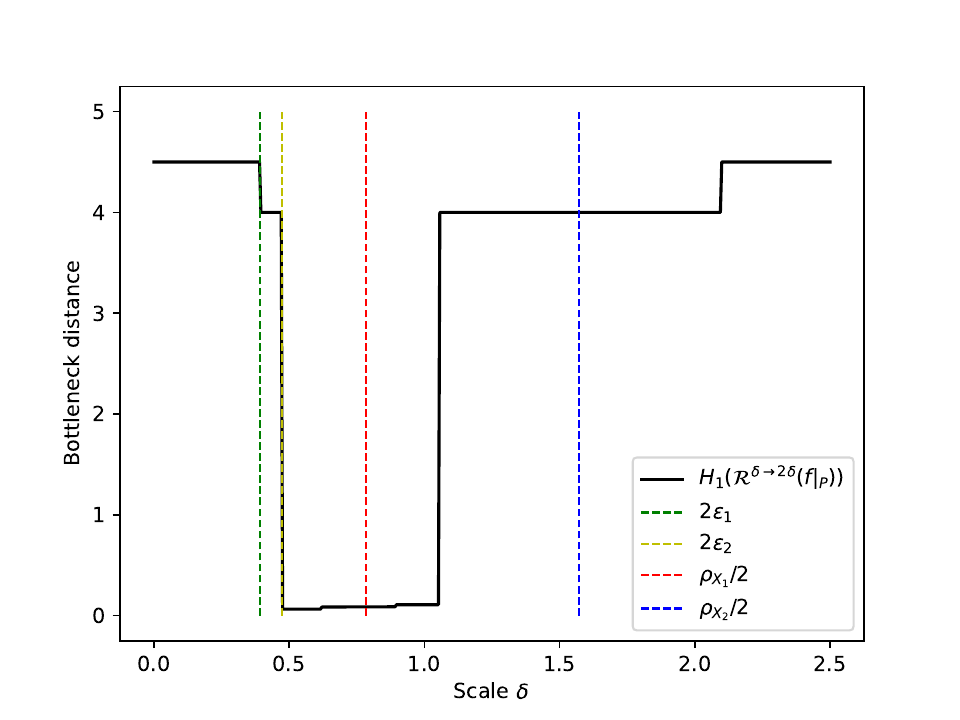}
		\caption{}
	\end{subfigure}
	\vspace{0.5cm}
	\begin{subfigure}[b]{0.44\textwidth}
		\centering
		\includegraphics[width=\textwidth]{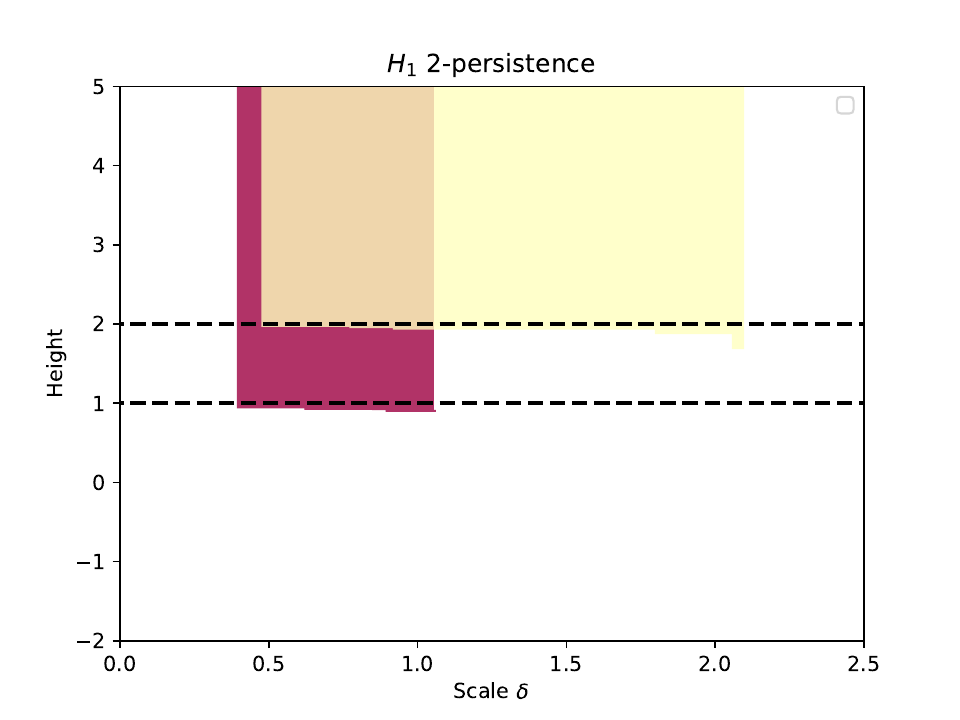}
		\caption{}
	\end{subfigure}
	\begin{subfigure}[b]{0.44\textwidth}
		\centering
		\includegraphics[width=\textwidth]{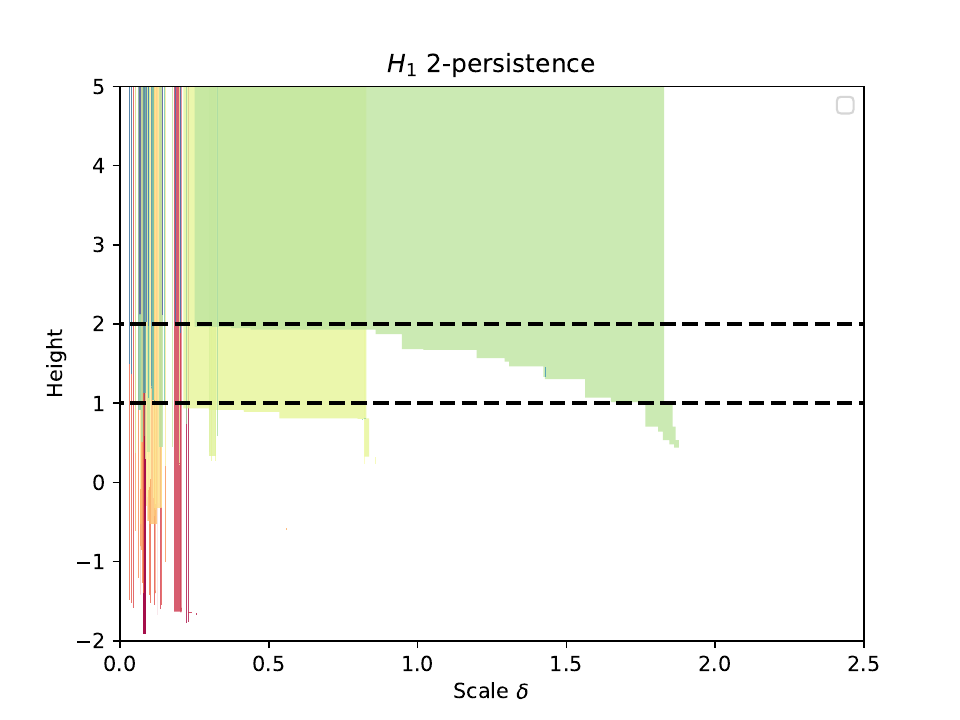}
		\caption{}
	\end{subfigure}
	\caption{A noisy analog of \cref{fig:expe:three_annulus_dataset}.}\label{fig:expe:three_annulus_noise_dataset}
\end{figure}

\newpage

\subsection{Real dataset of immune cells (involving Lipschitz functions)}
In this section, we consider the digitized immunohistochemistry
dataset from~\cite{vipondMultiparameterPersistentHomology2021}.
This dataset is comprised of manually annotated cell locations from three cell types, namely cytotoxic T
	lymphocytes, regulatory T lymphocytes, and macrophages, each characterized by the expression of one of the proteins
CD8, FoxP3, and CD68, respectively---see~\cref{fig:expe:immuno_point_cloud}.

\begin{figure}[t!]
	\centering
	\includegraphics[width=.8\textwidth]{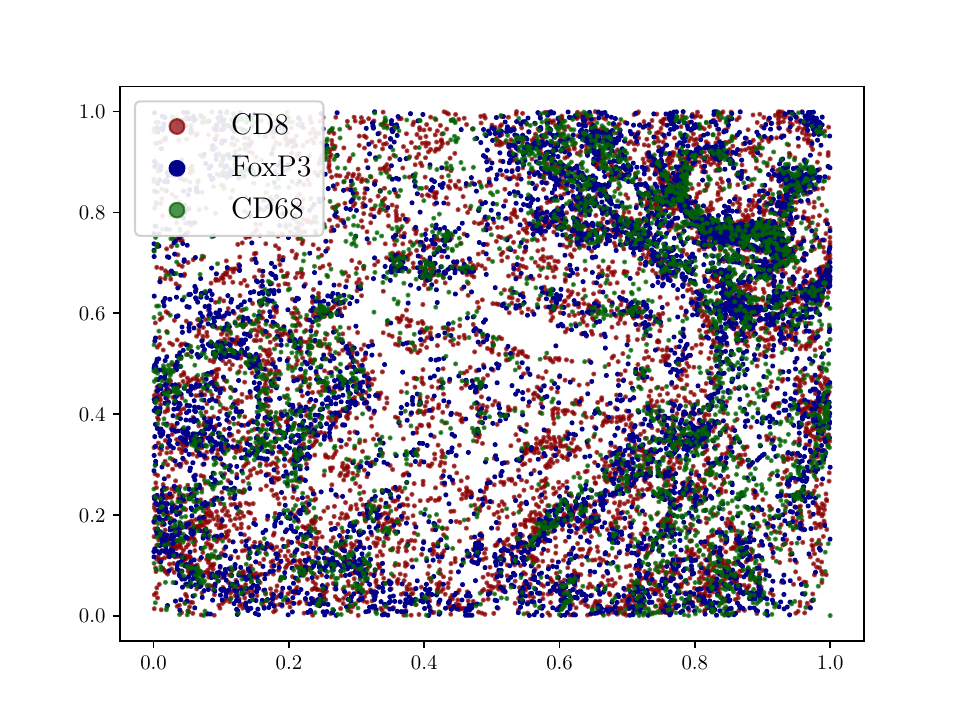}
	\caption{Point cloud dataset, with colors indicating the protein labels CD8, FoxP3, and CD68.
	}\label{fig:expe:immuno_point_cloud}
\end{figure}

\begin{figure}[t!]
	\centering
	\begin{subfigure}{0.48\textwidth}
		\centering
		\includegraphics[width=\textwidth]{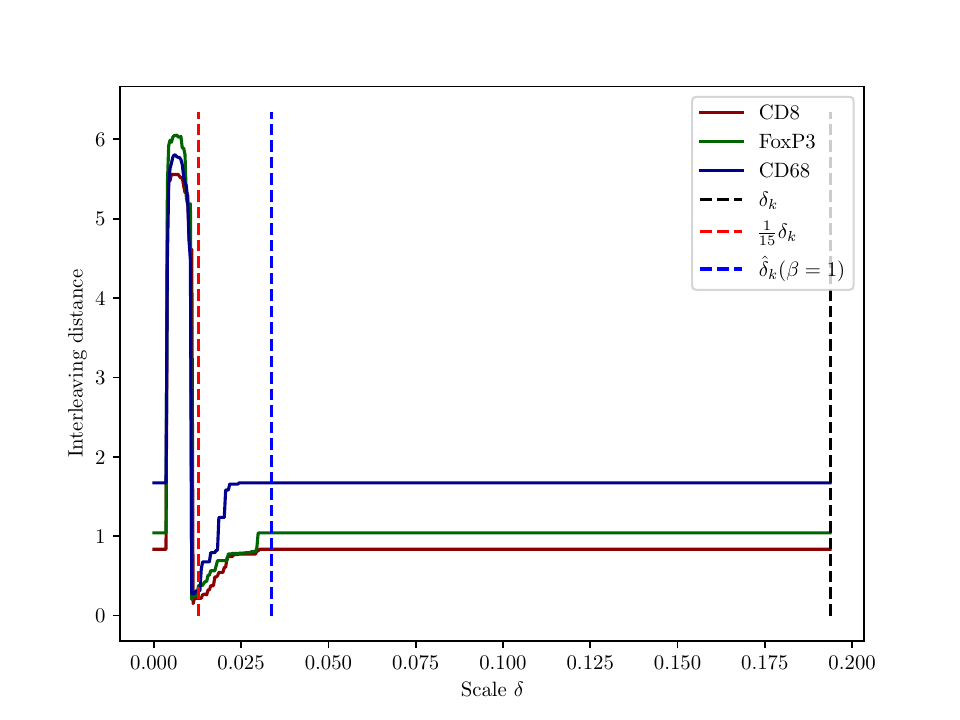}
		\caption{Errors for $H_1(\CC^{\delta}(\ff_i|_{X_{10~000}}))$.}
		\label{fig:expe:immuno_error}
	\end{subfigure}
	\hfill
	\begin{subfigure}{0.48\textwidth}
		\centering
		\includegraphics[width=\textwidth]{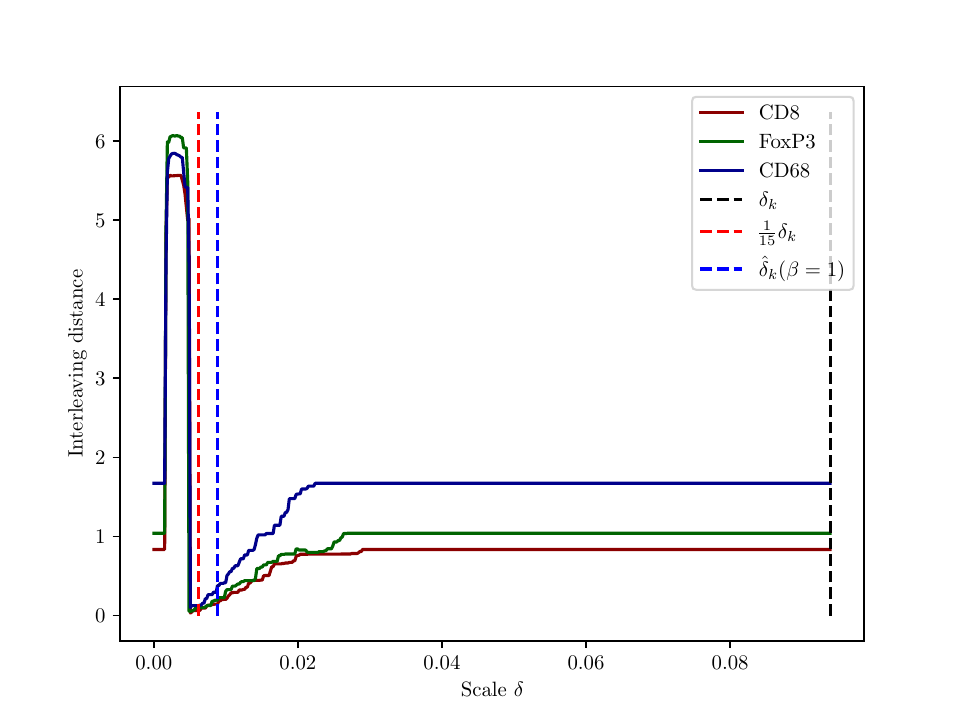}
		\caption{Errors for $H_1(\CC^{\delta}(\ff_i|_{X_{50~000}}))$.}
		\label{fig:expe:immuno_error_fixed_delta}
	\end{subfigure}
	\caption{Interleaving distance between our estimators and the target $H_1(\ff)$ w.r.t. the scale $\delta$.}
	\label{fig:immuno_error:all_delta}
\end{figure}

The index $i \in \{1,2,3\}$ refers to the protein in the ordered set
	$\{\mathrm{CD8},  \mathrm{FoxP3}, \mathrm{CD68}\}$, and
	$\ff_i \colon [0,1]^2 \to \mathbb{R}$ denotes the negative density
	estimate associated with that protein. Since the functions $\ff_i$ are unknown, we estimate the density of each point
	cloud using kernel density estimation, which yields approximations---also called $\ff_1, \ff_2, \ff_3$ for simplicity---that are Lipschitz continuous by construction. We then discretize the domain $[0,1]^2$ on a grid (using a resolution of $1~000 \times
	1~000$ for~\cref{sec:R_valued_Lipschitz} and of $500 \times 500$ for~\cref{sec:R3_valued_Lipschitz}) and compute the
	persistence module of the cubical complex induced by the estimated function on this grid, which serves as a proxy for the ground-truth target.

	In the following experiments, in order to reach smaller Hausdorff distances between
	the sample $X_k$ and the domain $X =[0,1]^2$, we first draw \(10{,}000\) points uniformly from \(X\), and then select \(X_k\) using \(k\)-means farthest point sampling applied to this set. In practice, the choice of the rate $\delta_k$ prescribed in~\cref{thm:stat_results_overview} can
	be  improved significantly when the sample $X_k$ is well-behaved.
	We propose heuristics to estimate good rates in non-asymptotic regimes.

\subsubsection{Estimating a single function $\ff_i$}\label{sec:R_valued_Lipschitz}
We empirically validate our theoretical results by examining the estimators \(H_1(\CC^{\delta}(\ff_i|_{X_k}))\) with
$i\in\{1,2,3\}$
built from the three associated point clouds.
%
We first focus on the convergence rate of the estimators as the sample size $k$
	increases, with respect to the choice of scale parameter~$\delta_k$.
	From \cref{thm:stat_results_overview}, since the sampling measure is known---namely the uniform measure on $[0,1]^2$, which is $(a,b)$-standard with
	parameters $a=\frac \pi 4$ and $b=2$---we can derive an asymptotic theoretical rate
	$\delta_k = 4 \left( \frac {8 \log(k)}{\pi k}\right)^{\frac 1 2}$.
	However, as shown in \cref{fig:immuno_error:all_delta}, this choice
	of $\delta_k$ is very conservative, and consequently the estimators
	$H_1(\CC^{\delta_k}(\ff_i|_{X_k}))$ do not perform better than the zero
	module.
	This is explained by the fact that $\delta_k$ is
	designed for worst-case sampling measures and 
	asymptotic regimes.
	As shown in~\cref{fig:immuno_error:all_delta}, the estimators
	$H_1(\CC^{\delta_k}(\ff_i|_{X_k}))$ achieve good performance only when
	$\delta_k$ is smaller than $0.02$, which would require sample sizes 
	$k$ larger than $10^6$.

        Another option is to consider the
        estimated
        rate $\hat\delta_k$ mentioned in \cref{thm:stat_results_overview} and defined in  \cref{eq:stats:hatdeltak} (\cref{sec:ab_unknown}). We pick $\beta=1$ as a default choice of parameter in the definition of $\hat\delta_k$.
	In \cref{fig:immuno_error:all_delta}, we observe that this rate
	is far less conservative than~$\delta_k$ and already exhibits its asymptotic behavior when
	$k$ is larger than $50~000$, achieving good performance.
	In this example though, the rate $\hat \delta_k$ has some limitations:
	\begin{itemize}
		\item the bad behavior of the corresponding estimators on small sample sizes ($k\le 10^4$);
		\item the dependence of $\hat \delta_k$ on a parameter $\beta$, which needs to be tuned;
		\item the fact that the knowledge  of the ambient dimension and the probability measure from which the points are sampled are not leveraged.
	\end{itemize}

	\begin{figure}[b!]
	\centering
	\begin{subfigure}{0.48\textwidth}
		\centering
		\includegraphics[width=\textwidth]{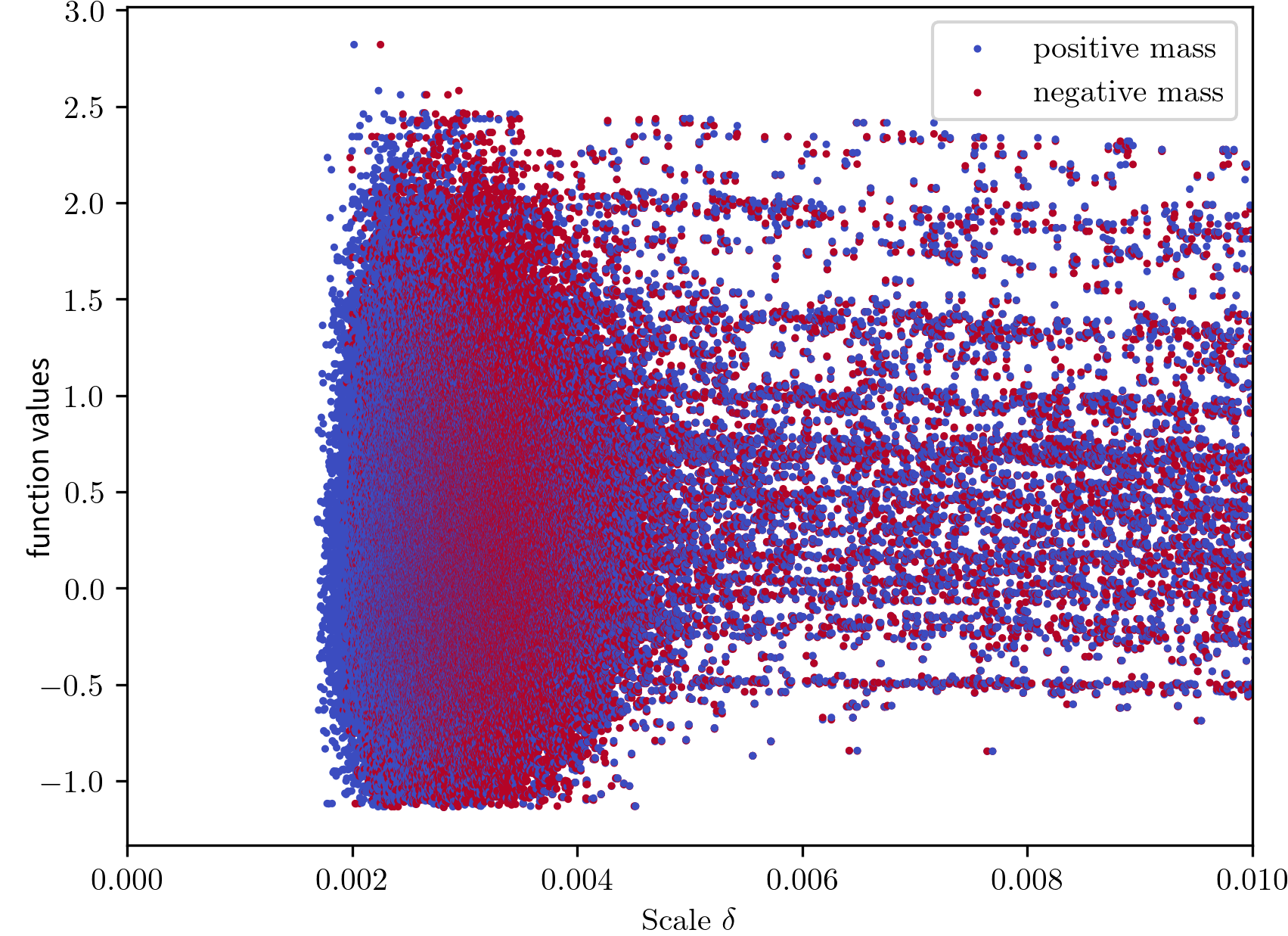}
		\caption{Hilbert function signed bars of $H_1(\CC^{\bullet}(\ff_1|_{X_{k}}))$.}
		\label{fig:expe:immuno_betti_concentration}
	\end{subfigure}
	\hfill
	\begin{subfigure}{0.48\textwidth}
		\centering
		\includegraphics[width=\textwidth]{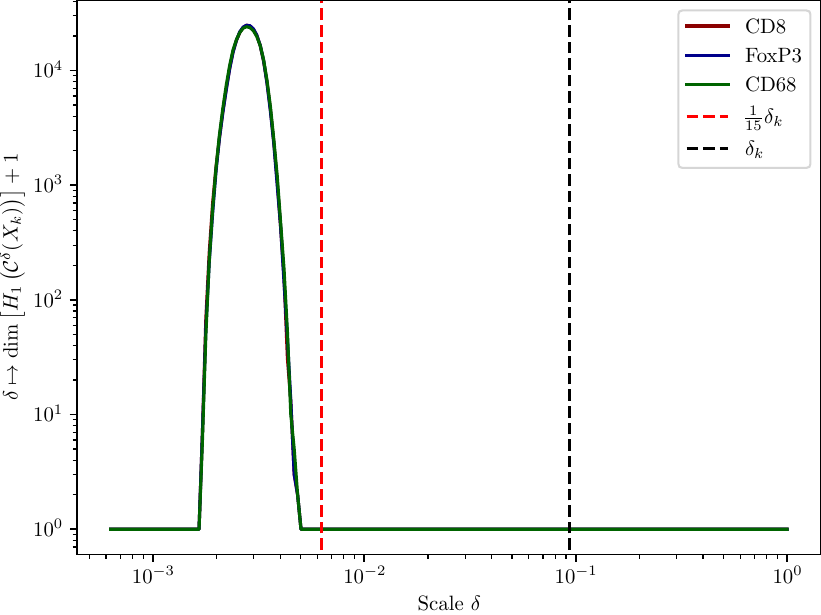}
		\caption{Pointwise dimension of $H_1(\CC^{\delta}(X_{k}))$ w.r.t $\delta$.}
		\label{fig:expe:immuno_pointwise_dim}
	\end{subfigure}
	\caption{Heuristic to determine a good multiplicative constant for the rate $\delta_k$, where $k=50~000$.}
	\end{figure}

	\begin{figure}[t]
	\centering
	\begin{subfigure}{0.48\textwidth}
		\centering
		\includegraphics[width=\textwidth]{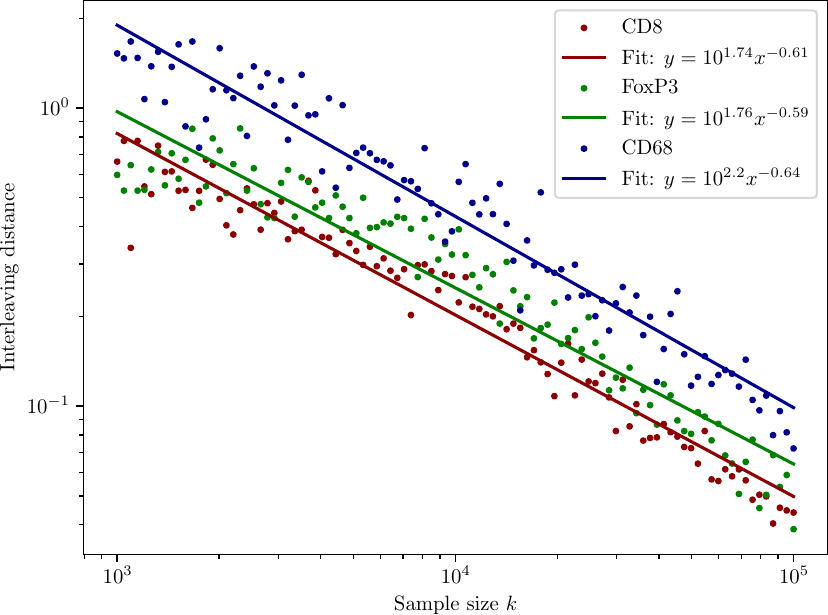}
		\caption{Convergence rate with $\delta'_k = \frac 1 {15} \delta_k$.}
		\label{fig:expe:immuno_delaunay_errors_n100000_delta_FT}
	\end{subfigure}
	\hfill
	\begin{subfigure}{0.48\textwidth}
		\centering
		\includegraphics[width=\textwidth]{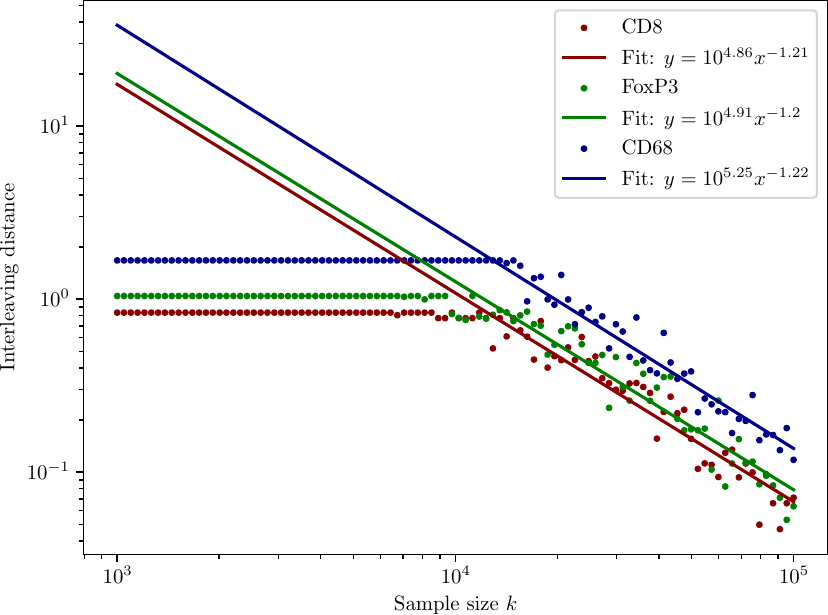}
		\caption{Convergence rate with $\hat{\delta}_k$.}
		\label{fig:expe:immuno_delaunay_errors_n100000_delta_hat}
	\end{subfigure}
        \caption{Interleaving distance between the estimators and the target w.r.t. the sample size~$k$.}
        \label{fig:immuno_error:all_k}
\end{figure}

\begin{figure}[t]
		\centering
		\includegraphics[width=.6\textwidth]{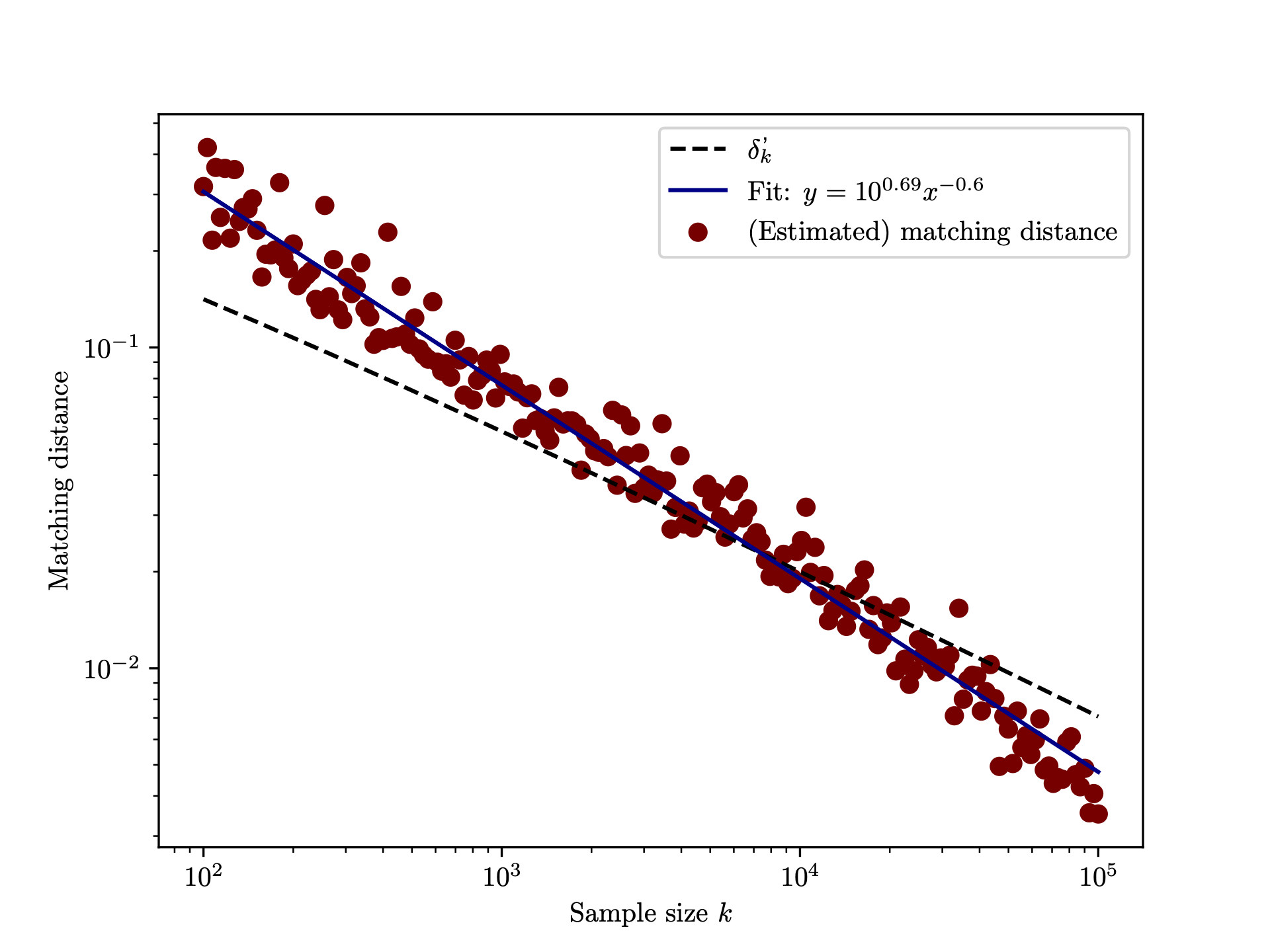}
		\caption{
			Estimated matching distance between the estimator $H_0(\RR^{\delta'_k\to 2\delta'_k}(\ff|_{X_k}))$
			and the ground truth $H_0(\ff)$
			as a function of the sample size $k$ under random sampling.
		}\label{fig:expe_3param}
\end{figure}

	Here is a heuristic we apply to select a relevant rate~$\delta'_k$. Since the dimension of the domain is known, we set $\delta'_k$ to  be a fixed
	fraction of the theoretical rate~$\delta_k$.
	This is motivated by \cref{eq:packing_covering_rates} in \cref{appendix:stats},
	which ensures that, up to a multiplicative constant,
	the rate $(1/k)^{1/d}$ cannot be improved for
	any (possibly deterministic) $k$-sample of a $d$-dimensional space.
	To determine a good multiplicative constant,
	we notice from \cref{fig:immuno_error:all_delta} that, when $\delta\in[0, \dH(X_k, X))$,
	the error of the estimators $H_1(\CC^{\delta}(\ff_i|_{X_k}))$ soars, due to
	the fact that $\CC^\delta(X_k)$ and $X$ do not have the same topology.
	According to \cref{thm:estimator_fixed_radius}, the optimal scale $\delta$ is
	then given as the smallest value lying above that interval, that is: $\dH(X_k, X)$. Our rate $\delta'_k$ is an estimate of this value.
	In order to compute it, we look at the topological changes of $\CC^\delta(X_k)$
	as $\delta$ increases. Since the measure is uniform and thus does not exhibit pathological local behaviors, the
	artifacts in the topological type of $\CC^\delta(X_k)$ should all vanish roughly at the same scale, and slightly before~$\dH(X_k, X)$.
	This phenomenon, suggested by the shape of the curves in~\cref{fig:immuno_error:all_delta}, is further
	illustrated in~\cref{fig:expe:immuno_betti_concentration}, where
	 we consider the {\em Hilbert function signed barcode}~\cite{loiseauxStableVectorizationMultiparameter2023,oudot2024stability} of $H_1(\CC^{\bullet}(\ff_1|_{X_k}))$:
        after an initial transition phase ($\delta<0.005$) during which a concentration of positive generators appears, followed then by a concentration of negative generators, the structure of the module stabilizes as these positive and negative generators cancel each other out. This is highlighted when summing the signed generators within the half-plane $(-\infty, \delta]\times\RRR$, or equivalently, when computing the pointwise dimension of $H_1(\CC^{\delta}(X_{50~000}))$, for $\delta$ ranging from $0$ to $+\infty$, as shown in~\cref{fig:expe:immuno_pointwise_dim}. Since the convexity radius $\rhox$ is infinite in this case, we have  $H_1(\CC^{\delta}(X_k)) \cong H_1(\OO^\delta(X_k)) \cong H_1(X)$ for all $\delta>\dH(X_k, X)$, and so
	our heuristic to determine a suitable rate ${(\delta_k')}_{k\in \mathbb{N}}$ from
	a sample $X_{k_0}$ for some sufficiently large $k_0\in \mathbb{N}$ is: 
	\begin{enumerate}
		\item Compute the pointwise dimension of $H_1(\CC^{\delta}(X_{k_0}))$ as in
			\cref{fig:expe:immuno_pointwise_dim} for a range of $\delta$;
		\item Identify the transition phase of $\dim H_1(\CC^{\delta}(X_{k_0}))$;
		\item Choose a scale $\delta^*$ past this phase, so that
                  $\delta \in
			[\delta^*, \infty)\mapsto\dim H_1(\CC^{\delta}(X_{k_0}))$ is constant; 
		\item Define the rate ${(\delta'_k)}_k$ for $k\in  \mathbb{N}$ as
			$\delta'_k := \frac {\delta^*}{\delta_{k_0}} \delta_k$. 
	\end{enumerate}

	In~\cref{fig:immuno_error:all_k},
	we observe that both rates $\delta'_k$ and $\hat{\delta}_k$ achieve good
	performance asymptotically, with a convergence rate that is at least of the order of 
	$\left(\frac{\log(k)}{k}\right)^{1/2}$.
        However, the asymptotic regime is reached earlier with $\delta'_k$ than with $\hat{\delta}_k$---while, as we said, it is not reached with~$\delta_k$.


%

	\subsubsection{Estimating the combined functions $(\ff_1, \ff_2, \ff_3)$}\label{sec:R3_valued_Lipschitz}
	We now study the function $\ff = (\ff_1, \ff_2, \ff_3)\colon [0,1]^2\to \mathbb{R}^3$ and
	approximate the corresponding 3-parameter persistence module $\hzero(\ff)$.
	For this we employ the estimator $\hzero(\RR^{\delta_k'\to2\delta_k'}(\ff|_{X_k})) \cong
	\hzero(\RR^{2\delta_k'}(\ff|_{X_k}))$, where, for each sample size~$k$, the scale parameter is
	given by the same $\delta_k'$ as in \cref{sec:R_valued_Lipschitz}.
	Since computing the interleaving distance in this context  is
	NP-hard~\cite{bjerkevik2020computing}, we
	use the matching distance~\cite{landi2018rank} as a proxy, estimated via Monte Carlo with $1,000$ random lines in~$\RRR^3$.
The results are
shown in~\cref{fig:expe_3param}.
Notably, the observed convergence rate matches
that of the previous experiment in \cref{fig:expe:immuno_delaunay_errors_n100000_delta_FT},
which suggests that the empirical
	convergence rate is independent of the number of parameters.

\begin{figure}[t!]
	\centering
	\includegraphics[width=.7\textwidth]{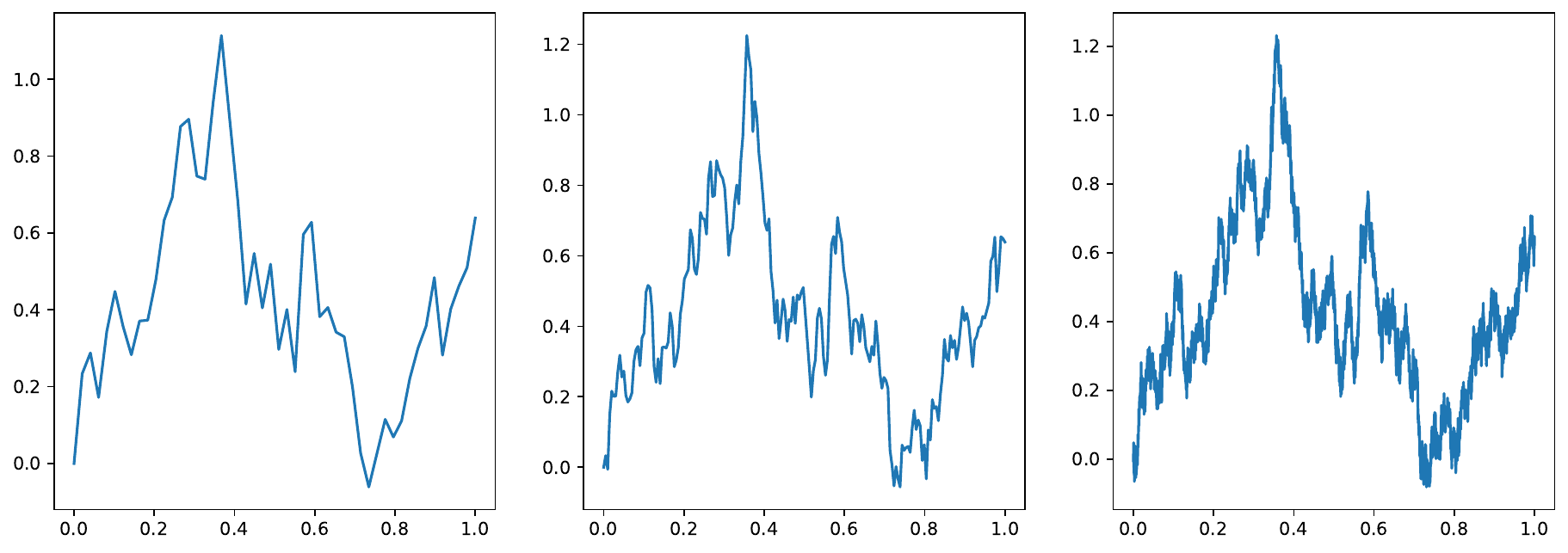}
	\caption{Samplings of a Brownian motion on $[0,1]$ with, respectively, $50$,  $200$, and $10^6$
	points.}
	\label{fig:brownian_motion_dataset}
\end{figure}

\begin{figure}[t!]
	\centering
	\includegraphics[width=.5\textwidth]{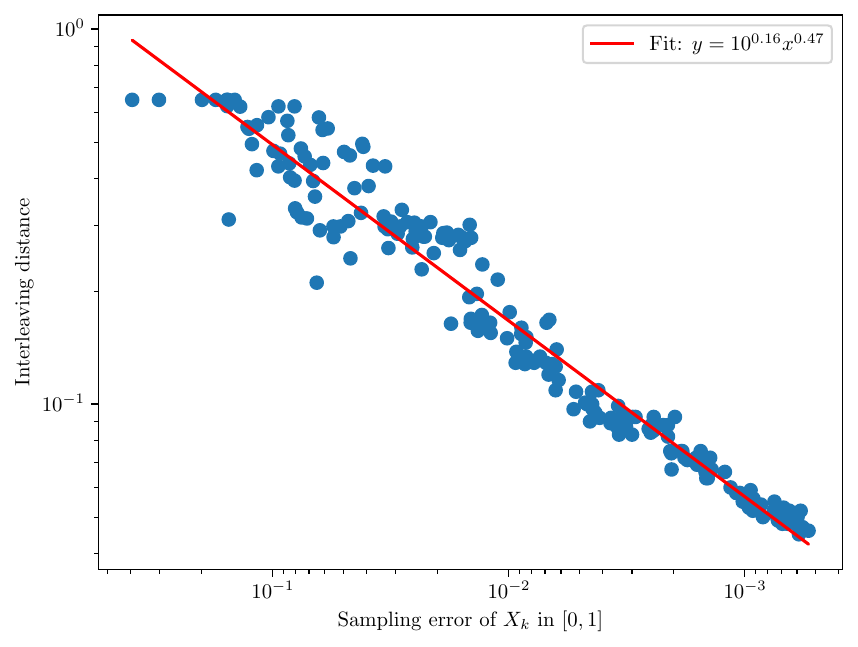}
	\caption{Log–log plot of the convergence of $\hzero(\CC^{\e_k}( \ff|_ {X_k}))$ toward $\hzero(\ff)$ for a
			Brownian motion $\ff$.
			The x-axis represents the sampling
			error $\e_k = \dH(X_k, [0,1])$, and the y-axis represents the interleaving distance between $\hzero(\CC^{\e_k}( \ff|_
			{X_k}))$ and $\hzero(\ff)$.
        }
	\label{fig:brownian_motion_cv_rate2}
\end{figure}

\subsection{Brownian motion in 1-d (involving a H\"older function)}

The Brownian motion is not Lipschitz continuous but almost surely almost everywhere
	$\omega$-continuous for any modulus of continuity $\omega$ of the
	form $\omega\colon x \mapsto c\, x^{\frac 1 2 - \varepsilon}$, with
	$\varepsilon>0$ and $c>0$ large enough~\cite[Corollaries~1.20 and~5.3]{mortersBrownianMotion2010}. Let $\ff:[0,1]\to \RRR$ be a realization of a standard Brownian motion, which can be approximated by a Rademacher random walk.

	For each $k\in\NNN$, we build a $k$-sample~$X_k$ in $[0,1]$ according to the uniform distribution, as shown in~\cref{fig:brownian_motion_dataset}
	for several values of $k$. Then, since our goal here is to illustrate \cref{thm:estimator_fixed_radius}, we assume we have access to the sampling error $\e_k:=\dH(X_k, [0,1])$ and so we
          use the estimator
	$\hzero(\CC^{\e_k}( \ff|_ {X_k}))$ to approximate the target $\hzero(\ff)$ for various sample sizes~$k$.
	The corresponding results are shown in~\cref{fig:brownian_motion_cv_rate2}.
	A regression analysis of the experimental data
	indicates that
	\(
		\dIo\left(\hzero(\CC^{\e_k}( \ff|_ {X_k})),\hzero(\ff)\right)\approx
	10^{0.16}\,{\e_k}^{0.47},
	\)
	which aligns closely with the theoretical prediction
	from~\cref{thm:estimator_fixed_radius}, stating that the interleaving distance between $\hzero(\CC^{\e_k}( \ff|_
	{X_k}))$ and $H_0(\ff)$ should be at most  $\omega(\e_k) \lesssim {\e_k}^{1/2}$.


\subsection{Timings}
In~\cref{tab:timings},
we report the running times for some of our experiments.
Only the first one involves our algorithm for computing a free presentation of $\ms$. For all computations
%
 we used an
  Intel\textsuperscript{\textregistered} Core\textsuperscript{\texttrademark} i9-12900K CPU (5.2\,GHz max)~equipped~with~125\,GB~of~RAM.

\begin{table}[H]
    \begin{tabular}{|p{2.1cm}|r|r|r|r|r|p{2.7cm}|}
    \hline
    Experiment & Estimator & \#Params & \#Points & Filt.\ size & Filt.\ time & Invariant time \\
    \hline\hline
    2 circles   &  $\msone$   & 2 & 220   & 13{,}110 & 11.8\,ms & 6.5\,ms\newline(presentation) \\
    \hline
    Immune cells \newline (single $\ff_i$) &  $H_1(\CC^{\bullet}(\ff_i|_{X_k}))$  & 2 & 6172  & 24{,}370  & 864\,ms & 50.3\,ms\newline(Hilbert\;measure) \\
    \hline
    Brownian \newline motion & $H_0(\CC^{\e_k}(\ff|_{X_k}))$ & 1 & 10{,}000 & 37{,}359 & 1.21\,s & 5.4\,ms\newline(barcode) \\
    \hline
  \end{tabular}
  \caption{Running times for some experiments. \emph{\#Params} and \emph{\#Points} denote the numbers of parameters and data points, respectively.
\emph{Filt.\ size} is the size of the filtered simplicial complex;
\emph{Filt.\ time} is the time required to construct the filtration;
and \emph{Invariant time} is the time required to compute the invariant from the filtration.}
	\label{tab:timings}
\end{table}

%% file: Sections/appendix.tex
\section{Proof of Theorem~\ref{thm:stability_betti_vertical}}
\label{sec:proof_stability_inv}
The proof closely mirrors that of \cite[Theorem~1]{oudot2024stability}, with
two key differences. The first one is that the proof
in~\cite{oudot2024stability} invokes the stability theorem for free persistence
modules under ordinary interleaving proven in~\cite{bakke2021stability}, while
our proof relies on the stability of free persistence modules under  vertical
interleaving, which we state as Lemma~\ref{lem:stability_free_flow}. That
lemma, in turn, relies on two intermediate results, namely:
Lemma~\ref{lemma:nonzero_morphism_same_region} and
Proposition~\ref{prop:submodules_interleaving}. The second key difference
with~\cite{oudot2024stability} is that our proof uses a persistent version of
Schanuel's lemma that deals with vertically interleaved projective resolutions,
whereas the version used in~\cite{oudot2024stability} is for ordinarily
interleaved projective resolutions. We state our version as
Lemma~\ref{lem:cor9_betti_paper}.

Let $F:[a,b]\times\rn\to\vect$ be a free persistence module. For any free
interval $J=[x_0,b]\times [x_1,+\infty)\times \cdots\times[x_n,+\infty)\subset [a,b]\times\rn$, define
$\alpha(J):=x_0\in [a,b]$. Then, for any $\delta\in[a,b]$, define the submodule:
\[ F^{\delta} := \bigoplus\limits_{\substack{J\in\BB(F)\\ \alpha(J)=\delta} } \kkk^J. \]

\begin{lemma}\label{lemma:nonzero_morphism_same_region}
	Let $J,J',K$ be free intervals in $[a,b]\times\rn$ and let $\e\geq 0$. Suppose there are two non-zero morphisms
	\(f_{J,K}:\kkk^{J}\rightarrow \kkk^{K}[\e\oneo]\) and \(g_{K,J'}:\kkk^{K}\rightarrow \kkk^{J'}[\e\oneo].\)
	If $\alpha(J)=\alpha(J')=\delta\in [a,b]$, then $\alpha(K)=\delta$.
\end{lemma}
\begin{proof}
	Since $f_{I,K}$ and $g_{K,J'}$ are non-zero, \cite[Lemma~4.5]{bakke2021stability} implies that:
	\begin{align*}
		\begin{aligned}
			\min(J) & \geq \min(K)-\e\oneo,  \\
			\min(K) & \geq \min(J')-\e\oneo.
		\end{aligned}
	\end{align*}
	It follows that $\delta=\alpha(J)\geq \alpha(K)\geq \alpha(J')=\delta$, so $\alpha(K)=\delta$.
\end{proof}

\begin{proposition}\label{prop:submodules_interleaving}
	Let $M,N:[a,b]\times\rn\to\vect$ be two free persistence modules and let
	$\varepsilon\geq 0$. If $M$ and $N$ are vertically
	${\varepsilon}$-interleaved, then the submodules $M^{\delta}$ and $N^{\delta}$ are also vertically
	${\varepsilon}$-interleaved, for any $\delta\in [a,b]$.
\end{proposition}
\begin{proof}
	The proof follows the same idea as that of \cite[Lemma~4.9]{bakke2021stability}.
	Since $M$ and $N$ are free persistence modules, we write $M\cong \bigoplus_{J\in \BB(M)}\kkk^{J}$ and $N\cong
	\bigoplus_{K\in \BB(N)}\kkk^{K}$.
	Considering $M$ and $N$ are vertically ${\varepsilon}$-interleaved, there exist two morphisms $f:M\rightarrow
	N[\e\oneo]$ and $g:N\rightarrow M[\e\oneo]$ such that
	\begin{align}
		g[\e\oneo]\circ f & = \varphi_{M}^{2\e\oneo},\label{equation:fg_phi} \\
		f[\e\oneo]\circ g & = \varphi_{N}^{2\e\oneo}.\label{equation:gf_phi}
	\end{align}
	For any $J\in \BB(M)$ and $K\in \BB(N)$, define morphisms $f_{J,K}:\kkk^J\to\kkk^K[\e\oneo]$ and
	$g_{K,J}:\kkk^K\to\kkk^J[\e\oneo]$ as the following compositions, respectively:
	\begin{align}
		f_{J,K} & :\kkk^{J} \xhookrightarrow{\iota_J} M \xrightarrow{f}N[\e\oneo]\stackrel{\pi_K[\e\oneo]}\twoheadrightarrow
		\kkk^{K}[\e\oneo],\label{equation:fxy_diagram}  \\
		g_{K,J} & :\kkk^{K}\xhookrightarrow{\iota_K} N \xrightarrow{g} M[\e\oneo]\stackrel{\pi_J[\e\oneo]}\twoheadrightarrow
		\kkk^{J}[\e\oneo], \label{equation:gyx_diagram}
	\end{align}
	where $\iota_\bullet$ is the canonical injection into the direct sum, and $\pi_\bullet$ is the canonical projection
	from the direct sum.

	Morphisms $f$ and $g$ are assembled from $\{f_{J,K}\}_{J\in\BB(M),
	K\in\BB(N)}$ and $\{g_{K,J}\}_{K\in\BB(N), J\in\BB(M)}$, respectively. Then
	define $f^{\delta}:M^{\delta}\to N^{\delta}[\e\oneo]$ assembled (in the same
	way as $f$) from $f^{\delta}_{J,K}:=f_{J,K}$ for all $J\in\BB(M)$ and $K\in\BB(N)$ such that
	$\alpha(J)=\alpha(K)=\delta$. Let $g^{\delta}: N^{\delta} \to M^{\delta}[\e\oneo]$ be defined analogously.
	If we can prove that the following equation holds for all $J,J'\in\BB(M)$ such that $\alpha(J)=\alpha(J')=\delta$:
	\begin{equation}\label{equ:interleaving_RJ}
		\sum\limits_{K\in\BB(N)}g_{K,J'}[\e\oneo]f_{J,K}
		= \sum\limits_{\substack{ K\in\BB(N) \\ \alpha(K)=\delta}}g^{\delta}_{K,J'}[\e\oneo]f^{\delta}_{J,K},
	\end{equation}
	and that the following equation holds  for all $K,K'\in\BB(N)$ such that $\alpha(K)=\alpha(K')=\delta$:
	\begin{equation}\label{equ:interleaving_RK}
		\sum\limits_{J\in\BB(M)}f_{J,K'}[\e\oneo]g_{K,J}
		= \sum\limits_{\substack{ J\in\BB(M) \\ \alpha(J)=\delta}}f^{\delta}_{J,K'}[\e\oneo]g^{\delta}_{K,J},
	\end{equation}
	then $f^{\delta}$ and $g^{\delta}$ are $\e\oneo$-interleaving morphisms
	between $M^\delta$ and $N^\delta$. The reason is because the left-hand side in~\eqref{equ:interleaving_RJ} is equal to
	the composition:
	\[ \kkk^{J}\xhookrightarrow{\iota_J} M
		\xrightarrow{\varphi^{2\varepsilon\oneo}_{M}}M[2\e\oneo]\stackrel{\pi_{J'}[2\e\oneo]}\twoheadrightarrow
	\kkk^{J'}[2\e\oneo], \]
	while the right-hand side is equal to the composition:
	\[ \kkk^{J}\xhookrightarrow{\iota_J} M^\delta \xrightarrow{f^\delta}
		N^\delta[\e\oneo]\xrightarrow{g^\delta[\e\oneo]} M^\delta[2\e\oneo] \stackrel{\pi_{J'}[2\e\oneo]}\twoheadrightarrow
	\kkk^{J'}[2\e\oneo]. \]
	And similarly for~\eqref{equ:interleaving_RK}.
	So, all that remains to do is to prove  \eqref{equ:interleaving_RJ} and
	\eqref{equ:interleaving_RK}. Here we only prove \eqref{equ:interleaving_RJ}, because the proof of
	\eqref{equ:interleaving_RK} is similar.

	Lemma~\ref{lemma:nonzero_morphism_same_region} says that, if
	$g_{K,J'}[\e\oneo]f_{J,K}\neq 0$ and $\alpha(J)=\alpha(J')=\delta$, then
	$\alpha(K)=\delta$. Thus, the left-hand side of \eqref{equ:interleaving_RJ} is equal to
	\[\sum\limits_{\substack{K\in\BB(N) \\ \alpha(K)=\delta}}g_{K,J'}[\e\oneo]f_{J,K}.\]
	By definition, $g_{K,J'}=g^{\delta}_{K,J'}$ and $f_{J,K}=f^{\delta}_{J,K}$ when
	$\alpha(J)=\alpha(K)=\alpha(J')=\delta$. This proves~\eqref{equ:interleaving_RJ}.
\end{proof}

\begin{lemma}\label{lem:stability_free_flow}
	Let $M,N:[a,b]\times\rn\to\vect$ be free persistence modules of finite rank. We have:
	\[
		\dBv(\BB(M),\BB(N))\leq
		\begin{cases}
			(n-1)\dIv(M,N) & \text{if}\ n>1, \\
			\dIv(M,N)      & \text{if}\ n=1.
		\end{cases}
	\]
\end{lemma}
\begin{proof}
	If $\dIv(M,N)=+\infty$, then the inequality holds trivially. Otherwise, let
	$\varepsilon=\dIv(M,N)<+\infty$. By
	Proposition~\ref{prop:submodules_interleaving}, for each $\delta\in [a,b]$ the
	submodules $M^{\delta}$ and $N^{\delta}$ are vertically
	${\varepsilon}$-interleaved which means that the restricted modules
	$M^{\delta}|_{\{\delta\}\times\rn}$ and $N^{\delta}|_{\{\delta\}\times\rn}$
	are ordinarily $\e$-interleaved. From \cite[Theorem~4.12]{bakke2021stability},
	when $n\geq 2$, there exists a $(n-1)\e\bm{1}$-bottleneck matching between
	$\BB(M^{\delta}|_{\{\delta\}\times\rn})$ and
	$\BB(N^{\delta}|_{\{\delta\}\times\rn})$, which induces a
	$(n-1)\e\oneo$-bottleneck matching between $\BB(M^{\delta})$ and
	$\BB(N^{\delta})$, since $\BB(M^{\delta})=\{ [\delta,b]\times I \mid I\in
	\BB(M^{\delta}|_{\{\delta\}\times\rn})\}$ and $\BB(N^{\delta})=\{ [\delta,b]\times I \mid I\in
	\BB(N^{\delta}|_{\{\delta\}\times\rn})\}$.
	Note that $M\cong \bigoplus _{\delta\in[a,b]}M^{\delta}$ and  $N\cong
	\bigoplus _{\delta\in[a,b]}N^{\delta}$, so $\BB(M)=\bigsqcup_{\delta\in[a,b]}
	\BB(M^{\delta})$ and  $\BB(N)=\bigsqcup_{\delta\in[a,b]} \BB(N^{\delta})$. By
	combining $(n-1)\e\oneo$-bottleneck matchings between $\BB(M^{\delta})$ and
	$\BB(N^{\delta})$ for all $\delta\in[a,b]$, we obtain a  ${(n-1)\e\oneo}$-bottleneck matching between $\BB(M)$ and
	$\BB(N)$.

	When $n=1$, the stability theorem for $1$-parameter persistence
	modules~\cite{bauer2015induced,chazal2009proximity,chazal2016structure} says
	that there exists a $\e\bm{1}$-bottleneck matching between
	$\BB(M^\delta|_{\{\delta\}\times\RRR})$ and $\BB(N^\delta|_{\{\delta\}\times\RRR})$. The conclusion in this case
	follows then by the same argument as above.
\end{proof}


\begin{lemma}\label{lem:cor9_betti_paper}
	Let $M,N:[a,b]\times\rn\to\vect$ be persistence modules. Let
	$P_\bullet\twoheadrightarrow M$ and $Q_\bullet \twoheadrightarrow N$ be
	projective resolutions of length at most $l$. If $M$ and $N$ are vertically
	${\varepsilon}$-interleaved, then $\bigoplus _{i\in\NNN}P_{2i}\oplus\bigoplus
	_{i\in\NNN}Q_{2i+1}$ and $\bigoplus _{i\in\NNN}P_{2i+1}\oplus\bigoplus _{i\in\NNN}Q_{2i}$ are vertically
	${(l+1)\varepsilon}$-interleaved.
\end{lemma}
The  proof of this lemma is literally the same as the one of~\cite[Corollary~9]{oudot2024stability}, which is oblivious
to the type of interleaving.

\begin{proof}[Proof of Theorem~\ref{thm:stability_betti_vertical}]
	By Hilbert's syzygy theorem, the minimal projective resolutions
	$P_\bullet\twoheadrightarrow M$ and $Q_\bullet\twoheadrightarrow N$ have
	length at most $n+1$. Then, Lemma~\ref{lem:cor9_betti_paper} implies that the
	free modules $\bigoplus _{i\in\NNN}P_{2i} \oplus \bigoplus
	_{i\in\NNN}Q_{2i+1}$ and $\bigoplus _{i\in\NNN}P_{2i+1} \oplus \bigoplus _{i\in\NNN}Q_{2i}$ are vertically
	${(n+2)\varepsilon}$-interleaved. Their respective barcodes are:
	\begin{align*}
		\BB\big( \bigoplus _{i\in\NNN}P_{2i}\oplus\bigoplus _{i\in\NNN}Q_{2i+1} \big) & =\beta_{\even}(M)\sqcup
		\beta_{\odd}(N),\ \text{and} \\
		\BB\big( \bigoplus _{i\in\NNN}P_{2i+1}\oplus\bigoplus _{i\in\NNN}Q_{2i} \big) & =\beta_{\odd}(M)\sqcup
		\beta_{\even}(N).
	\end{align*}
	Then, Lemma~\ref{lem:stability_free_flow} implies that there exists a vertical
	${(n-1)(n+2)\varepsilon}$-bottleneck matching between $\beta_{\even}(M)\sqcup
	\beta_{\odd}(N)$ and $\beta_{\odd}(M)\sqcup \beta_{\even}(N)$ when $n\geq 2$,
	and a vertical ${3\varepsilon}$-bottleneck between $\beta_{\even}(M)\sqcup \beta_{\odd}(N)$ and $\beta_{\odd}(M)\sqcup
	\beta_{\even}(N)$ when $n=1$.
\end{proof}

\section{Proof of~\cref{cor:thm_fixed_radius_implies_tame}}\label{sec:proof_of_tame}

\begin{proof}[Proof of~\cref{cor:thm_fixed_radius_implies_tame}]
		By \cref{thm:estimator_fixed_radius}~(i), we know that for any $ \e\in(0,\rhox)$, the persistence modules $\hsf$ and
		$\hs(\OO^{\e}(\FP))$ are ordinarily $\omega(\e)$-interleaved. Therefore, there exist two morphisms
		\[
			\kappa:\hsf\to \hs(\OeFP)[\omega(\e)\bm{1}] \quad \text{and} \quad \gamma:\hs(\OeFP)\to \hs(\ff)[\omega(\e)\bm{1}]
		\]
		such that, for any $\bmx\in\rn$, the following diagram is commutative.
		\[
			\begin{tikzcd}
				\hsf_{\bmx} \arrow[rr, "\hsf_{\bmx,\bmx+2\omega(\e)\bm{1}}"] \arrow[rd, "\kappa_{\bmx}"'] &
				& \hsf_{\bmx+2\omega(\e)\bm{1}} \\
				& \hs(\OeFP)_{\bmx+\omega(\e)\bm{1}} \arrow[ru, "\gamma_{\bmx+\omega(\e)\bm{1}}"'] &
			\end{tikzcd}
		\]
		It follows that
		\begin{equation}\label{equ:tame_of_hsf}
			\rk\left(\hsf_{\bmx,\bmx+2\omega(\e)\bm{1}}\right)=\rk\left(\gamma_{\bmx+\omega(\e)\bm{1}}\circ
			\kappa_{\bmx}\right)\leq
			\rk\left(\kappa_{\bmx}\right)\leq \dim\left(\hs(\OeFP)_{\bmx+\omega(\e)\bm{1}}\right)<+\infty,
		\end{equation}
		where the final inequality holds by the finiteness of $P$ and the construction of $\OeFP$.

		Since $\omega(\e)\to 0$ as $\e \to 0$, our statement follows from~\cref{equ:tame_of_hsf} by choosing $\e$ sufficiently
		small so that $\bmx+2\omega(\e)\bm{1}\leq\bmy$.
\end{proof}

\section{More details on the statistical aspects.}\label{appendix:stats}

\begin{lemma}\label{lemma:stats:needs_thresholding}
	Let $X$ be a topological space satisfying \cref{assumption:compact_support},
	and let $(\delta_k)_{k\in\NNN}$ be a sequence of positive numbers with
	$\delta_k\to 0$. Let $\mu$ be an $(a,b)$-standard probability measure
	supported on $X$, i.e., satisfying \cref{assumption:ab_std},
	and let $X_k=(Z_1,\cdots,Z_k)$ be a $k$-sample of $\mu$ for some index $k\in\NNN$.
	For any continuous function $\ff\colon X\to\rn$,
	let $\estgenericHi$ be chosen among the following $X_k$-measurable
	estimators:
	\begin{equation}\label{eq:stats:need_threshold:estimators}
		\estoffsetdHi, \quad \estcechdHi,\quad \textnormal{or} \quad \estripsdHi.
	\end{equation}
	\begin{enumerate}[label=(\roman*)]
		\item If $H_i(X)\neq 0$ for some integer $i>0$, then for any large enough $k\in\NNN$, we have
			\begin{equation*}
				\Ebb_{X_k\sim \mu^{\otimes k}} \left( \dIo \left( \estgenericHi, \oracleHi \right) \right) = +\infty.
			\end{equation*}
		\item If $X$ is path-connected and $X$ contains at least two distinct points, then for all sufficiently large $k$, we
			have
			\begin{equation*}
				\Ebb_{X_k\sim \mu^{\otimes k}} \left( \dIo \left( \estgenericHzero, \oracleHzero \right) \right) = +\infty.
			\end{equation*}
	\end{enumerate}
\end{lemma}
\begin{proof}
	Since $\ff$ is continuous and $X$ is compact, there exists  $\bar t\in
	\RRR$ such that $ \left\lVert \ff  \right\rVert_\infty\le
	\bar t$. Let $\bar\bmt=(\bar t,\cdots,\bar t)\in\rn$.
	In particular, for any degree $i\in\mathbb{N}$ we have $\hi(\ff)_{\bar\bmt}  =\Homology_i (X)$.
	Fix $x\in X$, and consider $k$ large enough so that $a (\delta_k)^b \le 1$.
	We then have:
	\begin{equation*}
		\Pbb{} \left( X_k\subseteq B_X(x,\delta_k) \right) = \left( \mu \left( B_X(x,\delta_k) \right) \right)^k \ge \left( a
		(\delta_k)^b\right)^k > 0.
	\end{equation*}

	For the first case when $\hi(X)\neq0$ for some integer $i>0$.
	Assume $k$ is large enough so that $\delta_k < \rhox$.
	On the event $ \left\{ X_k\subseteq B_X(x,\delta_k) \right\}$,
	we have $x\in\bigcap\limits_{j\in\{1,\dots,k\}}B_X(Z_j,\delta_k)\neq\emptyset$ and $d_X(Z_j,Z_{j'})<2\delta_k$ for
	any $j,j'\in\{1,\dots,k\}$.
	Thus, we have $\estgenericHi = 0$
	for each estimator $\estgenericHi$ given in
	\cref{eq:stats:need_threshold:estimators}.
	Now, for any $\boldsymbol \bnu = (\nu,\dots,\nu)$ with $\nu>0$, we have
	$\oracleHi_{\bar\bmt} \simeq \oracleHi_{\bar\bmt+2\bnu} \simeq \im (\oracleHi_{\bar\bmt} \to
	\oracleHi_{\bar\bmt+2\bnu}) \simeq \Homology_i(X) \neq
	0$, while $\estgenericHi = 0$. Hence no $\bnu$-interleaving exists
	between these two modules, and on this event we have $\dIo \left( \estgenericHi,
	\oracleHi \right) = +\infty$.
	We conclude with the inequality:
	\begin{equation*}
		\Ebb_{X_k\sim \mu^{\otimes k}} \left[ \dIo \left( \estgenericHi, \oracleHi \right) \right]
		\ge
		\Pbb{} \left( X_k\subseteq B_X(x,\delta_k) \right)  \Ebb \left[ \dIo \left( \estgenericHi,
		\oracleHi \right) \mid  X_k\subseteq B_X(x,\delta_k)\right]
		= +\infty.
	\end{equation*}

	The second statement in the lemma follows from a similar argument.
	Fix $x \neq y\in X$, and consider $k$ large enough so that $\delta_k < \frac 1 4 d_X(x,y)$.
	Now consider the event:
	\begin{equation*}
		E_k  = \left\{ (Z_1,\dots,Z_{\lfloor k/2\rfloor}) \subseteq B_X\left(x, {\delta_k/2} \right) \right\}
		\cap
		\left\{ (Z_{\lfloor k/2\rfloor + 1},\dots,Z_{k}) \subseteq B_X\left(y, {\delta_k/2} \right) \right\}.
	\end{equation*}
	By the $(a,b)$-standardness assumption, the event $E_k$ has a positive measure.
	Furthermore, for any $\boldsymbol \nu= (\nu,\dots,\nu)$ with $\nu>0$, the chosen ball radius {$\delta_k/2$} ensures
	that, on this event,
	the homology groups $(\estgenericHzero)_{\bar\bmt}\simeq
	(\estgenericHzero)_{\bar\bmt+2\bnu}$ are two-dimensional, while
	$\hzero(X) \simeq \oracleHzero_{\bar\bmt}\simeq \oracleHzero_{\bar\bmt+\bnu}$
	is one-dimensional. Hence, there is no
	$\boldsymbol\nu$-interleaving between the persistence modules $\estgenericHzero$
	and~$\oracleHzero$. The rest of the argument is the same as for the
	first statement.
\end{proof}

\begin{lemma}\label{lemma:b_ge_1}
	Let $X$ be a compact geodesic space with convexity radius $\rhox>0$
	(\cref{assumption:compact_support}) and let $\mu$ be an $(a,b)$-standard probability measure with support $X$
	(\cref{assumption:ab_std}). Then $b \ge 1$.
\end{lemma}
\begin{proof}
	Note that this result trivially holds if $X$ is a point.
	Fix two points $x\neq y \in X$ and
	let
        $\gamma$
        be a shortest path from $x$ to $y$.
	Under \cref{assumption:compact_support}, fix $k\in \mathbb{N}_{>0}$ and consider
	the sequence of points $ x^k_0,\dots,x^k_k$ along $\gamma$
	defined by:
	\begin{enumerate}
		\item $x_0^k = x$ and $x_k^k=y$, and
		\item for each $i\in \{0,\dots,k-1\}$, $d_X(x_i^k, x_{i+1}^k) = \frac{d_X(x,y)}{k}$.
	\end{enumerate}
	Note that,
        by the triangle inequality,
        the open balls $\Bigl\{ B_X\!\left(x_i^k, \tfrac{d_X(x,y)}{2k}\right)
	\Bigr\}_{i=0}^k$ are pairwise disjoint.
	Consequently, we have:
	\begin{equation*}
		\sum_{i=0}^k\mu\left[B_X\left(x_i^k, \frac {d_X(x,y)}{2k}\right)\right]
		= \mu \left[ \bigcup_{i=0}^k B_X \left( x_i^k,\frac {d_X(x,y)} {2k}  \right)\right]
		\le \mu(X)=1.
	\end{equation*}
	Suppose, for the sake of contradiction, that $b < 1$. Then, we have:
	\begin{equation*}
		1\ge
		\sum_{i=0}^k\mu\left[B_X\left(x_i^k, \frac {d_X(x,y)} {2k}\right)\right]
		\ge
		(k+1) \cdot \min\Bigl\{ a  \left( \frac {d_X(x,y)} {2k} \right)^b, 1\Bigr\}   \xrightarrow[k\to \infty]{}\infty.
	\end{equation*}
	Hence $b\geq 1$.
\end{proof}

\begin{lemma}\label{lemma:ours_vs_bertrand_assumption}
	If $\omega\colon \mathbb{R}_{\ge 0}\to \mathbb{R}_{\ge 0}$ is such that
	$\delta\mapsto\frac{\omega(\delta)}{\delta}$ is non-increasing, then either
	\cref{assumption:mod_cont_reg} holds or $\omega = 0$.
\end{lemma}
\begin{proof}
	Assume that $\delta\mapsto \frac{\omega(\delta)}\delta$ is non-increasing
	and that \cref{assumption:mod_cont_reg} is not satisfied.
	Then, for any constant $C>0$, there exists a small enough $\delta>0$ such that
	$\delta> C\omega(\delta)$, or equivalently, $\frac{\omega(\delta)}{\delta}
	\le \frac 1 C$.
	Now, given a $T>0$, since $\delta\mapsto \frac{\omega(\delta)}\delta$ is
	non-increasing, we get $ \left\lVert t\in [\delta,T] \mapsto
	\frac{\omega(t)}t \right\rVert_\infty \le \frac 1 C$, and consequently $
	\left\lVert t\in [\delta, T] \mapsto \omega(t) \right\rVert_\infty \le \frac
	T C$.
	Letting $C$ go to infinity (hence $\delta\to 0$),
	and then $T\to\infty$, we conclude that
	$\omega = 0$ on $\mathbb{R}_{\ge 0}$.
\end{proof}

\begin{lemma}\label{lemma:no_modcont_assumption}
	Assume that $\omega$ does not satisfy
	\cref{assumption:mod_cont_reg}, and that $X$ satisfies \cref{assumption:compact_support}.
	Then,
	$ \left\{ \ff \colon X \to \mathbb{R}^n \mid \ff  \textnormal{ is $\omega$-continuous}\right\} $
	is the set of constant functions on $X$.
\end{lemma}

\begin{proof}
	Note that the result trivially holds when $X$ a point.
	Otherwise, let $x\neq y\in X$ and let
        $\gamma$
	be a shortest path from $x$ to $y$.
	Fix $\varepsilon>0$, and define $C_\varepsilon:= \frac \varepsilon {d_X(x,y)}$.
	Since~\cref{assumption:mod_cont_reg} is not satisfied, there exists a small enough $
	\eta>0$ such that
	${
		\omega(\eta)
		\le
	\eta C_\varepsilon}$.

	Then, consider the sequence $x_1^\eta, \dots, x_m^\eta$ of points along~$\gamma$ 
	such that:
	\begin{enumerate}
		\item $x_1^\eta = x$ and $x_m^\eta = y$,
		\item $d_X(x_i^\eta, x_{i+1}^\eta) = \eta$ for each $i \in \{1,\dots,m-2\}$, and
			$d_X(x_{m-1}^\eta, x_m^\eta) \leq \eta$, and
		\item $d_X(x,y) = \sum_{i=1}^{m-1} d_X(x_i^\eta, x_{i+1}^\eta)$.
	\end{enumerate}

	Then, for any $\omega$-continuous function $\ff$, we have:
	\begin{align*}
		\Vert \ff(x)-\ff(y)\Vert
		\le & \sum_{i=1}^{m-1} \Vert \ff(x_i^\eta)-\ff(x_{i+1}^\eta)\Vert_\infty                               \\
		\le & \sum_{i=1}^{m-2} \omega(\eta) + \left\lVert \ff(x_{m-1}^\eta) -\ff(x_m^\eta) \right\rVert_\infty \\
		\le & \frac {\varepsilon}{d_X(x,y)}\sum_{i=1}^{m-2} d_X(x_i^\eta,x_{i+1}^\eta)
		+ \omega(\eta)                                                                                         \\
		\le &
		\varepsilon
		+
		\omega(\eta).
	\end{align*}
	This concludes that  $\ff(x)=\ff(y)$ by letting $\eta \to
	0$ and then $\varepsilon \to 0$.
\end{proof}

\begin{lemma}\label{lemma:stats:packing_bootstrap}
	Let $X$ be a compact $d$-manifold and $\mu$
	an $(a,b)$-standard probability measure with support $X$.
	Suppose $b=d$ (this is typically true under
	\cref{assumption:manifold_support}).
	Let $X_k$ be a $k$-sample of $\mu$.
	Then, for $k$ sufficiently large, with $\hat{\delta}_k$ defined
	as in~\Cref{eq:stats:hatdeltak},
	we have:
	\begin{equation}
		\mathbb{P}\left( {\dH}(X_k,X)> \frac{\hat
		\delta_k}{2}\right) \le \frac {2^b}{2k\log(k)}.
	\end{equation}
\end{lemma}
\begin{proof}
		We follow the proof ideas of \cite[Proposition 13]{carriere2018statistical}.
		Fix $t_k := 2 \left( \frac{2\log(k)}{ak} \right)^\frac 1 b$, and consider the following events:
		\begin{equation}\label{eq:stats:def_ak_bk}
			A_k : =
			\left\{ {\dH}(X_k,X) > \frac{\hat\delta_k} 2 \right\}
			=
			\left\{ 2\,{\dH}(X_k,X) > {\dH}(X_{s_k},X_k) \right\}
			\quad \textnormal{ and }\quad
			B_k : =
			\left\{ {\dH}(X_k,X) > t_k \right\}.
		\end{equation}
		By \cref{lemma:cuevas}, we have:
		\begin{equation}\label{eq:ak_to_0}
			\begin{array}{rcl}
				\Pbb(A_k) & =     & \mathbb{P}(A_k\cap B_k) +
				\Pbb(A_k \cap \statcompl{B_k}) \le \Pbb(B_k) +\Pbb(A_k\cap \statcompl{B_k})
				\\[1ex]
				& {\le} &
				{\frac{2^b}{a(\frac{2\log(k) }{ak})}\,e^{-ak(\frac{2\log(k)}{ak})} + \Pbb(A_k\cap \statcompl{B_k})}
				= \frac{2^b}{2k\log(k) } + \Pbb(A_k\cap \statcompl{B_k}).
			\end{array}
		\end{equation}

		The proof thus reduces to showing that $A_k \cap \statcompl{B_k}$ has probability 0 for sufficiently large~$k$. We
		establish this by proving that the event becomes empty once $k$ is large enough. Assume, for the sake of
		contradiction, that the event $A_k\cap \statcompl{B_k}$ is non-empty.
		On this event, we have:
		\begin{equation}\label{eq:dH_Xsk_X}
			\dH(X_{s_k},X)
			\le
			\dH(X_{s_k},X_k) + \dH(X_k,X)
			\overset{(A_k)}\lesssim
			\dH(X_k,X)
			\overset{(B_k^c)}\le
			t_k,
		\end{equation}
		which implies that $X_{s_k}$ is a $t_k$-covering set of $X$ up to a multiplicative constant.

		We now show, using a packing argument, that this is impossible for sufficiently large $k$.
		For any scale $t>0$, let $C_t$ be an optimal $t$-covering set, and $P_t$ be an optimal $t$-packing set,
		or more formally:
		\begin{equation}
			P_t \in \argmax_{Q \subseteq X} \left\{ |Q|\in \mathbb{N} \mid  \left\{B_X(x,t)\right\}_{x\in Q} \textnormal{ are
				pairwise disjoint}
			\right\},
		\end{equation}
		and
		\begin{equation}
			C_t \in \argmin_{Q\subseteq X}\left\{ |Q|\in \mathbb{N} \mid  X\subseteq \bigcup_{x\in Q} B_X(x,t)
			\right\}.
		\end{equation}
		By \cite[Lemma 17]{fasyConfidenceSetsPersistence2014} and
		\cite[Lemma~5.2]{niyogiFindingHomologySubmanifolds2008}, we have:
		\begin{equation}\label{eq:packing_covering_rates}
			t
			= \Theta\left(|P_t|^{-\frac 1 d}\right)
			= \Theta\left(|C_t|^{-\frac 1 d}\right).
		\end{equation}
		Consider now the scale
		$
		t^*:= \inf\left\{ t>0 \mid |C_t|\le s_k \right\} = \inf \left\{
		t>0 \mid \exists S\subseteq X,\, |S|\le s_k,\, \dH(X,S)\le t \right\}.
		$
		In particular, as $X_{s_k}$ is a $\dH(X_{s_k},X)$-covering of $X$ with
		$|X_{s_k}| = s_k \le s_k$, we have:
		\begin{equation}\label{eq:tstar_lowerbound}
			t^* \le \dH(X_{s_k},X)\lesssim t_k,
		\end{equation}
		where the first inequality follows from the definition of $t^*$ and the second from~\cref{eq:dH_Xsk_X}.
		Now, as $t\mapsto |C_t|$ is non-decreasing, the definition of $t^*$ implies that, for any $\e>0$,
		$|C_{t^*+\e}|\le s_k$.
		By~\cref{eq:packing_covering_rates}, we have
			$t^*+\e = \Theta(|C_{t^*+\e}|^{-\frac{1}{d}})$, and therefore
			$t^* +\e \gtrsim s_k^{-\frac 1 d}$. Since the multiplicative constant in \cref{eq:packing_covering_rates} is
			independent of $k$, taking the infimum over $\varepsilon>0$ and applying \cref{eq:tstar_lowerbound} yields:
			\begin{equation}\label{equ:tstar_sk}
				s_k^{-1/d} \lesssim t^* \leq \dH(X_{s_k}, X) \leq t_k.
			\end{equation}
		However, we have $t_k = \mathrm o \left( s_k^{-\frac 1 d} \right)$, since:
		\begin{equation}
			s_k = \left\lceil \frac k {(\log k)^{1+\beta}}\right\rceil \le  \frac {2k} {(\log k)^{1+\beta}} \lesssim \frac
			{t_k^{-d}}{(\log k)^\beta }.
		\end{equation}
		This yields the desired contradiction for sufficiently large $k$ and sufficiently small $\e$.

\end{proof}

\begin{lemma}[Quasi-minimax bound on $\hat\delta_k$]\label{lemma:minimax_distance_sample}
	Under the same assumptions and notations as \cref{lemma:stats:packing_bootstrap},
	for any bounded and non-decreasing measurable function
	$\omega\colon \mathbb{R}_{\ge 0}\to \mathbb{R}_{\ge 0}$
	satisfying $\delta \underset{\delta\to 0 } = O(\omega(\delta))$, we have:
	\begin{equation}
		\Ebb \left[ \omega(\hat\delta_k) \right]
		\lesssim \omega \left[  2\left( \frac{2(\log k)^{2+\beta}}{ak} \right)^{\frac 1 b}\right]
		\quad
		\textnormal{ and }
		\quad
		\sup_{\mu \in \abstd X}\mathbb E \left( \omega(\hat\delta_k) \right)
		\gtrsim \omega \left[ \left(\frac{C}{k}\right)^{\frac 1 b} \right],
	\end{equation}
	for some constant $C>0$ depending only on $X$.
\end{lemma}
\begin{proof}
	\textbf{Proof of the upper bound.}
	Define $u_k := 2\left( 2\frac{(\log k)^{2+\beta}}{ak} \right)^{\frac 1 b}$. We start from the following upper bound:
	\begin{equation}\label{eq:stats:fdeltakupperbound}
		\Ebb \left[ \omega(\hat\delta_k) \right]
		=
		\Ebb \left[ \ind{\hat\delta_k \le u_k}\omega(\hat\delta_k)  \right]
		+
		\Ebb \left[ \ind{\hat\delta_k > u_k}\omega(\hat\delta_k) \right]
		\le
		\omega(u) +
		\left\lVert \omega  \right\rVert_\infty\cdot\Pbb \left( \hat\delta_k > u_k \right).
	\end{equation}
	Now, since ${\dH(X_{s_k},X_k)} \le \dH(X_{s_k},X)$, we have the following upper bound:
	\begin{equation}
		\Pbb \left( \hat\delta_k > u_k  \right)  =
		\Pbb \left( \dH(X_{s_k},X_k) > u_k  \right)
		\le
		\Pbb \left( \dH(X_{s_k},X) >  u_k \right).
	\end{equation}
	\cref{lemma:cuevas} guarantees that:
	\begin{equation}\label{eq:stats:hatdeltakprobupperbound}
		\Pbb \left( \hat\delta_k >  u_k\right)
		\le
		\frac {2^b}{a \left( \frac {2(\log k)^{2+\beta}}{ak} \right)}
		e^{-a\left\lceil \frac {k}{(\log k)^{1+\beta}}\right\rceil  \left( \frac{2(\log k)^{2+\beta}}{ak} \right)}
		\le  \frac{2^b}{k(\log k)^{2+\beta}} = o(u_k) = o(\omega(u_k)),
	\end{equation}
	and hence, up to some constant, we conclude combining
	\Cref{eq:stats:hatdeltakprobupperbound,eq:stats:fdeltakupperbound}:
	\begin{equation}
		\Ebb \left[ \omega(\hat\delta_k) \right]
		\lesssim
		\omega(u_k) + o(\omega(u_k))
		\lesssim
		\omega(u_k) = \omega \left[  2\left( \frac{2(\log k)^{2+\beta}}{ak} \right)^{\frac 1 b}\right].
	\end{equation}
		\textbf{Proof of the lower bound.} Pick $A_k$ from~\cref{eq:stats:def_ak_bk}.
		We have:
		\begin{equation}\label{eq:stats:ak_bound}
			\Ebb \left( \omega(\hat\delta_k) \right)
			=
			\Ebb \left[ \ind{A_k} \omega(\hat\delta_k)\right]
			+
			\Ebb \left[ \ind{\statcompl{A_k}}\omega(\hat\delta_k) \right]
			\ge
			\Pbb \left( \statcompl{A_k} \right)\Ebb \left[  \omega(\hat\delta_k)\mid \statcompl{A_k} \right]
			\gtrsim \Ebb \left[ \omega(2\dH \left( X_k,X \right)) \right],
		\end{equation}

		where the last inequality follows from the fact that $\omega$ is non-decreasing.

		Under the assumption that $X$ is a compact $d$-dimensional smooth manifold,
		the lower bound can be shown using a packing argument, similar to the one of
		\cref{lemma:stats:packing_bootstrap}.
		Consider the scale $t^*:= \inf\left\{ t>0 \mid |C_t|\le k \right\} = \inf \left\{
		t>0 \mid \exists S\subseteq X,\, |S|\le k,\, \dH(X,S)\le t \right\} $,
		where $C_t$ denotes a $t$-optimal covering set at scale $t$.
		Using the same argument as for~\cref{eq:tstar_lowerbound,equ:tstar_sk}, we have:
			\begin{equation}\label{eq:stats:packing_sk_lowerbound}
				\dH(X_{k},X)
				\geq t^*
				\gtrsim k^{-\frac{1}{d}}.
			\end{equation}
		Now, since $\omega$ is non-decreasing,
		there exists a constant $C$ such
		that:
		\begin{equation}
			\omega(2\dH \left( X_{k},X \right))
			\ge \omega \left[ \left(\frac{C}{k}\right)^{\frac 1 b} \right],
		\end{equation}
		which concludes the proof when combined with \Cref{eq:stats:ak_bound}.
\end{proof}

\begin{lemma}[{Le Cam, version from \cite[Lemma 20]{carriere2018statistical}}]\label{lecams_lemma}
	Let $\mathcal{P}$ be a set of probability distributions, and consider a map
	$\theta \colon \mathcal{P } \to (M,\rho)$, where $(M,\rho)$ is a pseudo metric space.
	If $X_k\sim P^{\otimes k}$ is $k$-sample of some distribution $P\in \mathcal{P}$,
	then for any two pair of distributions $P_0,P_1\in \mathcal{P}$,
	and $X_k$-measurable estimator $\hat \theta$ of $\theta(P)$, we have:
	\begin{equation*}
		\sup_{P\in \mathcal{P}} \Ebb _{X_k\sim P^{\otimes k}} \left( \rho(\theta(P),
		\hat\theta) \right)
		\ge \frac 1 8
		\rho \left( \theta(P_0),\theta(P_1) \right) \left[ 1 - \mathrm{TV}(P_0,P_1) \right]^{2k},
	\end{equation*}
	where $\mathrm{TV}(P_0,P_1)$ is the total variation between $P_0$ and $P_1$,
	defined as:
	\begin{equation*}
		\mathrm{TV}(P_0,P_1) := \sup_{A \textnormal{ measurable}} |P_0(A)-P_1(A)|.
	\end{equation*}
\end{lemma}

%% file: Sections/main_proof_contractible.tex
\section{From Convexity to Contractibility}\label{sec:contractible_proof}
This section is devoted to proving the following result:
\begin{theorem}\label{thm:convex_to_contractible}
Let $(X,d_X)$ be a compact geodesic metric space. Suppose $A \subset X$ is geodesically convex, in the sense that for every pair of points $x,y \in A$, the shortest path in $X$ that connects $x$ to $y$ is unique and included in $A$. Then $A$ is contractible.
\end{theorem}

The proof relies on the following version of the Arzel\`a--Ascoli Compactness Theorem---see~\cite[Theorem~2.5.14]{burago2001course}.
\begin{theorem}[Arzel\`a--Ascoli Compactness Theorem]\label{thm:arzela-ascoli}
  In a compact metric space~$X$, any sequence of continuous paths $[0,1]\to X$ that are parametrized with constant speed and that have uniformly bounded lengths contains a subsequence that is
    convergent in the topology induced by the supremum distance
\(\displaystyle
d_\infty(\gamma,\gamma'):=\sup_{t\in[0,1]} d_X(\gamma(t),\gamma'(t)).
\)
\end{theorem}

\begin{proof}[Proof of~\cref{thm:convex_to_contractible}]
Fix a point $a_0\in A$. It suffices to construct a homotopy
\[
H:A\times[0,1]\to A
\]
such that $H(x,0)=x$ and $H(x,1)=a_0$ for all $x\in A$.

\smallskip
For every $x\in A$, let $\gamma_x:[0,1]\to X$ denote the unique shortest path from $x$ to $a_0$, parametrized with constant speed.  
Define
\[
H(x,t):=\gamma_x(t),\qquad (x,t)\in A\times[0,1].
\]
The identities $H(x,0)=x$ and $H(x,1)=a_0$ hold by construction.  
It remains to prove that $H$ is continuous.

Let $C([0,1],X)$ be the space of continuous maps $[0,1]\to X$ equipped with the supremum metric.
Take a sequence $(x_n)_{n\in\mathbb{N}}$ in $A$ with $x_n\to x \in A$, and set $\gamma_n:=\gamma_{x_n}$.  
%
Since $X$ is compact, each $\gamma_n$ satisfies $L(\gamma_n)\leq \diam(X)$, so,
by~\cref{thm:arzela-ascoli}, there exists a subsequence $(\gamma_{n_k})_{k\in\NNN}$ that converges
to some $\gamma_\infty\in C([0,1],X)$.
Since the endpoints converge: $\gamma_{n_k}(0)=x_{n_k}\to x$ and $\gamma_{n_k}(1)=a_0$, we have: 
\[
\gamma_\infty(0)=x,\qquad \gamma_\infty(1)=a_0.
\]
By lower-semicontinuity of the length structure induced by a metric~\cite[Prop.~2.3.4]{burago2001course}, we deduce: 
\[
L(\gamma_\infty)\le\liminf_{k\to\infty}L(\gamma_{n_k})
=\liminf_{k\to\infty} d_X(x_{n_k},a_0)
=d_X(x,a_0).
\]
Since $d_X(x,a_0)\le L(\gamma)$ for any path~$\gamma$ joining $x$ to $a_0$, we obtain:
\[
L(\gamma_\infty)=d_X(x,a_0),
\]
so $\gamma_\infty$ is a shortest path from $x$ to $a_0$.  
As $A$ is geodesically convex, such a path is unique, therefore
\[
\gamma_\infty=\gamma_x.
\]

Similarly, every
convergent subsequence of $(\gamma_{x_n})_{n\in\NNN}$ converges
to $\gamma_x$.

\medskip
Now, recall that, if every subsequence of a sequence in a metric space admits a further subsequence converging to the same limit, then the whole sequence itself converges to that limit.  
Thus,
\[
d_\infty(\gamma_{x_n},\gamma_x) \xrightarrow[n\to \infty]{} 0.
\]
In particular, the map
\[
G\colon A\to C([0,1],X),\qquad x\mapsto \gamma_x
\]
is continuous. Then, given any converging sequence $(x_n,t_n)\to(x,t)$ in $A\times[0,1]$, the triangle inequality yields:
\begin{align*}
d_X\bigl(H(x_n,t_n),H(x,t)\bigr)
&=d_X(\gamma_{x_n}(t_n),\gamma_x(t))\\
&\le d_X(\gamma_{x_n}(t_n),\gamma_x(t_n))
+ d_X(\gamma_x(t_n),\gamma_x(t))\\
&\le d_\infty(\gamma_{x_n},\gamma_x)
+ d_X(\gamma_x(t_n),\gamma_x(t)),
\end{align*}
where both terms go to~$0$ as $n\to\infty$ because, on the one hand,  $G$ is continuous, and on the other hand, $\gamma_x$ itself is continuous.  Therefore, $H(x_n,t_n)\to H(x,t)$, proving that $H$ is continuous.

In conclusion, $H$ is a homotopy from $\mathrm{id}_A$ to the constant map at $a_0$, therefore $A$ is contractible.
\end{proof}

%% file: ref.bib
@article{bjerkevik2020computing,
  title={Computing the interleaving distance is NP-hard},
  author={Bjerkevik, H{\aa}vard Bakke and Botnan, Magnus Bakke and Kerber, Michael},
  journal={Foundations of Computational Mathematics},
  volume={20},
  number={5},
  pages={1237--1271},
  year={2020},
  publisher={Springer}
}

@incollection{landi2018rank,
  title={The rank invariant stability via interleavings},
  author={Landi, Claudia},
  booktitle={Research in computational topology},
  pages={1--10},
  year={2018},
  publisher={Springer}
}

@article{oudot2010geodesic,
  title={Geodesic delaunay triangulations in bounded planar domains},
  author={Oudot, Steve Y and Guibas, Leonidas J and Gao, Jie and Wang, Yue},
  journal={ACM Transactions on Algorithms (TALG)},
  volume={6},
  number={4},
  pages={1--47},
  year={2010},
  publisher={ACM New York, NY, USA}
}

@article{erocal2016refined,
  title={Refined algorithms to compute syzygies},
  author={Er{\"o}cal, Bur{\c{c}}in and Motsak, Oleksandr and Schreyer, Frank-Olaf and Steenpa{\ss}, Andreas},
  journal={Journal of Symbolic Computation},
  volume={74},
  pages={308--327},
  year={2016},
  publisher={Elsevier}
}

@article{la1998strategies,
  title={Strategies for computing minimal free resolutions},
  author={La Scala, Roberto and Stillman, Michael},
  journal={Journal of Symbolic Computation},
  volume={26},
  number={4},
  pages={409--431},
  year={1998},
  publisher={Elsevier}
}

@book{cox2005using,
  title={Using algebraic geometry},
  author={Cox, David A and Little, John and O'shea, Donal},
  volume={185},
  year={2005},
  publisher={Springer Science \& Business Media}
}

@phdthesis{schreyer1980berechnung,
  title={Die berechnung von syzygien mit dem verallgemeinerten weierstra{\ss}schen divisionssatz und eine anwendung auf analytische cohen-macaulay stellenalgebren minimaler multiplizit{\"a}t},
  author={Schreyer, Frank-Olaf},
  year={1980}
}

@article{l-vrcms-01, 
author = "Janko Latschev",
title = "Vietoris-{R}ips complexes of metric spaces near a closed {R}iemannian manifold", 
journal = "Archiv der Mathematik", 
volume = 77, 
number = "6", 
pages = "522--528", 
year = 2001
}

@book{hatcher
, author =  "Allen Hatcher"
, title =   "Algebraic Topology"
, publisher =   "Cambridge Univ. Press"
, year =    2001
, url = "http://www.math.cornell.edu/~hatcher/"
, update =  "01.11 orourke"
}

@article{botnan2020decomposition,
  title={Decomposition of persistence modules},
  author={Botnan, Magnus and Crawley-Boevey, William},
  journal={Proceedings of the American Mathematical Society},
  volume={148},
  number={11},
  pages={4581--4596},
  year={2020}
}

@article{bauer2015induced,
  title={Induced matchings and the algebraic stability of persistence barcodes},
  author={Bauer, Ulrich and Lesnick, Michael},
  journal={Journal of Computational Geometry},
  volume={6},
  year={2015}
}

@article{bobrowski2017topological,
  title={Topological consistency via kernel estimation},
  author={Bobrowski, Omer and Mukherjee, Sayah and Taylor, Jonathan E.},
  journal={Bernoulli},
  volume={23},
  number={1},
  pages={288--328},
  year={2017}
}

@inproceedings{chazal2009proximity,
  title={Proximity of persistence modules and their diagrams},
  author={Chazal, Fr{\'e}d{\'e}ric and Cohen-Steiner, David and Glisse, Marc and Guibas, Leonidas J and Oudot, Steve Y},
  booktitle={Proceedings of the twenty-fifth annual symposium on Computational geometry},
  pages={237--246},
  year={2009}
}

@article{bakke2021stability,
  title={On the stability of interval decomposable persistence modules},
  author={Bakke Bjerkevik, H{\aa}vard},
  journal={Discrete \& Computational Geometry},
  volume={66},
  number={1},
  pages={92--121},
  year={2021},
  publisher={Springer}
}

@article{vipondMultiparameterPersistentHomology2021,
  title = {Multiparameter Persistent Homology Landscapes Identify Immune Cell Spatial Patterns in Tumors},
  author = {Vipond, Oliver and Bull, Joshua A. and Macklin, Philip S. and Tillmann, Ulrike and Pugh, Christopher W. and Byrne, Helen M. and Harrington, Heather A.},
  year = {2021},
  month = oct,
  journal = {Proceedings of the National Academy of Sciences},
  volume = {118},
  number = {41},
  pages = {e2102166118},
  issn = {0027-8424, 1091-6490},
  doi = {10/gpd543},
  langid = {english}
}

@article{blumberg2024stability,
  title={Stability of 2-parameter persistent homology},
  author={Blumberg, Andrew J and Lesnick, Michael},
  journal={Foundations of Computational Mathematics},
  volume={24},
  number={2},
  pages={385--427},
  year={2024},
  publisher={Springer},
  note={First made public as arXiv preprint 2010.09628 in 2020}
}

@article{loiseauxStableVectorizationMultiparameter2023,
  title = {Stable {{Vectorization}} of {{Multiparameter Persistent Homology}} Using {{Signed Barcodes}} as {{Measures}}},
  author = {Loiseaux, David and Scoccola, Luis and Carri{\`e}re, Mathieu and Botnan, Magnus Bakke and Oudot, Steve},
  year = {2023},
  month = dec,
  journal = {Advances in Neural Information Processing Systems},
  volume = {36},
  pages = {68316--68342},
  langid = {english}
}

@article{cuevas2004boundary,
  title={On boundary estimation},
  author={Cuevas, Antonio and Rodr{\'\i}guez-Casal, Alberto},
  journal={Advances in Applied Probability},
  volume={36},
  number={2},
  pages={340--354},
  year={2004},
  publisher={Cambridge University Press}
}

@article{lesnick2024sparse,
  title={Sparse Approximation of the Subdivision-{R}ips Bifiltration for Doubling Metrics},
  author={Lesnick, Michael and McCabe, Kenneth},
  journal={arXiv preprint arXiv:2408.16716},
  year={2024}
}

@article{rolle2024stable,
  title={Stable and Consistent Density-Based Clustering via Multiparameter Persistence},
  author={Rolle, Alexander and Scoccola, Luis},
  journal={Journal of Machine Learning Research},
  volume={25},
  number={258},
  pages={1--74},
  year={2024}
}

@article{scoccola2023persistable,
  title={Persistable: persistent and stable clustering},
  author={Scoccola, Luis and Rolle, Alexander},
  journal={Journal of Open Source Software},
  volume={8},
  number={83},
  year={2023}
}

@article{hellmer2024density,
  title={Density Sensitive Bifiltered Dowker Complexes via Total Weight},
  author={Hellmer, Niklas and Spali{\'n}ski, Jan},
  journal={arXiv preprint arXiv:2405.15592},
  year={2024}
}

@article{lesnick2024nerve,
  title={Nerve Models of Subdivision Bifiltrations},
  author={Lesnick, Michael and McCabe, Kenneth},
  journal={arXiv preprint arXiv:2406.07679},
  year={2024}
}

@article{buchet2024sparse,
  title={Sparse higher order {\uppercase{\v{C}}}ech filtrations},
  author={Buchet, Micka{\"e}l and B Dornelas, Bianca and Kerber, Michael},
  journal={Journal of the ACM},
  volume={71},
  number={4},
  pages={1--23},
  year={2024},
  publisher={ACM New York, NY}
}

@inproceedings{alonso2024sparse,
  title={A Sparse Multicover Bifiltration of Linear Size},
  author={Alonso, {\'A}ngel Javier},
  booktitle={41st International Symposium on Computational Geometry (SoCG 2025)},
  pages={6--1},
  year={2025},
  organization={Schloss Dagstuhl--Leibniz-Zentrum f{\"u}r Informatik}
}

@inproceedings{sheehy2012multicover,
  title={A Multicover Nerve for Geometric Inference},
  author={Sheehy, Donald R},
  booktitle={CCCG},
  pages={309--314},
  year={2012}
}

@article{edelsbrunner2021multi,
  title={The multi-cover persistence of {E}uclidean balls},
  author={Edelsbrunner, Herbert and Osang, Georg},
  journal={Discrete \& Computational Geometry},
  volume={65},
  pages={1296--1313},
  year={2021},
  publisher={Springer}
}

@inproceedings{carlsson2007theory,
  title={The theory of multidimensional persistence},
  author={Carlsson, Gunnar and Zomorodian, Afra},
  booktitle={Proceedings of the twenty-third annual symposium on Computational geometry},
  pages={184--193},
  year={2007}
}

@book{chazal2016structure,
  title={The structure and stability of persistence modules},
  author={Chazal, Fr{\'e}d{\'e}ric and De Silva, Vin and Glisse, Marc and Oudot, Steve},
  volume={10},
  year={2016},
  publisher={Springer}
}

@article{chazal2013persistence,
  title={Persistence-based clustering in {R}iemannian manifolds},
  author={Chazal, Fr{\'e}d{\'e}ric and Guibas, Leonidas J and Oudot, Steve Y and Skraba, Primoz},
  journal={Journal of the ACM (JACM)},
  volume={60},
  number={6},
  pages={1--38},
  year={2013},
  publisher={ACM New York, NY, USA}
}

@article{dey2025decomposing,
  title={Decomposing multiparameter persistence modules},
  author={Dey, Tamal K and Jendrysiak, Jan and Kerber, Michael},
  journal={arXiv preprint arXiv:2504.08119},
  year={2025}
}

@article{lesnick-wright,
  title={Interactive visualization of 2-D persistence modules},
  author={Lesnick, Michael and Wright, Matthew},
  journal={arXiv preprint arXiv:1512.00180},
  year={2015}
}

@article{chazal2011scalar,
  title={Scalar field analysis over point cloud data},
  author={Chazal, Fr{\'e}d{\'e}ric and Guibas, Leonidas J and Oudot, Steve Y and Skraba, Primoz},
  journal={Discrete \& Computational Geometry},
  volume={46},
  number={4},
  pages={743--775},
  year={2011},
  publisher={Springer}
}

@inproceedings{chazal2014convergence,
  title={Convergence rates for persistence diagram estimation in topological data analysis},
  author={Chazal, Fr{\'e}d{\'e}ric and Glisse, Marc and Labru{\`e}re, Catherine and Michel, Bertrand},
  booktitle={International Conference on Machine Learning},
  pages={163--171},
  year={2014},
  organization={PMLR}
}

@article{carriere2018statistical,
  title={Statistical analysis and parameter selection for mapper},
  author={Carriere, Mathieu and Michel, Bertrand and Oudot, Steve},
  journal={Journal of Machine Learning Research},
  volume={19},
  number={12},
  pages={1--39},
  year={2018}
}

@inproceedings{co-tpbr-08,
 author =      "F. Chazal and S. Y. Oudot",
 title =       "Towards Persistence-Based Reconstruction in {Euclidean} Spaces",
 booktitle = "Proc. 24th ACM Sympos. Comput. Geom.",
 year =    2008,
 pages = "232--241"
}

@inproceedings{kerber2021fast,
  title={Fast Minimal Presentations of Bi-graded Persistence Modules},
  author={Kerber, Michael and Rolle, Alexander},
  booktitle={2021 Proceedings of the Workshop on Algorithm Engineering and Experiments (ALENEX)},
  pages={207--220},
  year={2021},
  organization={SIAM}
}

@book{burago2001course,
  title={A course in metric geometry},
  author={Burago, Dmitri and Burago, Yuri and Ivanov, Sergei and others},
  volume={33},
  year={2001},
  publisher={American Mathematical Society Providence}
}

@article{loiseaux2022fast,
  title = {Multi-Parameter {{Module Approximation}}: An Efficient and Interpretable Invariant for Multi-Parameter Persistence Modules with Guarantees},
  shorttitle = {Multi-Parameter {{Module Approximation}}},
  author = {Loiseaux, David and Carri{\`e}re, Mathieu and Blumberg, Andrew J.},
  year = 2025,
  month = dec,
  journal = {Journal of Applied and Computational Topology},
  volume = {9},
  number = {4},
  pages = {26},
  issn = {2367-1726, 2367-1734},
  doi = {10.1007/s41468-025-00222-y},
  langid = {english}
}

@article{multipers,
  title = {Multipers: {{Multiparameter Persistence}} for {{Machine Learning}}},
  shorttitle = {Multipers},
  author = {Loiseaux, David and Schreiber, Hannah},
  year = {2024},
  month = nov,
  journal = {Journal of Open Source Software},
  volume = {9},
  number = {103},
  pages = {6773},
  issn = {2475-9066},
  doi = {10.21105/joss.06773},
  langid = {english},
}

@article{lesnick2022computing,
  title={Computing minimal presentations and bigraded {B}etti numbers of 2-parameter persistent homology},
  author={Lesnick, Michael and Wright, Matthew},
  journal={SIAM Journal on Applied Algebra and Geometry},
  volume={6},
  number={2},
  pages={267--298},
  year={2022},
  publisher={SIAM}
}

@inproceedings{alonso2024delaunay,
  title={Delaunay bifiltrations of functions on point clouds},
  author={Alonso, {\'A}ngel Javier and Kerber, Michael and Lam, Tung and Lesnick, Michael},
  booktitle={Proceedings of the 2024 Annual ACM-SIAM Symposium on Discrete Algorithms (SODA)},
  pages={4872--4891},
  year={2024},
  organization={SIAM}
}

@article{bjerkevik2021ell,
  title = {$\ell^{p}$-Distances on Multiparameter Persistence Modules},
  author = {Bjerkevik, H{\aa}vard Bakke and Lesnick, Michael},
  journal = {arXiv preprint arXiv:2106.13589},
  year = {2021}
}

@article{oudot2024stability,
  title={On the stability of multigraded {B}etti numbers and {H}ilbert functions},
  author={Oudot, Steve and Scoccola, Luis},
  journal={SIAM Journal on Applied Algebra and Geometry},
  volume={8},
  number={1},
  pages={54--88},
  year={2024},
  publisher={SIAM}
}

@book{mac2013categories,
  title={Categories for the working mathematician},
  author={Mac Lane, Saunders},
  volume={5},
  year={2013},
  publisher={Springer Science \& Business Media}
}

@book{mortersBrownianMotion2010,
  title = {Brownian Motion},
  author = {M{\"o}rters, Peter and Peres, Yuval},
  year = {2010},
  volume = {30},
  publisher = {Cambridge University Press}
}

@article{fasyConfidenceSetsPersistence2014,
  title = {Confidence Sets for Persistence Diagrams},
  author = {Fasy, Brittany Terese and Lecci, Fabrizio and Rinaldo, Alessandro and Wasserman, Larry and Balakrishnan, Sivaraman and Singh, Aarti},
  year = {2014},
  month = dec,
  journal = {The Annals of Statistics},
  volume = {42},
  number = {6},
  eprint = {1303.7117},
  issn = {0090-5364},
  doi = {10/gfrbfh},
  archiveprefix = {arXiv},
  langid = {english}
}

@article{niyogiFindingHomologySubmanifolds2008,
  title = {Finding the {{Homology}} of {{Submanifolds}} with {{High Confidence}} from~{{Random~Samples}}},
  author = {Niyogi, Partha and Smale, Stephen and Weinberger, Shmuel},
  year = {2008},
  month = mar,
  journal = {Discrete \& Computational Geometry},
  volume = {39},
  number = {1},
  pages = {419--441},
  issn = {1432-0444},
  doi = {10.1007/s00454-008-9053-2},
  langid = {english}
}
